\documentclass{article}

\newcommand{\Addresses}{{
		\bigskip
		\footnotesize
		
		\textsc{Department of Mathematics, Technion - Israel Institute of Technology, Haifa, Israel}\par\nopagebreak
		\textit{E-mail address:} \texttt{ofir.gor@technion.ac.il}

        \medskip

        \textsc{Department of Mathematical Sciences, Durham University, Stockton Road, Durham DH1 3LE}\par\nopagebreak
		\textit{E-mail address:} \texttt{mo-dick.wong@durham.ac.uk}
}}

\title{On the limiting distribution of sums of random multiplicative functions}

\author{Ofir Gorodetsky, Mo Dick Wong}
\date{}

\usepackage[margin=1in]{geometry}
\usepackage{amssymb}
\usepackage{amsfonts}
\usepackage{amsthm}
\usepackage{thmtools}
\usepackage{amsmath}
\usepackage{hyperref}
\usepackage{cleveref}
\usepackage{mathtools}
\usepackage{url}
\usepackage{xcolor}
\usepackage{enumitem}

\theoremstyle{plain}
\newtheorem{thm}{Theorem}[section]
\newtheorem{lem}[thm]{Lemma}  
\newtheorem{proposition}[thm]{Proposition}
\newtheorem{cor}[thm]{Corollary}
\newtheorem{definition}[thm]{Definition}
\theoremstyle{remark}
\newtheorem{rem}{Remark}[section]

\newcommand{\PP}{\mathbb{P}}
\newcommand{\RR}{\mathbb{R}}
\newcommand{\CC}{\mathbb{C}}
\newcommand{\EE}{\mathbb{E}}
\newcommand{\NN}{\mathbb{N}}
\newcommand{\ZZ}{\mathbb{Z}}
\newcommand{\FF}{\mathbb{F}}

\newcommand{\Fa}{\mathcal{F}}
\newcommand{\Ga}{\mathcal{G}}
\newcommand{\Na}{\mathcal{N}}

\newcommand{\Ea}{\mathcal{E}}

\newcommand{\Ia}{\mathcal{I}}

\newcommand{\OurEpsilon}{\varepsilon}
\newcommand{\weight}{\Phi}

\DeclareMathOperator*{\argmin}{arg\,min}

\numberwithin{equation}{section}

\begin{document}

\maketitle

\begin{abstract}
We establish the limiting distribution of $\frac{{(\log \log x)}^{1/4}}{\sqrt{x}} \sum_{n\le x}\alpha(n)$ where $\alpha$ is a Steinhaus random multiplicative function, answering a question of Harper.
The distributional convergence is proved by applying the martingale central limit theorem to a suitably truncated sum. This truncation is inspired by work of Najnudel, Paquette, Simm and Vu on subcritical holomorphic multiplicative chaos setting, but analysed with a different conditioning argument generalised from Harper's work on fractional moments to circumvent integrability issues at criticality.

A significant part of the proof is devoted to the convergence in probability of the associated partial Euler product to a critical multiplicative chaos measure, independent of the mild shift away from the critical line. Our approach to the universality of critical non-Gaussian multiplicative chaos bypasses the barrier analysis with the help of a modified second moment method, and employs a novel argument based on coupling and homogenisation by change of measure, which could be of independent interest. 
\end{abstract}

\section{Introduction}
\subsection{Background}
Let $\alpha\colon \NN \to \CC$ be a Steinhaus random multiplicative function: this is a function such that $(\alpha(p))_{p\text{ prime}}$ are i.i.d.~random variables uniformly distributed on $\{ z \in \CC: |z|=1\}$, and $\alpha(n) := \prod_{p}\alpha(p)^{a_p}$ if $n$ factorises into primes as $n=\prod_{p} p^{a_p}$. In this article, we would like to study the asymptotic distribution of $\sum_{n \le x} \alpha(n)$ and $\sum_{n \le x} \mu(n) \alpha(n)$ as $x \to \infty$, where $\mu$ is the M\"obius function.

This problem is of great importance in analytic number theory because it provides a random model for the cancellation behaviour of fundamental arithmetic quantities. For instance, the Riemann hypothesis can be reformulated as $\sum_{n \le x} \mu(n) \ll x^{\frac{1}{2} + \OurEpsilon}$ for any $\OurEpsilon > 0$, where
\begin{equation}\label{eq:Perron-Mobius}
    \sum_{n \le x} \mu(n) = \frac{1}{2\pi} \int_{\RR} \prod_{p \le x} \left(1 - \frac{1}{p^{\frac{1}{2} + it}}\right) \frac{x^{\frac{1}{2}+it}}{\frac{1}{2} + it}dt
\end{equation}
by Perron's formula. As one expects non-trivial contributions and cancellations to the integral to come from a growing region of $t$ as $x \to \infty$, the rotations $(p^{it})_p$ start to resemble that of i.i.d.~random variables distributed on the complex unit circle. One may thus wonder whether the randomised sum $\sum_{n \le x} \mu(n) \alpha(n)$ would provide new insights for the analysis of \eqref{eq:Perron-Mobius}.
A more common application is to use the Steinhaus random multiplicative function as a model for character and zeta sums, see the work of Harper \cite{harper2023typical} which successfully uses the random model to deduce intriguing and surprising estimates for deterministic counterparts.

At the same time, this problem is motivated by questions from harmonic analysis. In his last paper \cite{Helson}, Helson considered an infinite-dimensional generalisation of a theorem of Nehari on Hankel forms, and he conjectured that $\EE|\sum_{n \le x} \alpha(n)| = o(\sqrt{x})$, which would be sufficient for disproving such generalisation.
This conjecture stirred up much discussion at the time because it was not consistent with the mean-square heuristic in probability (it is easy to check that $\EE[|\sum_{n \le x} \alpha(n)|^2] = \lfloor x\rfloor$ by orthogonality). The better-than-squareroot cancellation for low fractional moments was eventually resolved in the seminal work of Harper \cite{Har2020}, but the asymptotic distribution for $\sum_{n \le x} \alpha(n)$ remains a mystery and is arguably the biggest and most fundamental open problem in the study of random multiplicative functions.

\subsection{Main results}
Before stating the distributional convergence of $\sum_{n \le x} \alpha(n)$, we need to explain some related results concerning the associated random Euler product
\[A_y(s) := \prod_{p \le y} (1-\alpha(p)p^{-s})^{-1}=\prod_{p \le y} \bigg(\sum_{k \ge 0} \frac{\alpha(p^k)  }{p^{ks}}\bigg)=\sum_{\substack{n\ge 1\\p\mid n \implies p \le y}} \frac{\alpha(n)}{n^s}.\]
Here $y \ge 2$ and $s$ is a complex number with $\Re s >0$. 
Our first theorem concerns the asymptotic behaviour of $|A_y|^2$ near the critical line as $y \to \infty$.
\begin{thm}\label{thm:mc-critical}
Let
\[\sigma_y(u) := \frac{1}{2}\left(1 + \frac{u}{\log y}\right)\]
and
\[m_{y, u}(dt) := \sqrt{\log \log y}\frac{|A_y(\sigma_y(u) + it)|^2}{\EE\left[|A_y(\sigma_y(u) + it)|^2\right]}dt, \quad t \in \RR.\]
There exists a non-trivial non-atomic random Radon measure $m_\infty(dt)$ supported on $\RR$ such that the following are true for any fixed $u \ge 0$: for any bounded interval $I$ and any test function $h \in C(I)$, we have
\[   m_{y, u}(h) \xrightarrow[y \to \infty]{p} m_\infty(h).\]
In particular, $m_{y, u} \xrightarrow[y \to \infty]{p} m_\infty$, independent of $u \ge 0$.\footnote{In this article, whenever we say a sequence of measures converge in probability, we are implicitly assuming that the convergence takes place in the space of Radon measures  equipped with the weak$^*$ topology.}
\end{thm}
When $u = 0$, the convergence $m_{y, 0} \to m_\infty$ was established in the weaker distributional sense by Saksman and Webb in \cite{SW}. This is insufficient for the application of martingale central limit theorem (\Cref{lem:mCLT}), and we need to strengthen it to convergence in probability - we will explain it in \Cref{sec:case-u=0}. 

At a heuristic level, $m_\infty$ may be seen as a random measure with formal Lebesgue density given by the infinite Euler product $\prod_{p}\left|1 - \alpha(p) p^{-1/2 - it}\right|^{-2}$ (up to renormalisation), though it should be pointed out that $m_\infty$ is singular with respect to the Lebesgue measure. Moreover, it was conjectured in \cite[Conjecture 1.11]{SW} (and will be established in a forthcoming work of Harper, Saksman and Webb) that $m_\infty$ is the critical chaos measure associated with randomly shifted Riemann zeta function: if $U_T \sim \mathrm{Uniform}([0, T])$, then
\[
    \sqrt{\log \log T} \frac{\left|\zeta(\tfrac{1}{2} + i(x+U_T))\right|^2}{\EE\big[\left|\zeta(\tfrac{1}{2} + i(x+U_T))\right|^2\big]} dx \xrightarrow[T \to \infty]{d} m_\infty(dx),
\]
highlighting the connection between the Steinhaus model and the behaviour of $\zeta$ on the critical line.

The following corollary extends \Cref{thm:mc-critical} to test functions with sufficient decay at infinity.
\begin{cor}\label{cor:mc-critical}
Let $h \in C(\RR)$ be such that $h(t)\ll |t|^{-a}$ for some $a > 1$. Then $m_{y, u}(h) \xrightarrow[y \to \infty]{p} m_\infty(h)$ for any fixed $u \ge 0$.
\end{cor}
We use \Cref{cor:mc-critical} and the martingale central limit theorem to establish the following result.
\begin{thm}\label{thm:summain}
We have
\[   \frac{(\log \log x)^{\frac{1}{4}}}{\sqrt{x}}\sum_{n\le x} \alpha(n) \xrightarrow[x \to \infty]{d} \sqrt{V_\infty} \ G\]
where $\displaystyle V_\infty:= \frac{1}{2\pi}\int_{\RR} \frac{m_\infty(dt)}{|\tfrac{1}{2} + it|^2}$ is almost surely finite and strictly positive, and is independent of $G \sim \Na_\CC(0, 1)$. 
Moreover, the convergence in distribution is stable in the sense of \Cref{def:stable}.
\end{thm}
In his seminal work, Harper proved that \cite{Har2020}
\begin{equation}\label{eq:harpermoms}
\EE\Bigg[  \bigg|\sum_{n\le x}\alpha(n)\bigg|^{2q} \Bigg] \asymp \bigg( \frac{x}{1+(1-q)\sqrt{\log \log x}}\bigg)^q
\end{equation}
holds uniformly for $x\ge 3$ and $0\le q \le 1$. In particular, $|\sum_{n \le x}\alpha(n) (\log \log x)^{\frac{1}{4}}/\sqrt{x}|^{2q}$ is uniformly integrable for any $q\in (0,1)$ and so by a general principle \cite[Lemma 5.11]{Kal2021} one deduces moment convergence from \Cref{thm:summain}.
\begin{cor}
Fix $q\in (0,1)$. In the notation of \Cref{thm:summain},
\[ \lim_{x\to \infty} \EE \bigg[ \bigg| \frac{(\log \log x)^{\frac{1}{4}}}{\sqrt{x}}\sum_{n\le x} \alpha(n)  \bigg|^{2q} \bigg] = \EE [|\sqrt{V_{\infty}} G|^{2q}]=\Gamma(1+q)\EE [(V_{\infty})^q].\]
\end{cor}
All the results above can be extended analogously to $\frac{(\log \log x)^{1/4}}{\sqrt{x}} \sum_{n \le x} \mu(n) \alpha(n)$ with the obvious modifications, i.e.~replacing $A_y(s)$ with $\prod_{p\le y}(1-\alpha(p)p^{-s})$ in the definition of $m_{\infty}$.

\subsection{Previous literature}
Wintner \cite{Wintner} initiated the area of random multiplicative functions in the 1940s, and proved the first results on their partial sums and associated Dirichlet series. Early efforts in the area, including the work of Wintner, focused mostly on almost-sure bounds, but in recent years there have been increasing interest in exploring the moment asymptotics. 

For instance, even moments of $|\sum_{n\le x} \alpha(n)|$ were studied by Heap--Lindqvist \cite{HL}, Harper--Nikeghbali--Radziwi\l\l\ \cite{HNR} and in unpublished work of Granville--Soundararajan using complex analytic methods (see \cite{ACZ} for the fourth moment and \cite{Hofmann} for even moments in the function field setting), thanks to the useful orthogonality relation $\EE[ \alpha(n)\overline{\alpha(m)}]=\delta_{n,m}$.
On the other hand, a tantalising conjecture in the area, due to Helson, was that the absolute first moment is vanishingly small compared to the 
standard deviation:
\[ \EE |\sum_{n\le x}\alpha(n)|=o(\sqrt{x}).\]
After earlier attempts \cite{BS,Weber}, Helson's conjecture was finally resolved by Harper's landmark work \cite{Har2020}, using ideas from Gaussian multiplicative chaos (GMC) and in particular of Berestycki's work \cite{Ber2017}. Harper also established lower and upper bounds on (not necessarily integer) `high' moments \cite{HarperHigh}, which were also not understood until then. See Harper's excellent short review of these works \cite{harper2024moments}.

Harper first introduced the martingale perspective to the investigation of distributional convergence in \cite{Harper2013}, which studied the sum of $\alpha$ restricted to the sparse set of integers with prescribed number of prime factors. Since then many activities were centred around establishing distributional results for $\sum_{n\le x}a_n\alpha(n)$ with deterministic weights $(a_n)_n$, see especially Klurman, Shkredov and Xu \cite{KSX} and Soundararajan and Xu \cite{SoundXu}; the main ingredient in these works is McLeish martingale central limit theorem, which reduces the task of showing a Gaussian limit to a fourth moment computation (an earlier work of Chatterjee and Soundararajan \cite{CS2012} on partial sums on short intervals of length $o(x / \log x)$ also relied on fourth moment estimates, but its approach to distributional convergence was based on Stein's method). This is inadequate for the study of $\sum_{n \le x} \alpha(n)$, however, as non-Gaussian limit distribution was anticipated. 

In his introduction \cite{Har2020} on Helson's conjecture, Harper posed the question of finding the limiting distribution of $\tfrac{1}{\sqrt{x}}\sum_{n\le x}\alpha(n)$ after normalisation by the factor $(\log \log x)^{1/4}$ that appeared in his moment bounds. While his work hinted a connection to the integral of the Euler product $|A_x|^2$ on the critical line, precise conjectures for the limit were not available until our work (see \cite[Conjecture 1.6]{GWL2}), where we established the first non-Gaussian limit distribution for partial sums of random multiplicative functions with so-called $L^2$ twists (later extended to the $L^1$-regime in \cite{GWL1}).
Partial progress towards our conjecture was attained by Hardy \cite{Hardy}, who studied the distributional limit of $\sum_{n\le x,\, P(n)>\sqrt{x}}\alpha(n)$ where $P(n)$ is the largest prime factor of $n$. This modified problem captured some properties of the original sum including the extra $(\log \log x)^{1/4}$ scaling, but benefits from considerable simplification due to restriction to integers with a large prime factor: this structural difference introduces independence that is indispensable for conditional Gaussian approximation at the heart of his proof. See Mastrostefano \cite{Mastrostefano} for another problem that is solved in the $P(n)>\sqrt{x}$ model and is open otherwise, namely almost-sure upper bound on the sum.

Let us also mention the topic of Fourier analysis of holomorphic multiplicative chaos (HMC), which shares various features with the problem considered in this paper but enjoys simplifications due to its Gaussian nature and the lack of number-theoretic complications. This model was first studied in \cite{NPS} by Najnudel, Paquette and Simm who were motivated by the connections to circular $\beta$ ensemble in random matrix theory. Their paper included not only moment bounds for the relevant Fourier coefficients, but also a result on the limiting distribution of these coefficients in the $L^2$-regime (this was later extended to the $L^1$-regime by the same authors and Vu in \cite{NPSV}). Independently,  Soundararajan and Zaman  \cite{SZ} introduced the model as a large-$q$ limit of an equivalent problem in the function field setting (that is, over $\FF_q[T]$), and adapted Harper's techniques to demonstrate similar phenomenon of better-than-squareroot cancellation. Motivated by harmonic analysis of multifractal measures and applications in conformal field theory, Garban and Vargas \cite{GV2023} studied an analogous problem for real GMCs on the unit circle, and they established (among other results) the distribution of the Fourier coefficient in the high frequency limit. Their method, however, relied crucially on Gaussianity and also stopped short of the critical case.\footnote{This point often creates confusion in the community and we provide a brief clarification. Due to the `doubling effect' of the intermittency parameter in the random variance, the parameter $\theta$ used by \cite{NPS, NPSV} in HMC as well as \cite{GWL2, GWL1} in random multiplicative functions should be related to the intermittency parameter $\gamma$ in the real GMC case by $2\theta = (2\gamma)^2$. When we speak of subcritical/$L^1$-regime in this article we are referring to $2\gamma$ which appears in the random variance of \cite{GV2023}; with respect to the original parameter $\gamma$ this is actually the $L^4$-regime of real GMCs.}

Following our previous works \cite{GWL2, GWL1} on random multiplicative functions, we apply the martingale central limit theorem (a generalisation of McLeish to allow for random variance) to deduce the distributional convergence. Our proof rests on two key ingredients:
\begin{itemize}
    \item the universality of the limit of $m_{y,u}$ as in \Cref{thm:mc-critical}, i.e.~the limiting random measure $m_\infty$ is independent of the choice of $u \ge 0$; and
    \item a truncation argument which enables us to throw away integers with atypical factorisation from the original sum, an idea that first appeared in the HMC context in \cite{NPS} and developed further in \cite{NPSV} (see also \cite{SoundXu} for a related idea).
\end{itemize}
The rapid loss of integrability in the critical case, on top of the departure from the Gaussian world, presents unique and formidable challenges from both the number-theoretic and probabilistic perspectives, and novel arguments are needed to achieve the aforementioned steps - we refer the readers to the upcoming subsections for the main ideas of our analysis. We mention in passing that, to the best of our knowledge, \Cref{thm:mc-critical} is the first universality result in the literature of non-Gaussian critical multiplicative chaos, and we expect our techniques to find further applications in the future.

Let us conclude our discussion by mentioning that our method can handle more general sums. For instance, our proof allows us to consider partial sums restricted to $P(n) > n^{\OurEpsilon}$ for any fixed $\OurEpsilon \in (0, 1)$, and in particular recovers the result of Hardy \cite{Hardy}. More importantly, our general \Cref{thm:summainw} covers partial sums of the form $\sum_{x \le n \le x + y} \alpha(n)$ where $y = cx$ for any fixed $c > 0$; the distributional result here - a non-Gaussian limit that depends on $(\alpha(p))_p$ due to stable convergence - complements and is in contrast to a forthcoming work of Harper, Soundararajan and Xu, who observe Gaussianity as soon as $y = o(x)$. We also remark that the critical case in the HMC setup has been pursued independently and resolved in a recent work of Atherfold and Najnudel \cite{AN2025}.

\subsection{Main idea for \texorpdfstring{\Cref{thm:mc-critical}}{Theorem \ref{thm:mc-critical}}}
Let $f(p) \equiv \mathbf{1}$ for our discussion below.\footnote{Our analysis actually applies to more general class of twists $f$ - see the beginning of \Cref{sec:chaos} for details.} It is very natural to study
\[\Ga_{y, u}(t) := 2 \Re \sum_{p \le y} \frac{f(p) \alpha(p)}{p^{\sigma_y(u) + it}}
\qquad \text{and} \qquad 
\nu_{y, u}(dt) := \sqrt{\log \log y}\frac{\exp  \Ga_{y, u}(t)}{\EE\left[ \exp \Ga_{y, u}(t)\right]}dt, \quad t \in \RR,\]
as they capture the first-order Taylor series expansion of $\log|A_y(\sigma_y(u) + it)|^2$, which is responsible for the emergence of multifractality of the random Euler product as $y \to \infty$. Indeed, using a density removal lemma (see \Cref{sec:density-removal}), \Cref{thm:mc-critical} may be deduced from the following statement for $\nu_{y, u}$:
\begin{thm}\label{thm:mc-critical2}
There exists a non-trivial non-atomic random Radon measure $\nu_\infty(dt)$ supported on $\RR$ such that the following are true for any fixed $u \ge 0$: for any compact interval $I$ and any test function $h \in C(I)$, we have $\nu_{y, u}(h) \xrightarrow[y \to \infty]{p} \nu_\infty(h)$. In particular, $\nu_{y, u} \xrightarrow[y \to \infty]{p} \nu_\infty$ independent of $u \ge 0$.
\end{thm}
\begin{rem}
Since $\nu_\infty$ is non-atomic and any $h \in C(I)$ can be estimated from above and below by piecewise constant functions, it suffices to establish \Cref{thm:mc-critical2} by taking $h = \mathbf{1}_I$ for any interval $I \subset \RR$.
\end{rem}
The main contribution of \Cref{thm:mc-critical2} is the universality of critical chaos, i.e.~the limit $\nu_\infty$ is independent of the shift parameter $u \ge 0$. In the context of the theory of GMCs, a common approach to universality would be comparison principle: one interpolates between $\Ga_{y, 0}$ and $\Ga_{y, u}$ and shows that their fractional moments agree in the limit as $y \to \infty$. This method crucially depends on Kahane's convexity inequality (see e.g.~\cite{JS2017}) and unfortunately does not extend beyond Gaussianity.

For a more robust approach, one may consider the second moment method, which amounts to checking $\lim_{y \to \infty} \EE[|\nu_{y, u}(I) - \nu_{y, 0}(I)|^2] = 0$. In this naive form the approach would not work, however, because critical multiplicative chaos does not possess first moment, let alone second moment, in the limit $y \to \infty$. To avoid the integrability issue, one often introduces barrier events and restrict the analysis of $\nu_{y, u}(dt)$ to points $t \in I$ where the growth of $\Ga_{y, u}(t)$ (as a function of $y$) is not atypically large - these are ultimately the points that support the limit measure $\nu_\infty$. Such method was inspired by earlier works in branching Brownian motion and has been widely used for the construction of critical GMCs in the literature, but its adaptation to the current setting presents non-trivial technical challenges because we are working with non-Gaussian fields - it should be emphasised that in studying second moments restricted to barrier events, we are not only interested in their order but also their precise asymptotics which are a lot more delicate. 
The only non-Gaussian example where the barrier approach was successfully implemented in the probability literature was the so-called critical Brownian multiplicative chaos \cite{Jeg2021}: this is a very specific problem concerning the local times of 2-dimensional Brownian motion, and its analysis crucially relies on various special distributional identities and connections to Bessel processes (via Ray-Knight isomorphism theorem) such that techniques from stochastic calculus are possible. Even in the Gaussian world, the barrier analysis is not directly applied to the construction general critical GMCs. Rather, one first performs the analysis within the class of so-called almost star-scale invariant fields, where the crucial white-noise decomposition and connections to 3-dimensional Bessel process allow one to establish convergence of the associated critical chaos, and then extends it to more general logarithmically correlated Gaussian fields by means of decomposition theorem, see for instance \cite{JSW2019} and also the nice review article \cite{Pow2021} by Powell. Indeed, it is worth noting that the convergence $\nu_{y, 0} \to \nu_\infty$ was obtained in \cite{SaksmanWebb, SW} by Gaussian approximation instead of barrier analysis, which allowed Saksman and Webb to reduce the problem to that of a convenient critical GMC.

We propose an alternative solution which we call the modified second moment method (see \Cref{sec:mod-2nd}). This amounts to showing that for any fixed $u \ge 0$ and fixed $L > 0$,
\[   \lim_{y \to \infty} \EE\left[|\nu_{y, 0}(I) - \nu_{y, u}(I)|^2 e^{-L\nu_{y, 0}(I)}\right]= 0. \]
In essence, the damping factor $e^{-L\nu_{y, 0}(I)}$ penalises the bad events on which $\nu_{y, u}(I)$ is exceptionally large to ensure that the weighted second moment does not blow up as $y \to \infty$. In contrast to the barrier method which requires control of $\left(\Ga_{y, u}(t)\right)_{y \ge 3}$ at each point $t \in I$ and the need to study the contributions from good and bad events separately, we shall see very soon that our global weighting factor allows for a very simplified analysis. Expanding the square, one can easily see that it suffices to establish the following claim:
\begin{lem}\label{lem:mod-cross-mom}
For any fixed $L >0$ and $u_1, u_2 \ge 0$, we have
\[    \lim_{y \to \infty} \EE\left[ \nu_{y, u_1}(I) \nu_{y, u_2}(I) e^{-L \nu_{y, 0}(I)}\right] = 
    \EE\left[ \nu_{\infty}(I)^2 e^{-L \nu_{\infty}(I)}\right].\]
\end{lem}
\subsubsection{A first attempt at change of measure}
At first glance it might be surprising how one could evaluate the limit of the complicated expectation in \Cref{lem:mod-cross-mom}, but it turns out that this could be achieved with clever uses of change of measure. Indeed, this philosophy was adopted for the first time in \cite{GWL1} when we established the universality of non-Gaussian multiplicative chaos in the subcritical (i.e.~$L^1$) regime. The analysis in the subcritical case, however, relies on various arguments that either exploit $L^{1+}$-integrability or are too crude for the critical regime. It is instructive to highlight a few issues with the old approach to motivate our novel argument later.

For any vector $(v(p))_p$ indexed by primes, write $\nu_{y, u}(dt; v) := \sqrt{\log \log y} M_y(u)^{-1} \exp\left(\Ga_{y, u}(t; v)\right)dt$ where
\[     \Ga_{y, u}(t; v) = 2 \Re \sum_{p \le y} \frac{f(p) v(p)}{p^{\sigma_y(u) + it}},
    \qquad \text{and} \qquad M_{y}(u) := \EE[\exp(\Ga_{y, u}(t; \alpha))]. \]
Following the approach in \cite{GWL1}, let us write $\mathbf{u} = (u_1, u_2)$, $\mathbf{t} = (t_1, t_2)$, and consider
\begin{align*}
& \EE\left[ \nu_{y, u_1}(I) \nu_{y, u_2}(I) e^{-L \nu_{y, 0}(I)}\right]\\
& \quad =  \left(\log \log y\right)\int_{I \times I} \frac{\EE\left[\exp\left(\sum_{j=1}^2 \Ga_{y, u}(t_j;\alpha) \right)\right]}{M_y(u_1)M_y(u_2)}\EE\left[\frac{\exp\left(\sum_{j=1}^2 \Ga_{y, u}(t_j;\alpha) \right)}{\EE\left[\exp\left(\sum_{j=1}^2 \Ga_{y, u}(t_j;\alpha) \right)\right]}
\exp\left\{-L \nu_{y, 0}(I; \alpha)\right\}\right] dt_1 dt_2\\
& \quad =:  \left(\log \log y\right)\int_{I \times I} \widehat{K}_y(\mathbf{u}, \mathbf{t}) \widehat{\EE}_y^{\mathbf{u}, \mathbf{t}}\left[\exp\left\{-L \nu_{y, 0}(I; \alpha)\right\}\right] dt_1 dt_2.
\end{align*}
If one could show that $\widehat{K}_y(\mathbf{u}, \mathbf{t}) \approx \widehat{K}_y(\mathbf{0}, \mathbf{t})$ and $\widehat{\EE}_y^{\mathbf{u}, \mathbf{t}}\left[\exp\left\{-L \nu_{y, 0}(I; \alpha)\right\}\right] \approx \widehat{\EE}_y^{\mathbf{0}, \mathbf{t}}\left[\exp\left\{-L \nu_{y, 0}(I; \alpha)\right\}\right]$ in some suitable sense as $y \to \infty$, then we would be able to relate the overall object back to $\EE\left[ \nu_{y, 0}(I) \nu_{y, 0}(I) e^{-L \nu_{y, 0}(I)}\right]$ which does converge to the desired limit by dominated convergence (as $x \mapsto x^2 e^{-Lx}$ is bounded).

To achieve the latter approximation in the $L^1$-regime, we used an approximate Girsanov's theorem (\cite[Theorem 2.21]{GWL1}). At a high level, this allowed us to write, for each $\mathbf{t} \in I^2$, that
\[    \widehat{\EE}_y^{\mathbf{u}, \mathbf{t}}\left[\exp\left\{-L \nu_{y, 0}(I; \alpha)\right\}\right] \approx \widehat{\EE}_y^{\mathbf{0}, \mathbf{t}}\left[\exp\left\{-L \nu_{y, 0}(I; \alpha + \widehat{\Delta})\right\}\right] \]
where $\widehat{\Delta}$ is some residual fluctuation. Since $\nu_{y, 0}(I; \alpha + \widehat{\Delta}) = \int_I \exp\left\{\Ga_{y, 0}(t_0; \widehat{\Delta})\right\}\nu_{y, 0}(dt_0; \alpha)$, if one could show (under $\widehat{\EE}_y^{\mathbf{0}, \mathbf{t}}$) that $\sup_{t_0 \in I} |\Ga_{y, 0}(t_0; \widehat{\Delta})|$ is negligible, then it might be possible to carry out the strategy above. Unfortunately, $\widehat{\Delta}$ is not independent of $\alpha$, and in \cite{GWL1} we were forced to apply a rough estimate for $\sup_{t_0 \in I} |\Ga_{y, 0}(t_0; \widehat{\Delta})|$ that is too crude for the critical case even before considering the extra $\log \log y$ factor in front of the integral. One might naively propose directly tackling $\widehat{\EE}_y^{\mathbf{u}, \mathbf{t}}\left[\exp\left\{-L \nu_{y, 0}(I; \alpha)\right\}\right]$ instead, but this is extremely tricky because $\nu_{y, 0}(I; \alpha)$ under $\widehat{\EE}_y^{\mathbf{u}, \mathbf{t}}$ diverges as $y \to \infty$. A careful analysis of the leading order of the expectation would involve `fusion estimates' for multiplicative chaos which are highly technical - even in the Gaussian case it is intractable unless one works with very special logarithmically correlated fields (recall that we require not only the order but the exact asymptotics), see for example \cite{BW2018, Won2019}.

\subsubsection{Homogenisation by change of measure}
In the critical case, we pursue a different approach based on the following identity: for any $u \ge 0$ we have
\begin{equation}\label{eq:shift-critical}
    \nu_{y, u}(dt; \alpha) = \frac{M_y(0)}{M_y(u)}\exp\left\{\Ga_{y, 0}\left(t; \OurEpsilon_{y, u}\alpha\right)\right\} \nu_{y, 0}(dt; \alpha)
    \qquad \text{with} \qquad \OurEpsilon_{y, u}(p) := p^{-\frac{u}{2 \log y}} - 1,
\end{equation}
where $\OurEpsilon_{y, u} \alpha$ is the shorthand for the vector $\left(\OurEpsilon_{y, u}(p) \alpha(p)\right)_{p}$. We would like to pretend that the extra density in front of $\nu_{y, 0}(dt; \alpha)$ on the right-hand side of \eqref{eq:shift-critical} is negligible as $y \to \infty$. A quick inspection, however, shows that uniform estimates are impossible because these are oscillatory terms - that they become irrelevant in the limit is a consequence of homogenisation. Since $\Ga_{y, 0}\left(t; \OurEpsilon_{y, u}\alpha\right)$ and $\nu_{y, 0}(dt; \alpha)$ are not independent of each other (they are both functions of $\alpha$) and we are constrained by limited integrability, moment-based proof of homogenisation could involve technical challenges. We shall circumvent this issue by means of change of measure.

Let us define $E_y\left(\mathbf{u}, \mathbf{t}\right) = E_y\left((u_1, u_2), (t_1, t_2)\right) := \EE\left[\exp\left(\sum_{j=1}^2 \Ga_{y, 0}(t_j; \OurEpsilon_{y, u_j}\alpha) \right)\right]$. Applying Fubini (for fixed $y$), one can rewrite
\begin{align}
\notag
&\EE\left[ \nu_{y, u_1}(I) \nu_{y, u_2}(I) e^{-L \nu_{y, 0}(I)}\right]
= \int_{I \times I} \EE\left[\nu_{y, u_1}(dt_1)\nu_{y, u_1}(dt_2)e^{-L \nu_{y, 0}(I)}\right]\\
\notag
& \qquad = \int_{I \times I} \frac{M_y(0)^2 E_y(\mathbf{u}, \mathbf{t})}{M_y(u_1)M_y(u_2)}\EE\left[\frac{\exp\left(\sum_{j=1}^2 \Ga_{y, 0}(t_j; \OurEpsilon_{y, u_j}\alpha) \right)}{E_y(\mathbf{u}, \mathbf{t})}\nu_{y,0}(dt_1; \alpha)\nu_{y, 0}(dt_2; \alpha) 
e^{-L \nu_{y, 0}(I; \alpha)}\right]\\
\label{eq:change-of-measure}
& \qquad =: \int_{I \times I}\frac{M_y(0)^2 E_y(\mathbf{u}, \mathbf{t})}{M_y(u_1)M_y(u_2)} \EE_y^{\mathbf{u}, \mathbf{t}}\left[\nu_{y,0}(dt_1; \alpha)\nu_{y, 0}(dt_2; \alpha) 
e^{-L \nu_{y, 0}(I; \alpha)}\right]
\end{align}
where on the last line we have introduced 
the expectation $\EE_y^{\mathbf{u}, \mathbf{t}}$ with respect to the new probability measure
\[    d\PP_y^{\mathbf{u}, \mathbf{t}}
    = \frac{\exp\left(\sum_{j=1}^2 \Ga_{y, 0}(t_j; \OurEpsilon_{y, u_j}\alpha) \right)}{E_y(\mathbf{u}, \mathbf{t})} d\PP\]
characterised by the property that $\left(\alpha(p)\right)_{p \le y}$ are independent random variables on the unit circle with marginal distributions given by
\begin{equation}\label{eq:change-of-measure-marginal}
    \EE_y^{\mathbf{u}, \mathbf{t}}\left[g(\alpha(p))\right]
    = \EE\left[\frac{\exp\left(2\Re\sum_{j=1}^2  \frac{f(p) \alpha(p)}{p^{1/2 + it_j}} \OurEpsilon_{y, u_j}(p) \right)}{\EE\left[\exp\left(2\Re\sum_{j=1}^2 \frac{f(p) \alpha(p)}{p^{1/2 + it_j}} \OurEpsilon_{y, u_j}(p) \right)\right]} g(\alpha(p))\right] \qquad \forall p \le y.
\end{equation}
To analyse \eqref{eq:change-of-measure}, we would like to couple $\EE_y^{\mathbf{u}, \mathbf{t}}$ simultaneously for all $\mathbf{t} = (t_1, t_2) \in I^2$. In more probabilistic terms, this means defining, apart from the Steinhaus random multiplicative function $\alpha$, a new collection of random variables $\left(\alpha_{y}^{\mathbf{u}, \mathbf{t}}(p)\right)_{\text{$p$ prime}, \mathbf{t} \in I^2}$ on the same (extended) probability space $(\Omega, \Fa, \PP)$ such that
\begin{enumerate}
    \item[(i)] $\left(\alpha_y^{\mathbf{u}, \mathbf{t}}(p)\right)_{\mathbf{t} \in I^2}$ (as collections of random variables indexed by $p$) are mutually independent; and
    \item[(ii)] $\left(\alpha_y^{\mathbf{u}, \mathbf{t}}(p)\right)_{\mathbf{t} \in I^2}$ has the correct marginals \eqref{eq:change-of-measure-marginal}, i.e.~$\EE\left[g\left(\alpha_y^{\mathbf{u}, \mathbf{t}}(p)\right)\right] = \EE_y^{\mathbf{u}, \mathbf{t}}\left[g(\alpha(p))\right]$ for each $p$ and $\mathbf{t} \in I^2$.\footnote{In \eqref{eq:change-of-measure} we imposed the condition that $p \le y$, but here we ignore the condition for brevity because our expression in the expectation only depends on $\alpha(p)$ for $p \le y$ and thus tweaking the distribution for $\left(\alpha(p)\right)_{p > y}$ in any arbitrary way is irrelevant to our analysis.}
\end{enumerate}
Suppose this is possible, and denote $\Delta_y^{\mathbf{u}, \mathbf{t}}(p) := \alpha_y^{\mathbf{u}, \mathbf{t}}(p) - \alpha(p)$. Then we can rewrite \eqref{eq:change-of-measure} as
\begin{equation}\label{eq:coupling_idea}
\EE\left[\int_{I \times I} \frac{M_y(0)^2 E_y(\mathbf{u}, \mathbf{t})}{M_y(u_1)M_y(u_2)} \nu_{y,0}(dt_1; \alpha + \Delta_y^{\mathbf{u}, \mathbf{t}})\nu_{y, 0}(dt_2; \alpha + \Delta_y^{\mathbf{u}, \mathbf{t}}) 
\exp\left(-L \nu_{y, 0}(I; \alpha + \Delta_{y}^{\mathbf{u}, \mathbf{t}})\right)\right].
\end{equation}
Since $\PP_{y}^{\mathbf{u}, \mathbf{t}}$ is a small perturbation of the original probability measure $\PP$, it may be reasonable to expect, under a suitable coupling, that $\Delta_y^{\mathbf{u}, \mathbf{t}}(p)$ behaves like 
\begin{equation}\label{eq:coupling-question}
\Delta_y^{\mathbf{u}, \mathbf{t}}(p)
    \overset{?}{=} \EE\left[\alpha_y^{\mathbf{u}, \mathbf{t}}(p)\right] + \text{residual random variable}
\end{equation}
where the residual term is negligible as $y \to \infty$. This would allow us to treat $\Delta_{y}^{\mathbf{u}, \mathbf{t}}(p)$ as if they were deterministic, and the rest of the analysis of \eqref{eq:coupling_idea} could be reduced to a straightforward computation. It may appear that this new change of measure argument is not too different from the approach adopted in the subcritical case, but a key distinction here is the use of simultaneous coupling of all measures $\PP_y^{\mathbf{u}, \mathbf{t}}$. This enables us to move the integral back into the expectation (as in \eqref{eq:coupling_idea}) and use the bound $x^2 e^{-LCx} \le 4/(LC)^2$, which helps us derive uniform integrability (see Step 1 in \Cref{sec:main_lemma}) and ultimately avoid various technical issues such as the intractable `fusion estimates'.

Note that this change-of-measure/coupling approach is novel even in the Gaussian context, and provides a new and elegant solution to the universality of critical GMCs under Seneta--Heyde renormalisation for instance, or an analogous problem in the context of HMC. Thanks to Girsanov's theorem, significant simplification could be achieved in the Gaussian case as one could choose a coupling such that the residual fluctuation in \eqref{eq:coupling-question} is identically equal to $0$. As a warm-up, we will illustrate these ideas in \Cref{sec:case-u=0} where we upgrade the distributional convergence of $\nu_{y, 0} \to \nu_\infty$ to convergence in probability via a Gaussian reduction by Saksman and Webb.

While the Steinhaus model does not admit an exact Girsanov's theorem, we will explain in \Cref{sec:mod-cross-mom} how to control the residual fluctuation uniformly under the so-called monotone coupling of phases, using techniques of concentration and generic chaining. With some extra effort, it should be possible to generalise our construction to more general class of random multiplicative functions $\alpha$ and twists $f$.

\subsection{Main idea for \texorpdfstring{\Cref{thm:summain}}{Theorem \ref{thm:summain}}}
As in the works on holomorphic multiplicative chaos \cite{NPS,NPSV} and our earlier works on random multiplicative functions \cite{GWL2,GWL1}, the tool behind the proof is the martingale central limit theorem. Let $P(n)$ be the largest prime factor of $n$. As observed by Harper \cite{Harper2013}, one can write our sum as
\[ S_x:=\frac{(\log \log x)^{\frac{1}{4}}}{\sqrt{x}}\sum_{n\le x} \alpha(n)= \sum_{p\le x} Z_{x,p}
\qquad \text{where} \qquad 
Z_{x,p} := \frac{(\log \log x)^{\frac{1}{4}}}{\sqrt{x}}\sum_{n\le x ,\,P(n) = p} \alpha(n).\]
The sequence $(Z_{x,p})_p$ indexed by primes is a \textit{martingale difference sequence} with respect to the filtration $\Fa_{p} := \sigma(\alpha(q), q \le p)$.  The martingale central limit theorem allows one to obtain the limiting distribution of $S_x$ under some conditions, the most important and difficult one being the computation of the limit of the \textit{bracket process}, which in our case is the random variable
\[T_{x} :=\sum_{p\le x} \EE\left[ |Z_{x,p}|^2 \mid \Fa_{p^-}\right] \approx \frac{\sqrt{\log \log x}}{x}\sum_{p \le x}  \big|\sum_{\substack{m\le x/p\\P(m)<p}}  \alpha(m)\big|^2, \]
where $\Fa_{p^-} := \sigma(\alpha(q), q < p)$. If we introduce the notation
\[ s_{t,y} := t^{-\frac{1}{2}}\sum_{n\le t,\, P(n) \le y}\alpha(n)\]
then, informally,
\begin{equation}\label{eq:Txapprox}
T_{x}\approx \sqrt{\log \log x}\sum_{p \le x} \frac{1}{p} |s_{x/p,p}|^2\approx \sqrt{\log \log x}\int_{2}^{x} \frac{|s_{x/t,t}|^2 dt}{t\log t}
\end{equation}
since $\sum_{p\le t} 1/p \approx \log \log t$ and $(\log \log t)' = 1/(t\log t)$. Given \Cref{thm:mc-critical}, we may obtain rather easily the limit of a somewhat similar expression, namely
\begin{equation}\label{eq:planc}
\sqrt{\log \log x}\int_{0}^{\infty} q(t^{1/\log x})\frac{|s_{x/t,x^a}|^2 dt}{t\log x}
\end{equation}
for any fixed $a>0$ and any `nice' function $q$, see \Cref{lem:plancherelapp} and \Cref{cor:sandwich}. The relationship of \eqref{eq:planc} to \Cref{thm:mc-critical} is via Plancherel's theorem. However, knowing the limit of \eqref{eq:planc} is not enough for studying \eqref{eq:Txapprox}: the term $|s_{x/t,t}|^2$ in the integrand of \eqref{eq:Txapprox}, where $t$ appears in both indices, is more complicated  than $|s_{x/t,x^a}|^2$ appearing in the integrand of \eqref{eq:planc}.

In our work in the $L^1$-regime we circumvented this issue by defining a new sum $S'_x$ which contained only a subset of the terms in $S_x$, adapting a similar truncation performed in the holomorphic multiplicative chaos (HMC) setting \cite{NPSV}. We use the same sum in the critical regime as well, see \eqref{eq:trunc} for its explicit definition. It goes as follows: divide the primes in $[2,x]$ into finitely many disjoint intervals $(x_k,x_{k+1}]$ ($k\ge 0$, $x_k$ is increasing), and if $n\le x$ has $P(n) \in (x_k,x_{k+1}]$ for $k\ge 1$, keep this $n$ in $S'_x$ only if $P(n/P(n))$ (the second largest prime factor of $n$) does not exceed $x_k$, and otherwise discard it. If $P(n)\le x_1$, we discard this $n$ as well. In this way, the martingale difference sequence of $S'_x$ is $(Z'_{x,p})_p$ for 
\[Z'_{x,p} := \frac{(\log \log x)^{\frac{1}{4}}}{\sqrt{x}}\sum_{\substack{n\le x,\, P(n) = p\\ P(n/P(n))\le x_k}} \alpha(n)\]
if $p\in (x_k,x_{k+1}]$ ($k \ge 1$), and the bracket process of $S'_x$ takes the shape
\begin{equation}\label{eq:newbracket}
T'_x \approx \sqrt{\log \log x}\sum_{k\ge 1} \sum_{x_k<p\le x_{k+1}} \frac{1}{p} |s_{x/p, x_k}|^2 \approx \sqrt{\log \log x}\sum_{k\ge 1} \int_{x_k}^{x_{k+1}} \frac{|s_{x/t,x_k}|^2 dt}{t\log t} .
\end{equation}
Each integral in the right-hand side can be handled using our understanding of \eqref{eq:planc}, by an appropriate choice of a function $q$. The second approximation in \eqref{eq:newbracket} is justified in \Cref{lem:close}, and requires some mild information on the asymptotics of primes in short intervals. It remains to explain how one justifies this truncation.

The justification in the $L^1$-regime is done in \cite[Lemma~3.4]{GWL1}. It goes by computing the second moment of the contribution of the discarded integers:
\begin{equation}\label{eq:diff} \EE \left[|S_x-S'_x|^2\right],
\end{equation}
and showing that it is small for an appropriate sequence $(x_k)_k$. The estimation of \eqref{eq:diff} is achieved by using the orthogonality relation
\begin{equation}\label{eq:orth}
\EE [\alpha(n)\overline{\alpha(m)}]=\delta_{n,m}
\end{equation}
for all $n,m \in \NN$, which reduces the task to computing the density of the discarded integers. We claim it is easy to see that the integers that are discarded constitute a sparse set - let us focus on the discarded integers with $x_{k+1}\ge P(n)\ge P(n/P(n))> x_k$ for some $k\ge 1$. These integers are divisible by $pq$ for two primes $p,q\in (x_k,x_{k+1}]$. The probability that a number is divisible by $pq$ is $\sim 1/(pq)$; summing this over $p,q\in(x_k,x_{k+1}]$ and then over $k\ge 1$ shows that the proportion of such integers is (at most)
\[ \sum_{k\ge 1} (\sum_{x_k<p\le x_{k+1}} 1/p)^2 \ll \sum_{k\ge 1} \log^2 (x_{k+1}/x_k)\]
by Mertens' theorem \eqref{eq:mertens35}. This bound goes to $0$ if we take $x_k=x^{\OurEpsilon+k\delta}$ and send $\delta$ to $0^+$. This brief argument was essentially enough in our work in the $L^1$-regime \cite{GWL1}. However, computing \eqref{eq:diff} in the critical regime is lossy due to the $(\log \log x)^{1/4}$-factor which breaks down the argument.

Instead, in \Cref{prop:negdelta} we bound the first absolute moment
\[ \EE |S_x-S'_x|,\]
which requires the same kind of tools required in the proof of Helson's conjecture. The proof of \Cref{prop:negdelta} occupies the entirety of Section \ref{sec:trunc}. Instead of using Cauchy--Schwarz and relating this expectation to $\EE\left[ |S_x-S'_x|^2\right]$ (which is easy to compute but gives a useless bound), we first condition of $\Fa_y = \sigma(\alpha(q),q\le p)$ where $y$ is a parameter at our disposal and only then apply Cauchy--Schwarz, obtaining
\[ \EE |S_x-S'_x| \le \EE\left[\big(\EE [|S_x-S'_x|^2 \mid \Fa_y] \big)^{1/2}\right]\]
by the law of total expectation. This is motivated by a conditioning argument of Harper \cite{Har2020} that was used by him to bound $\EE|S_x|$, and more generally to prove \eqref{eq:harpermoms}. We then estimate the inner expectation, $\EE [|S_x-S'_x|^2 \mid \Fa_y]$, and use basic analytic number theory to bound it by a sum of two random variables: one involves the integral  $\int_{0}^{\infty}|s_{t,y}|^2dt/t$ and the other has expectation which increases with $y$ (see e.g.~\Cref{lem:conditioning}), so we are motivated to take $y$ small relative to $x$. By Plancherel's theorem,  $\int_{0}^{\infty}|s_{t,y}|^2dt/t$ is \textit{equal} to a weighted integral of $|A_y(1/2+it)|^2$. This leads us to another required ingredient:
\begin{thm}[Harper, Saksman--Webb]\label{thm:ayintegral}
For $y \ge 3$ and fixed $q\in [0,1)$,
\begin{equation}\label{eq:Harper}
\EE \bigg[ \bigg(\frac{1}{\log y}\int_{\RR} \left|\frac{A_y(1/2+it)}{1/2+it}\right|^{2} dt \bigg)^{q}\bigg]\ll_q (\log \log y)^{-q/2}.
\end{equation}
\end{thm}
For $q=1/2$, the decay of the $1/2$-moment of the integral counteracts the $(\log \log x)^{1/4}$-factor in $S_x-S'_x$, as long as $y$ is not too small. For a proof of \Cref{thm:ayintegral} see Harper \cite[Key Propositions 1 and 2]{Har2020} \cite[Section 3]{harper2024moments} and Saksman and Webb \cite{SaksmanWebb,SW} (cf.~\cite[Lemma~1.3]{GWBLMS}). In fact, Harper's result is uniform in $q\in [0,1]$ and includes a matching lower bound. 
\subsection{Outline}
The remainder of the paper consists of two main parts that may be read independently.
\begin{itemize}
    \item Part 1 consists of \Cref{sec:chaos} and \Cref{sec:mod-cross-mom} which are dedicated to the proof of \Cref{thm:mc-critical}.
    \begin{itemize}
        \item \Cref{sec:chaos} contains various preparations for the universality result for $m_\infty$. We will first explain the density removal lemma (\Cref{lem:density-removal}), which allows us to reduce the problem of $m_{y, u} \to m_\infty$ (\Cref{thm:mc-critical}) to that of $\nu_{y, u} \to \nu_\infty$ (\Cref{thm:mc-critical2}). We will then explain the modified second moment method (\Cref{lem:mod-2ndmom}), and as a warm-up use it to strengthen the result $\nu_{y, 0} \xrightarrow[y \to \infty]{d} \nu_\infty$ of Saksman and Webb to convergence in probability in \Cref{sec:case-u=0}.
        \item \Cref{sec:mod-cross-mom} focuses on the implementation of the modified second moment method. We will first explain the monotone coupling of phases (\Cref{lem:couple-estimate}) and derive uniform estimates for the so-called residual field (\Cref{cor:residual-estimate}). These are then used in \Cref{sec:main_lemma} as key ingredients for uniform integrability and convergence in probability, from which we conclude \Cref{lem:mod-cross-mom} by generalised dominated convergence. We will conclude the section with a short proof of \Cref{cor:mc-critical}.
    \end{itemize}
    \item Part 2 consists of \Cref{sec:prepsum}, \Cref{sec:trunc} and \Cref{sec:bracket} which are dedicated to the proof of \Cref{thm:summain}.
    \begin{itemize}
        \item \Cref{sec:prepsum} establishes \Cref{thm:summain} using the martingale central limit theorem up to two tasks: the justification of the truncation (\Cref{prop:negdelta}) and the computation of the limit of the bracket process (\Cref{prop:bracket}).
        \item \Cref{sec:trunc} is concerned with the truncation and with verifying \Cref{prop:negdelta}.
        \item \Cref{sec:bracket} computes the bracket process using \Cref{cor:mc-critical}, establishing \Cref{prop:bracket}. 
\end{itemize}
\end{itemize}
\section{Preparation for the proof of \texorpdfstring{\Cref{thm:mc-critical}}{Theorem \ref{thm:mc-critical}}}\label{sec:chaos}
For \Cref{sec:chaos} and \Cref{sec:mod-cross-mom}, we will state our results in the form of general random Euler product
\[A_y(z) 
:=
\prod_{p \le y} \bigg(\sum_{k \ge 0} \frac{f(p^k)\alpha(p^k)  }{p^{kz}}\bigg)=\sum_{\substack{n\ge 1\\p\mid n \implies p \le y}} \frac{f(n)\alpha(n)}{n^z}. \]
Even though we are primarily interested in the case $f \equiv \mathbf{1}$ (or $f \in \{\mu, \mu^2\}$ for partial sums restricted to squarefree numbers), our proofs will track the dependence on $f$ in most places and one can easily see that our arguments do not need any modifications if we work with the following assumptions:
\begin{itemize}
\item $\displaystyle \sum_p \bigg[\frac{|f(p)|^2}{p}
+ \frac{|f(p^2)|^2}{p^{2}} +\sum_{k \ge 3} \frac{|f(p^k)|}{p^{k/2}} \bigg]\log^2 p < \infty$; and
\item $\sum_{p \le x} |f(p)|^2 = \mathrm{Li}(x) + \Ea_{|f|^2}(x)$ where $\Ea_{|f|^2}(x) = o(x / \log x)$ as $x \to \infty$ and $\int_{2}^\infty x^{-2}|\Ea_{|f|^2}(x)| dx < \infty.$
\end{itemize}
The only place where the assumption $f(p) \in \{\pm 1\}$ is implicitly used in these two sections is the discussion at the beginning of \Cref{sec:case-u=0}, where we explained the Gaussian reduction of Saksman and Webb when they established the distributional convergence of $\nu_{y, 0} \to \nu_\infty$. We note, however, that this Gaussian reduction could also be extended to the aforementioned class of twists $f$ using the techniques in \cite{SaksmanWebb, SW} and we refer the interested readers to the original papers for details.

\subsection{A reduction: the density removal lemma}\label{sec:density-removal}
Writing
\begin{equation}\label{eq:density}
m_{y, u}(dt) = X_{y, u}(t) \nu_{y, u}(dt) 
\qquad \text{where} \qquad 
X_{y, u}(t) := \frac{\EE\left[ \exp \Ga_{y, u}(t)\right]}{\EE\left[|A_y(\sigma_y(u) + it)|^2\right]}\frac{|A_y(\sigma_y(u) + it)|^2}{\exp  \Ga_{y, u}(t)},
\end{equation}
it is intuitive to approach \Cref{thm:mc-critical} by separately obtaining the convergence of $X_{y, u}(t)$ in a suitable function space and then combining that with \Cref{thm:mc-critical2}. This is the content of the density removal lemma, which was first used in \cite{GWL1} and is now specialised to our current problem.\footnote{An almost-sure version was also used in the work of Saksman and Webb \cite{SW}.}
\begin{lem}[{\cite[Lemma 2.10]{GWL1}}]\label{lem:density-removal}
Let $I \subset \RR$ be a compact interval. Suppose there exist a random continuous function $X_\infty \in C(I)$ and a random measure $\nu_\infty$ on $I$ such that 
\[    \forall u \ge 0, \qquad X_{y, u} \xrightarrow[n \to \infty]{p} X_\infty \quad \text{in $C(I)$}
    \qquad \text{and} \qquad
    \nu_{y, u} \xrightarrow[n \to \infty]{p} \nu_\infty. \]
Then
\[    \forall u \ge 0, \qquad  m_{y, u}(dt) \xrightarrow[y \to \infty]{p} m_\infty(dt) := X_\infty(t) \nu_\infty(dt). \]
\end{lem}
In view of \Cref{lem:density-removal}, let us verify that $X_{y, u}$ indeed has the desired convergent property.
\begin{lem}\label{lem:density-estimate}
Let $I \subset \RR$ be a compact interval and $X_{y, u}$ be defined in \eqref{eq:density}. Then $\sup_{y \ge 3} \EE\left[\sup_{t \in I} |X_{y, u}(t)|^r\right] < \infty$ for any $r \in \RR$. Moreover, there exists some random continuous function $X_{\infty}(\cdot)$ such that $X_{y, u} \xrightarrow[y \to \infty]{p} X_\infty$ in $C(I)$ for any $u \ge 0$.
\end{lem}
This was essentially established in \cite[Lemma 2.11]{GWL1} under slightly different settings and notations. For completeness we give a quick sketch of proof, which relies on the following lemma which was a small generalisation of \cite[Lemma 3.1]{SW}:
\begin{lem}[{\cite[Lemma 2.9]{GWL1}}]\label{lem:randomFourier}
Let $I \subset \RR$ be a compact interval, and consider the sequence of random functions $F_n(\cdot) := \sum_{k=1}^n U_k f_{n, k}(\cdot)$ where
\begin{itemize}
\item $(U_k)_{k}$ are i.i.d.~(real- or complex-valued) random variables that are either (i) standard normal or (ii) symmetric and uniformly bounded (i.e.~there exists some constant $C$ such that $|U_k| \le C$ almost surely);
\item $f_{n, k} \xrightarrow[n \to \infty]{} f_k$ in $C^1(I)$ for each $k \ge 1$, and
$
\sum_{k=1}^\infty \sup_{n \ge 1} \left[ \| f_{n,k}\|_{\infty}^2 + \| f_{n,k}'\|_{\infty}^2\right] < \infty.
$
\end{itemize}
Then $\sup_{n \ge 1} \EE\left[\exp\left( \lambda \|F_n\|_{\infty}\right)\right] < \infty$ for any $\lambda > 0$, and $F_n(\cdot) \xrightarrow[n \to \infty]{p} F(\cdot) := \sum_{k \ge 1} U_k f_k(\cdot)$ in $C(I)$. 
\end{lem}
\begin{proof}[Sketch of proof of \Cref{lem:density-estimate}]
Let us write $X_{y, u}(t) := \prod_{j=1}^3 X_{y, u}^{(j)}(t)$ where
\begin{align*}
X_{y, u}^{(1)}(t) & := \frac{\EE\exp \left(\Ga_{y, u}(t)\right)}{\EE\left[|A_y(\sigma_y(u) + it)|^2 \right]}
= \prod_{p \le y} \frac{\EE\left[\exp\{2 \Re \frac{f(p)\alpha(p)}{p^{\sigma_y(u) + it}}\}\right]}{\EE\left[ \left| \sum_{k \ge 0} \frac{f(p^k)\alpha(p)^k}{p^{k(\sigma_y(u) + it)}}\right|\right]},\\
X_{y, u}^{(2)}(t) & :=\prod_{p \le y}  \left|1 + \sum_{k =1}^2 \frac{\alpha(p)^k f(p^k)  }{p^{k(\sigma_y(u) + it)}}\right|^{-2}\left|1 + \sum_{k \ge 1} \frac{\alpha(p)^k f(p^k)  }{p^{k(\sigma_y(u)+ it)}}\right|^2,  \\
X_{y, u}^{(3)}(s) & :=\prod_{p \le y} \exp\left(-2\Re \frac{f(p) \alpha(p)}{p^{\sigma_y(u) + it}}\right) \left|1 + \sum_{k =1}^2 \frac{\alpha(p)^k f(p^k)  }{p^{k(1/2 + it)}}\right|^{2}.
\end{align*}
We only need to verify analogous claims for $X_{y, u}^{(j)}$ separately. The integrability of $X_{y, u}^{(1)}(\cdot)$ and $X_{y, u}^{(2)}(\cdot)$ is trivial since they are bounded above and also away from $0$ deterministically and uniformly in $y \ge 3$, and their convergence in probability (which actually also holds in the stronger almost sure sense) can be verified by dominated convergence - we omit the details\footnote{This should be easy to observe when $f \in \{\mathbf{1}, \mu, \mu^2\}$. For the general class of twists, we leave the straightforward verification to the interested readers; see \cite[Lemma 2.11]{GWL1} for similar computations.} and instead focus on the $X_{y, u}^{(3)}$ now. Observe that
\begin{align*}
&\left|1 + \sum_{k=1}^2 \frac{\alpha(p)^k f(p^k)  }{p^{k(\sigma_y(u)+it)}}\right|^2
 = \exp \left\{
2\Re \left(\sum_{j=1}^2 \frac{(-1)^{j-1}}{j} \left(\sum_{k=1}^2 \frac{\alpha(p)^k f(p^k)  }{p^{k(\sigma_y(u)+it)}}\right)^j\right) + O\left(\frac{|f(p)|^3}{p^{3/2}} + \frac{|f(p^2)|^3}{p^3} \right)
 \right\}\\
&\quad  = \exp \left\{ \Re \left[ 2\frac{\alpha(p) f(p)}{p^{\sigma_y(u) + it}} + \frac{\alpha(p)^2 (2f(p^2) - f(p)^2)}{p^{2(\sigma_t +is)}} - 2\frac{\alpha(p)^3f(p)f(p^2)}{p^{3(\sigma_y(u) + it)}}\right] + O\left(\frac{|f(p)|^3}{p^{3/2}} + \frac{|f(p^2)|^2}{p^2}\right)\right\}.
\end{align*}
The ``error terms" in $O\left(\frac{|f(p)|^3}{p^{3/2}} + \frac{|f(p^2)|^2}{p^2}\right)$ are random continuous functions that have summable deterministic bounds (uniformly in $y\ge 3$), and this means the sum over $p \le y$ converges almost surely to some continuous field as $y \to \infty$ by dominated convergence.
Finally, we need to show that the field
\[ \Re \sum_{p \le y} \left[ \frac{\alpha(p)^2 (2f(p^2) - f(p)^2)}{p^{2(\sigma_y(u) + it)}} - 2\frac{\alpha(p)^3f(p)f(p^2)}{p^{3 (\sigma_y(u) +it)}}\right] \]
has uniformly bounded exponential moments of all orders and that it converges in probability in $C(I)$, but 
\[ \sum_{p}\left\{ \frac{\left[|f(p^2)|^2 + |f(p)|^4\right]\log^2 p}{p^2} + \frac{ |f(p)|^2 |f(p^2)|^2 \log^2 p}{p^3}\right\} < \infty, \]
allows us to conclude the proof by \Cref{lem:randomFourier}.
\end{proof}

\subsection{The modified second moment method}\label{sec:mod-2nd}
To establish the universality result, i.e.~$\nu_{y, u} \xrightarrow[y \to \infty]{p} \nu_\infty$ regardless of the value of $u \ge 0$, we make use of the modified second moment method, as explained in the lemma below.
\begin{lem}\label{lem:mod-2ndmom}
Let $(X_n)_{n \ge 1}, (Y_n)_{n \ge 1}$ be two sequences of real-valued random variables. Suppose $X_n \xrightarrow[n \to \infty]{p} X_\infty$ and $\lim_{n \to \infty} \EE[|X_n - Y_n|^2 e^{-L X_n}] = 0$ for any $L > 0$. Then $Y_n \xrightarrow[n \to \infty]{p} X_\infty$.    
\end{lem}
\begin{proof}
For any $\delta > 0$ and $L > 0$,
\begin{align*}
\PP(|Y_n - X_\infty| > \delta)
&\le \PP\left(|Y_n - X_n| > \frac{\delta}{2}\right)
+ \PP\left(|X_n - X_\infty| > \frac{\delta}{2}\right)\\
&\le \PP(X_n > L)
+ \PP\left(|Y_n - X_n| > \frac{\delta}{2}, X_n \le L\right)
+ \PP\left(|X_n - X_\infty| > \frac{\delta}{2}\right)\\
& \le \PP(X_n > L)
+ (2/\delta)^2\EE\left[|X_n - Y_n|^2 e^{1 - X_n / L}\right]
+ \PP\left(|X_n - X_\infty| > \frac{\delta}{2}\right).
\end{align*}
As $n \to \infty$, the second term on the right-hand side goes to zero by assumption, and the third term also vanishes due to $X_n \xrightarrow[n \to \infty]{p} X_\infty$. Therefore
\[    \limsup_{n \to \infty} \PP(|Y_n - X_\infty| > \delta) \le \limsup_{n \to \infty} \PP(X_n > L)
     \le \PP(X_\infty \ge L) \]
where the last step follows from Portmanteau theorem. The desired claim now follows by sending $L \to \infty$.
\end{proof}
\begin{proof}[Proof of \Cref{thm:mc-critical2}, assuming the case $u=0$ and \Cref{lem:mod-cross-mom}]
Let $y_n$ be any positive sequence tending to $\infty$. We want to show, for any $u \ge 0$ and compact interval $I \subset \RR$, that $\nu_{y_n, u}(I) \xrightarrow[n \to \infty]{p} \nu_\infty(I)$. 
For any $L > 0$, 
\begin{align*}
&\lim_{n \to \infty} \EE\left[|\nu_{y_n, 0}(I) - \nu_{y_n, u}(I)|^2 e^{-L\nu_{y, 0}(I)}\right]\\
& \quad = \lim_{n \to \infty}
\left\{\EE\left[ \nu_{y_n, 0}(I)^2 e^{-L \nu_{y, 0}(I)}\right]
- 2\EE\left[ \nu_{y_n, 0}(I) \nu_{y_n, u}(I) e^{-L \nu_{y, 0}(I)}\right]
+ \EE\left[ \nu_{y_n, u}(I)^2(I) e^{-L \nu_{y, 0}(I)}\right]
\right\}
= 0
\end{align*}
by \Cref{lem:mod-cross-mom}. The result now follows from \Cref{lem:mod-2ndmom} with $X_n := \nu_{y_n, 0}(I)$ and $Y_n := \nu_{y_n, u}(I)$.
\end{proof}

\subsection{The case \texorpdfstring{$u=0$}{u=0} and the existence of \texorpdfstring{$\nu_\infty$}{nuinfty}}\label{sec:case-u=0}
Let us first briefly comment on the approach of of Saksman and Webb in \cite{SaksmanWebb, SW} where they established $\nu_{y, 0} \to \nu_\infty$. Their main idea was Gaussian approximation: it is possible to construct another random completely multiplicative function $\widetilde{\alpha}$ on the same (extended) probability space where $\widetilde{\alpha}(p)$ are i.i.d.~standard complex Gaussian random variables, such that $\Ga_{y, 0}(\cdot; \widetilde{\alpha}) - \Ga_{y, 0}(\cdot;\widetilde{\alpha})$ converges almost surely to a random continuous function, see e.g.~\cite[Theorem 1.7]{SW}. After further manipulations (\cite[Section 7]{SW}), they identified another sequence of Gaussian fields $G_y(\cdot)$  such that $\Ga_{y, 0}(\cdot; \alpha) - G_y(\cdot)$ converges almost surely to another random continuous function, and the critical chaos associated to $G_y(\cdot)$ exists in the limit as $y \to \infty$. 

The reason why convergence did not hold in probability is that the last step made use of \cite[Theorem 1.1]{JS2017}, which guarantees distributional convergence due to its abstract formulation. It may be possible to refine their analysis to upgrade the convergence, but this requires us to explain the steps leading to the identification of the Gaussian fields $G_y(\cdot)$. As promised, we will instead provide a relatively self-contained proof via modified second moment and homogenisation by change of measure. In view of \Cref{lem:density-removal}, we just need to prove that $\nu_{y, 0}(\cdot; \widetilde{\alpha})$ converges in probability to some limit $\widetilde{\nu}_\infty(\cdot)$ as $y \to \infty$. We will restrict ourselves to $I := [-1/2, 1/2]$ for simplicity, though the arguments could be easily adapted to general intervals.

\subsubsection{A candidate for the limit of \texorpdfstring{$\nu_y(\cdot; \widetilde{\alpha})$}{nu infty dot tilde alpha}}
We first need to identify a limit candidate. Consider
\[\EE\left[ \Ga_{y, 0}(t_1; \widetilde{\alpha})\Ga_{y, 0}(t_2; \widetilde{\alpha})\right]
= 2\sum_{p\le y} \frac{|f(p)|^2}{p} \cos((t_1 - t_2) \log p) =: K_y(t_1 - t_2).\]
When $y = \infty$, we see that
\begin{align*}
    K_\infty(t_1-t_2) &= 2\int_{2^-}^\infty \frac{\cos\left((t_1 - t_2 )\log x\right)}{x} d\left[\mathrm{Li}(x) + \Ea_{|f|^2}(x)\right]\\
    & = 2\left[\int_{(t_1 - t_2) \log 2}^1 \frac{du}{u} + \int_{(t_1 - t_2) \log 2}^1 \frac{\cos u - 1}{u} du + \int_1^\infty \cos u\frac{ du}{u}\right]\\
    & \qquad - 2 \Ea_{|f|^2}(2^-) \cos((t_1 - t_2) \log 2) + 2 \int_2^\infty \frac{\Ea(x)}{x^2} \left[ 1 + (t_1 - t_2) \sin\left((t_1 - t_2)\log x\right)\right] dx\\
    &= - 2 \log|t_1 - t_2| + \text{smooth function in $(t_1 - t_2)$}.
\end{align*}
By \cite[Theorem 5.3]{JSW2019} or \cite[Theorem 1.3 and C.2]{Lac2024}, the critical GMC associated to $\Ga_{\infty, 0}(\cdot; \widetilde{\alpha})$ under the convolution construction with Seneta--Heyde renormalisation exists. In other words, there exists $\widetilde{\nu}_\infty$ on $I$ such that the following is true: for any fixed smooth mollifier\footnote{For our purpose, a smooth mollifier $\theta$ is a smooth, symmetric and compactly supported function that is non-negative and $\int_{\RR} \theta = 1$, though it is known that the convolution construction of critical GMCs can be extended to less regular setting.} $\theta_\OurEpsilon(\cdot) := \OurEpsilon^{-1} \theta(\cdot / \OurEpsilon)$, the sequence of measures
\[    \widetilde{\nu}_{y, 0}(dt; \widetilde{\alpha}) := \sqrt{\log \frac{1}{\OurEpsilon}} \frac{\exp\left\{\Ga_{\infty, 0}^\OurEpsilon\left(t; \widetilde{\alpha}\right)\right\}}{\EE\left[\exp\left\{\Ga_{\infty, 0}^\OurEpsilon\left(t; \widetilde{\alpha}\right)\right\}\right]}dt
    \qquad \text{with} \qquad \OurEpsilon^{-1} = \log y, \quad \Ga_{\infty, 0}^\OurEpsilon(t; \widetilde{\alpha}) := \left(\Ga_{\infty, 0}(\cdot; \widetilde{\alpha}) \ast \theta_\OurEpsilon\right)(t)\]
converges in probability to $\widetilde{\nu}_\infty(dt)$ as $y \to \infty$. The support properties (i.e.~non-atomicity and full support) of $\widetilde{\nu}_\infty$ are standard facts in the literature and are conveniently summarised in e.g.~\cite[Theorem 1.3]{Lac2024}, though they follow from earlier works on critical GMCs. For instance, non-atomicity is implied by stronger statements of modulus of continuity such as \cite[Theorem 2]{BKNSW2015},\footnote{This was proved for scale-invariant kernels, but could be extended by a decomposition theorem in \cite{JSW2019}.} while full support could be deduced from the existence of negative moments of critical GMC measures \cite[Corollary 6]{DRSV2014}.\footnote{This was proved for star-scale invariant fields, but may be extended by Gaussian comparison (\cite[Lemma 16]{DRSV2014}). See the proof of \cite[Lemma 3.4]{GWBLMS} for a similar argument applied to positive moments.}

\subsubsection{Convergence of \texorpdfstring{$\nu_{y, 0}(\cdot; \widetilde{\alpha})$}{nu y 0 dot tilde alpha}}
We now demonstrate the modified second moment method and show that
\[
    \qquad \forall L > 0, \qquad \lim_{y \to \infty} \EE\left[\left|\nu_{y, 0}(I;\widetilde{\alpha}) - \widetilde{\nu}_{y, 0}(I;\widetilde{\alpha})\right|^2 e^{- L\widetilde{\nu}_{y, 0}(I;\widetilde{\alpha})} \right] = 0.
\]
Expanding the square, this amounts to proving that
\[
    \lim_{y \to \infty} \EE\left[\nu_{y, 0}(I;\widetilde{\alpha})^2e^{- L\widetilde{\nu}_{y, 0}(I;\widetilde{\alpha})} \right]
    = \lim_{y \to \infty} \EE\left[\nu_{y, 0}(I;\widetilde{\alpha})\widetilde{\nu}_{y, 0}(I;\widetilde{\alpha}) e^{- L\widetilde{\nu}_{y, 0}(I;\widetilde{\alpha})} \right] 
    = \lim_{y \to \infty} \EE\left[\widetilde{\nu}_{y, 0}(I;\widetilde{\alpha})^2 e^{- L\widetilde{\nu}_{y, 0}(I;\widetilde{\alpha})} \right].
\]
We shall only study the first term below since the cross term can be treated in a similar fashion. Let us introduce notations similar to those in \eqref{eq:change-of-measure}:
denote
\[
    d\widetilde{\PP}_y^{\mathbf{t}} := \frac{\exp\left\{\sum_{j=1}^2 \left[\Ga_{\infty, 0}^\OurEpsilon\left(t_j; \widetilde{\alpha}\right) - \Ga_{y, 0}\left(t_j; \widetilde{\alpha}\right)\right]\right\}}{\widetilde{E}_y(\mathbf{t})}d\PP,
    \quad \widetilde{E}_y(\mathbf{t}) := \EE\left[\exp\left\{\sum_{j=1}^2 \left[\Ga_{\infty, 0}^\OurEpsilon\left(t_j; \widetilde{\alpha}\right) - \Ga_{y, 0}\left(t_j; \widetilde{\alpha}\right)\right]\right\}\right],
\]
and rewrite $\EE\left[\nu_{y, 0}(I;\widetilde{\alpha})^2e^{- L\widetilde{\nu}_{y, 0}(I;\widetilde{\alpha})} \right]$ as
\[ \EE\left[\int_{I \times I} \frac{
    \EE\left[\exp\left\{\Ga_{y, 0}\left(0; \widetilde{\alpha}\right)\right\}\right]^2}{\EE\left[\exp\left\{\Ga_{\infty, 0}^\OurEpsilon\left(0; \widetilde{\alpha}\right)\right\}\right]^2}
\widetilde{E}_y(\mathbf{t})
\widetilde{\nu}_{y, 0}(dt_1; \widetilde{\alpha}_y^\mathbf{t})\widetilde{\nu}_{y, 0}(dt_2; \widetilde{\alpha}_y^\mathbf{t})e^{- L \widetilde{\nu}_{y, 0}(I; \widetilde{\alpha}_y^\mathbf{t})}\right] \]
where $\widetilde{\alpha}_y^{\mathbf{t}}(p)$ under $\PP$ has the same distribution as $\widetilde{\alpha}(p)$ under $\widetilde{\PP}_y^{\mathbf{t}}$ for each $p$. 

Before proceeding further, let us evaluate and introduce some notations for various moment generating functions: if we write $K_y^\OurEpsilon(t_1 - t_2) = \EE\left[\Ga_{y, 0}(t_1; \widetilde{\alpha})\Ga_{\infty, 0}^\OurEpsilon(t_2; \widetilde{\alpha})\right]
= \EE\left[\Ga_{y, 0}(t_2; \widetilde{\alpha})\Ga_{y, 0}^\OurEpsilon(t_1; \widetilde{\alpha})\right]$,\footnote{We will use the fact that $K_y^{\OurEpsilon}$ is symmetric in $(t_1, t_2)$ and translation invariant, which follows from our assumption that the smooth mollifier is symmetric. This is a simplifying assumption only for the purpose of introducing fewer notations.} then
\begin{align*}
    & \EE\left[\exp\left\{\Ga_{y, 0}\left(0; \widetilde{\alpha}\right)\right\}\right] = \exp\left\{\frac{1}{2} K_y(0)\right\},
    \qquad \EE\left[\exp\left\{\Ga_{\infty, 0}^\OurEpsilon\left(0; \widetilde{\alpha}\right)\right\}\right] = \exp\left\{\frac{1}{2} K_\infty^\OurEpsilon(0)\right\},\\
    \text{and} \qquad  &
    \widetilde{E}_y(\mathbf{t}) = \exp\left\{K_\infty^\OurEpsilon(0) + K_y(0) - 2K_y^\OurEpsilon(0) + K_\infty^\OurEpsilon(t_1 - t_2) + K_y(t_1 - t_2) - 2K_y^\OurEpsilon(t_1 - t_2) \right\}.
\end{align*}
At this stage, recall Girsanov's theorem which says that 
\[    \EE\left[\frac{\exp(N_1)}{\EE\exp(N_1)} g(N_2)\right]
    = \EE\left[g(N_2 + \EE[N_1 N_2])\right]  \]
for any centred $(N_1, N_2)$ that are jointly Gaussian (and $N_1$ is real-valued), i.e.~the effect of change of measure is merely a shift in mean. Therefore, the most convenient coupling in the current setting would be to take
\[ 
\widetilde{\alpha}_y^{\mathbf{t}}(p) 
:= \widetilde{\alpha}(p) + \EE\left[\widetilde{\alpha}(p)\sum_{j=1}^2 \left[\Ga_{\infty, 0}^\OurEpsilon\left(t_j; \widetilde{\alpha}\right) - \Ga_{y, 0}\left(t_j; \widetilde{\alpha}\right)\right]\right].
\]
With this particular coupling (and linearity of $\widetilde{\nu}_{y, 0}(\cdot; \cdot)$ in the second variable), it is easy to check that
\begin{align*}
    \widetilde{\nu}_{y, 0}(dt_0; \widetilde{\alpha}_y^{\mathbf{t}})
    &= \exp\left\{\EE\left[\Ga_{\infty, 0}^{\OurEpsilon}\left(t_0; \widetilde{\alpha}\right)\sum_{j=1}^2 \left[\Ga_{\infty, 0}^\OurEpsilon\left(t_j; \widetilde{\alpha}\right) - \Ga_{y, 0}\left(t_j; \widetilde{\alpha}\right)\right] \right]\right\}\widetilde{\nu}_{y, 0}(dt_0; \widetilde{\alpha})\\
    & = \exp\left\{\sum_{j=1}^2 \left[K_\infty^\OurEpsilon(t_0 - t_j) - K_y(t_0 - t_j)\right]\right\}\widetilde{\nu}_{y, 0}(dt_0; \widetilde{\alpha}).
\end{align*}
At the same time, we also have
\begin{align*}
&\frac{
    \EE\left[\exp\left\{\Ga_{y, 0}\left(0; \widetilde{\alpha}\right)\right\}\right]^2}{\EE\left[\exp\left\{\Ga_{\infty, 0}^\OurEpsilon\left(0; \widetilde{\alpha}\right)\right\}\right]^2}
\widetilde{E}_y(\mathbf{t})
\widetilde{\nu}_{y, 0}(dt_1; \widetilde{\alpha}_y^\mathbf{t})\widetilde{\nu}_{y, 0}(dt_2; \widetilde{\alpha}_y^\mathbf{t})\\
& = \exp\left\{
3K_\infty^\OurEpsilon(t_1 - t_2) - K_y(t_1 - t_2) - 2K_y^\OurEpsilon(t_1 - t_2)
\right\}
\widetilde{\nu}_{y, 0}(dt_1; \widetilde{\alpha})\widetilde{\nu}_{y, 0}(dt_2; \widetilde{\alpha}).
\end{align*}
It is a straightforward to check (see also e.g.~\cite[Lemma 3.5]{Ber2017}) that all three kernels $K_\infty^\OurEpsilon(t_1 - t_2), K_y(t_1 - t_2)$ and $K_y^\OurEpsilon(t_1 - t_2)$ (with $\OurEpsilon = 1 / \log y$) are equal to $-2 \log \max\left(|t_1 - t_2|, \OurEpsilon\right) + O(1)$ uniformly in $y \ge 3$ and $t_1, t_2 \in I$, and that they all converge to $K_\infty(t_1 - t_2)$ uniformly for $|t_1 - t_2|$ bounded away from $0$. As such,
\begin{equation}\label{eq:upgrade-intermediate3}
   \int_{I \times I}\frac{
    \EE\left[\exp\left\{\Ga_{y, 0}\left(0; \widetilde{\alpha}\right)\right\}\right]^2}{\EE\left[\exp\left\{\Ga_{\infty, 0}^\OurEpsilon\left(0; \widetilde{\alpha}\right)\right\}\right]^2}
\widetilde{E}_y(\mathbf{t})
\widetilde{\nu}_{y, 0}(dt_1; \widetilde{\alpha}_y^\mathbf{t})\widetilde{\nu}_{y, 0}(dt_2; \widetilde{\alpha}_y^\mathbf{t})e^{- L \widetilde{\nu}_{y, 0}(I; \widetilde{\alpha}_y^\mathbf{t})}
\end{equation}
\begin{itemize}
    \item is uniformly bounded by $C \widetilde{\nu}_{y, 0}(I; \widetilde{\alpha})^2 e^{-LC^{-1} \widetilde{\nu}_{y, 0}(I; \widetilde{\alpha})} \le 4C^3 / L^2$ for some deterministic $C \in (0, \infty)$; and
    \item converges in probability as $y \to \infty$ to $\widetilde{\nu}_{\infty}(I)^2e^{-L\widetilde{\nu}_{\infty}(I)}$ due to the non-atomicity of $\widetilde{\nu}_\infty$.\footnote{Differences between pairs of $K_\infty^\OurEpsilon(t_1 - t_2), K_y(t_1 - t_2)$ or $K_y^\OurEpsilon(t_1 - t_2)$ on the diagonal $t_1 = t_2$ do not necessarily converge to $0$ in the limit $y \to \infty$, but the non-atomicity of $\widetilde{\nu}_\infty$ essentially allows us to ignore these issues. Such analysis will be performed more carefully when we study the actual Steinhaus model, see Step 2 in \Cref{sec:main_lemma}.}
\end{itemize}
Therefore, the expectation of \eqref{eq:upgrade-intermediate3} converges to $\EE\left[\widetilde{\nu}_\infty(I;\widetilde{\alpha})^2 e^{-L\widetilde{\nu}_\infty(I;\widetilde{\alpha})}\right]$ by dominated convergence and our proof is complete.

\section{Convergence of modified moments: proof of \texorpdfstring{\Cref{lem:mod-cross-mom}}{Theorem \ref{lem:mod-cross-mom}}}\label{sec:mod-cross-mom}
Now that we have seen the method of homogenisation by change of measure in action, let us return to our original problem on Steinhaus multiplicative function. Recall from \eqref{eq:change-of-measure} and \eqref{eq:coupling_idea} that
\[ \EE\left[ \nu_{y, u_1}(I) \nu_{y, u_2}(I) e^{-L \nu_{y, 0}(I)}\right]
= \EE\left[\int_{I \times I}\frac{M_y(0)^2 E_y(\mathbf{u}, \mathbf{t})}{M_y(u_1)M_y(u_2)} \nu_{y,0}(dt_1; \alpha_y^{\mathbf{u}, \mathbf{t}})\nu_{y, 0}(dt_2; \alpha_y^{\mathbf{u}, \mathbf{t}}) 
\exp\left(-L \nu_{y, 0}(I; \alpha_{y}^{\mathbf{u}, \mathbf{t}})\right)\right] \]
where 
\[\nu_{y, u}(dt_0; v) := \sqrt{\log \log y}M_y(u)^{-1} \Ga_{y, u}(t_0; v) dt_0 \qquad \text{with} \qquad \Ga_{y, u}(t_0; v) := 2\Re\sum_{p \le y} \frac{f(p) v(p)}{p^{\sigma_y(u) + it_0}},\]
$M_y(u_j) := \EE[\exp(\Ga_{y, u_j}(t; \alpha))]$,  and $ E_y(\mathbf{u}, \mathbf{t}) = E_y\left((u_1, u_2), (t_1, t_2)\right) := \EE\left[\exp\left(\sum_{j=1}^2 \Ga_{y, 0}(t_j; \OurEpsilon_{y, u_j}\alpha) \right)\right].$

We would like to simultaneously define the random variables
$\left(\alpha(p)\right)_{\text{$p$}}$ and $\left(\alpha_{y}^{\mathbf{u}, \mathbf{t}}(p)\right)_{\text{$p$}, ~\mathbf{t} \in [0,1]^2}$ on the same probability space. Ideally one would expect that whenever $\mathbf{t}, \tilde{\mathbf{t}} \in [0,1]^2$ are close to each other, the realisation of the two random collections $\left(\alpha_{y}^{\mathbf{u}, \mathbf{t}}(p)\right)_{p}$ and $\left(\alpha_{y}^{\mathbf{u}, \tilde{\mathbf{t}}}(p)\right)_{p}$ should also be close to each other.

\subsection{Monotone coupling of phases}
Let us represent our random variables by their phases
\[    \alpha(p) := \exp\left(2\pi i \phi(p)\right)\qquad \text{and} \qquad \alpha_y^{\mathbf{u}, \mathbf{t}}(p) := \exp\left(2 \pi i \phi_y^{\mathbf{u}, \mathbf{t}}(p)\right) \]
where $\phi(p), \phi_y^{\mathbf{u}, \mathbf{t}}(p) \in [0, 1]$, and denote their probability densities by $d_\infty(\cdot)$ and $d_y^{\mathbf{u}, \mathbf{t}}(\cdot)$ respectively, i.e.
\[    d_\infty(\phi) \equiv 1 \qquad \text{and} \qquad 
    d_y^{\mathbf{u}, \mathbf{t}}(\phi) := \frac{\exp\left(2 \Re \sum_{j=1}^2 \frac{f(p) e^{2\pi i \phi}}{p^{\frac{1}{2} + it_j}} \OurEpsilon_{y, u_j}(p)\right)}{\EE\left[\exp\left(2 \Re \sum_{j=1}^2 \frac{f(p)\alpha(p)}{p^{\frac{1}{2} + it_j}} \OurEpsilon_{y, u_j}(p)\right)\right]}
    \qquad \forall \phi \in [0, 1]. \]
We would now like to find a coupling of the phases $\left(\phi(p), \phi_y^{\mathbf{u}, \mathbf{t}}(p)\right)$ such that the $r$-Wasserstein distance $\EE\left[|\phi(p) - \phi_y^{\mathbf{u}, \mathbf{t}}(p)|^r\right]^{1/r}$
is minimised.\footnote{We are using the Euclidean metric when measuring $|\phi(p) - \phi_y^{\mathbf{u}, \mathbf{t}}(p)|$. This may not be optimal in the sense that it ignores the fact that the end points of $[0,1]$ are identified (because we are ultimately interested in coupling $(\alpha(p), \alpha_y^{\mathbf{u}, \mathbf{t}}(p))$, but the estimates under the Euclidean metric are simpler to obtain and sufficient for our purpose).} It is well-known (see e.g.~\cite[Section 2.2]{San2015}) that the minimum may be attained simultaneously for all $r \ge 1$ using monotone coupling, i.e.~we set $\phi_y^{\mathbf{u}, \mathbf{t}}(p) := (D_y^{\mathbf{u}, \mathbf{t}})^{-1}(\phi(p))$ where $D_y^{\mathbf{u}, \mathbf{t}}(x) := \int_0^x d_y^{\mathbf{u}, \mathbf{t}}(x_0) dx_0$ is the cumulative distribution function of the random variable $\phi_y^{\mathbf{u}, \mathbf{t}}(p)$, and $(D_y^{\mathbf{u}, \mathbf{t}})^{-1}$ is its inverse.
\begin{lem}\label{lem:couple-estimate}
Let  $\Delta_y^{\mathbf{u}, \mathbf{t}}(p) := \alpha_y^{\mathbf{u}, \mathbf{t}}(p) - \alpha$ and $\widetilde{\Delta}_y^{\mathbf{u}, \mathbf{t}}(p) = \Delta_y^{\mathbf{u}, \mathbf{t}}(p) - \EE\left[\Delta_y^{\mathbf{u}, \mathbf{t}}(p)\right]$, where $\left(\alpha_y^{\mathbf{u}, \mathbf{t}}\right)_{\mathbf{t} \in I \times I}$ and $\alpha(p)$ are jointly defined on the same probability space using monotone coupling of phases. There exists some universal constant $C>0$ such that
\begin{align}
\label{eq:coupling-estimate1}
\left|\Delta_y^{\mathbf{u}, \mathbf{t}}(p) - \Delta_y^{\mathbf{u}, \mathbf{t}'}(p) \right|
& \le \frac{C (u_1+u_2)}{\log y} \frac{|f(p)| \log^{2} p}{p^{\frac{1}{2}}} |\mathbf{t} - \mathbf{t}'|,\\
\label{eq:coupling-estimate2}
\left|\widetilde{\Delta}_y^{\mathbf{u}, \mathbf{t}}(p) - \widetilde{\Delta}_y^{\mathbf{u}, \mathbf{t}'}(p) \right|
&\le \frac{C (u_1+u_2)}{\log y} \frac{|f(p)| \log^{2} p}{p^{\frac{1}{2}}} |\mathbf{t} - \mathbf{t}'|, \\
\label{eq:coupling-estimate3}
\left|\Delta_y^{\mathbf{u}, \mathbf{t}}(p) - \Delta_y^{\mathbf{u}', \mathbf{t}}(p) \right|
& \le \frac{C}{\log y} \frac{|f(p)| \log p}{p^{\frac{1}{2}}} |\mathbf{u} - \mathbf{u}'|,\\
\label{eq:coupling-estimate4}
\text{and} \qquad 
\left|\widetilde{\Delta}_y^{\mathbf{u}, \mathbf{t}}(p) - \widetilde{\Delta}_y^{\mathbf{u}', \mathbf{t}}(p) \right|
&\le \frac{C}{\log y} \frac{|f(p)| \log p}{p^{\frac{1}{2}}} |\mathbf{u} - \mathbf{u}'|.
\end{align}
\end{lem}
\Cref{lem:couple-estimate} is a statement that does not involve any randomness (in particular, the universal constant $C$ is deterministic) because we can determine the value of $\Delta_y^{\mathbf{u}, \mathbf{t}}(p)$ as soon as the phase $\phi(p)$ of $\alpha(p)$ is given.
\begin{proof}
Note that \eqref{eq:coupling-estimate2} follows from \eqref{eq:coupling-estimate1} (up to a different choice of absolute constant $C$) since
\[ \left|\widetilde{\Delta}_y^{\mathbf{u}, \mathbf{t}}(p) - \widetilde{\Delta}_y^{\mathbf{u}, \mathbf{t'}}(p) \right|
\le \left|{\Delta}_y^{\mathbf{u}, \mathbf{t}}(p) - {\Delta}_y^{\mathbf{u}, \mathbf{t'}}(p) \right|
+\EE\left|{\Delta}_y^{\mathbf{u}, \mathbf{t}}(p) - {\Delta}_y^{\mathbf{u}, \mathbf{t'}}(p) \right|. \]
For this reason, we shall focus on the proof of \eqref{eq:coupling-estimate1}, and consider
\begin{align*}
\left|\Delta_y^{\mathbf{u}, \mathbf{t}}(p) - \Delta_y^{\mathbf{u}, \mathbf{t'}}(p) \right|
= 
\left|\alpha_y^{\mathbf{u}, \mathbf{t}}(p) - \alpha_y^{\mathbf{u}, \mathbf{t'}}(p) \right|
& = 
\left|\exp\left(2 \pi i \phi_y^{\mathbf{u}, \mathbf{t}}(p)\right)
    - \exp\left(2 \pi i \phi_y^{\mathbf{u}, \mathbf{t}'}(p)\right)\right|\\
    & \le 2\pi | \phi_y^{\mathbf{u}, \mathbf{t}}(p) - \phi_y^{\mathbf{u}, \mathbf{t}'}(p)|
    = 2\pi\left|(D_y^{\mathbf{u}, \mathbf{t}})^{-1}(\phi(p)) - (D_y^{\mathbf{u}, \mathbf{t}'})^{-1}(\phi(p))\right|\\
    & \le 2\pi \sup_{\tilde{\mathbf{t}} \in I^2, \phi \in [0,1]} \big|\nabla_{\tilde{\mathbf{t}}}\left(D_y^{\mathbf{u}, \tilde{\mathbf{t}}}\right)^{-1}(\phi)\big| |\mathbf{t} - \mathbf{t}'|.
\end{align*}
If $F_\theta\colon [0,1] \to [0,1]$ is a bijection with continuous strictly positive derivative and smooth dependence on $\theta$, then the same holds for $F_\theta^{-1}\colon [0, 1] \to [0,1]$ by the implicit function theorem and we have
\[    \forall x \in [0,1], \qquad \frac{\partial}{\partial\theta}F_\theta(F_\theta^{-1}(x)) = \frac{\partial x}{\partial\theta} = 0
    \qquad \Rightarrow \qquad \frac{\partial (F_\theta)^{-1}(x)}{\partial \theta} = -\frac{\frac{\partial F_\theta}{\partial \theta}(F_\theta^{-1}(x))}{\frac{\partial F_\theta}{\partial x}(F_\theta^{-1}(x))}. \]
Specialising this to $D_y^{\mathbf{u}, \tilde{\mathbf{t}}}$, it is not difficult to check that
\begin{align*}
    \sup_{x \in [0,1]} \left|\frac{\partial}{\partial \tilde{t}_j} \left(D_y^{\mathbf{u}, \tilde{\mathbf{t}}}\right)^{-1}(x)\right|
    &=\sup_{x \in [0, 1]} \left|\dfrac{\int_0^{\left(D_y^{\mathbf{u}, \tilde{\mathbf{t}}}\right)^{-1}(x)} \partial_{\tilde{t}_j} d_y^{\mathbf{u}, \tilde{\mathbf{t}}}(\phi) d\phi}{d_y^{\mathbf{u}, \tilde{\mathbf{t}}}\left(\left(D_y^{\mathbf{u}, \tilde{\mathbf{t}}}\right)^{-1}(x)\right)}\right|\\
    & \le \sup_{\phi, \phi' \in [0,1]} \left|\frac{\partial_{\tilde{t}_j} d_y^{\mathbf{u}, \tilde{\mathbf{t}}}(\phi)}{d_y^{\mathbf{u}, \tilde{\mathbf{t}}}(\phi')}\right|
    \ll \frac{|f(p)|}{p^{\frac{1}{2}}} |\OurEpsilon_{y, u_j}(p)| \log p
    \ll \frac{|f(p)|}{p^{\frac{1}{2}}} \frac{u_j \log^2 p}{\log y}
\end{align*}
uniformly in $\mathbf{\tilde{t}} \in I \times I$, $y \ge 2$ and $p$. This implies \eqref{eq:coupling-estimate1}.

Similarly, \eqref{eq:coupling-estimate4} can be deduced from \eqref{eq:coupling-estimate3}, and the same method allows us to show that
\[
\left|\Delta_y^{\mathbf{u}, \mathbf{t}}(p) - \Delta_y^{\mathbf{u}', \mathbf{t}}(p) \right| 
\ll \left[\sum_{j=1}^2\sup_{\phi, \phi' \in [0,1], \tilde{u}_1, \tilde{u}_2 \ge 0} \left|\frac{\partial_{\tilde{u}_j} d_y^{\tilde{\mathbf{u}}, \mathbf{t}}(\phi)}{d_y^{\tilde{\mathbf{u}}, \mathbf{t}}(\phi')}\right|\right] |\mathbf{u} - \mathbf{u}'|
\ll \frac{|f(p)|}{p^{\frac{1}{2}}} \frac{\log p}{\log y}|\mathbf{u} - \mathbf{u}'|
\]
with the desired uniformity. This concludes the proof.
\end{proof}

\subsection{Regularity estimates for the residual field}
Now that we have constructed the coupling, we would like to show that the `residual field'
\[
\widetilde{\Ga}^{\mathbf{u}, \mathbf{t}}_{y, 0}(t_0) := \Ga_y\left(t_0;\widetilde{\Delta}_y^{\mathbf{u}, \mathbf{t}}\right) 
= 2\Re \sum_{p \le y} \frac{f(p) \widetilde{\Delta}_y^{\mathbf{u}, \mathbf{t}}(p)}{p^{\frac{1}{2} + it_0}}
\]
is negligible as $y \to \infty$. We commence with the two-point estimate below, which essentially shows that $\widetilde{\Ga}^{\mathbf{u}, \mathbf{t}}_{y, 0}(t_0)$ varies continuously in the variables $(t_0, \mathbf{t})$ with a decaying rate of change as $y \to \infty$.
\begin{lem}\label{lem:residual-field-bound}
Let $\vec{\mathbf{t}}:= (t_0, \mathbf{t}), \vec{\mathbf{t}}':= (t_0', \mathbf{t}')$ any two $3$-tuples in $I^3$ and write $d(\vec{\mathbf{t}}, \vec{\mathbf{t}}') := \sqrt{\sum_{j=0}^2 |t_j - t_j'|^2}$ for the Euclidean metric. There exists some constant $C_{f} \in (0, \infty)$ depending only on $f$ such that for any $x \ge 0$,
\begin{equation}\label{eq:field-Bernstein}
\PP\left(\left|\widetilde{\Ga}^{\mathbf{u}, \mathbf{t}}_{y, 0}(t_0) -
\widetilde{\Ga}^{\mathbf{u}, \mathbf{t}'}_{y, 0}(t_0')\right| > x\right) \le 4 \exp\left\{- \min\left(\dfrac{x^2}{\left[ \dfrac{C_{f} (u_1 + u_2)}{\log y} d(\vec{\mathbf{t}}, \vec{\mathbf{t}}')\right]^2}, 
\dfrac{x}{\dfrac{C_{f} (u_1 + u_2)}{\log y} d(\vec{\mathbf{t}}, \vec{\mathbf{t}}')}
\right)\right\}.
\end{equation}
\end{lem}
\begin{proof}
Consider
\[
\PP\left(\left|\widetilde{\Ga}^{\mathbf{u}, \mathbf{t}}_{y, 0}(t_0) -
\widetilde{\Ga}^{\mathbf{u}, \mathbf{t}'}_{y, 0}(t_0')\right| > x\right)
\le 
\PP\left(\left|\widetilde{\Ga}^{\mathbf{u}, \mathbf{t}}_{y, 0}(t_0) -
\widetilde{\Ga}^{\mathbf{u}, \mathbf{t}'}_{y, 0}(t_0)\right| > \frac{x}{2}\right)
+ 
\PP\left(\left|\widetilde{\Ga}^{\mathbf{u}, \mathbf{t}'}_{y, 0}(t_0) -
\widetilde{\Ga}^{\mathbf{u}, \mathbf{t}'}_{y, 0}(t_0')\right| > \frac{x}{2}\right)
=: \mathrm{I} + \mathrm{II}.
\]
Let us focus on the first term involving 
\[
\widetilde{\Ga}^{\mathbf{u}, \mathbf{t}}_{y, 0}(t_0) -
\widetilde{\Ga}^{\mathbf{u}, \mathbf{t}'}_{y, 0}(t_0)
= \Ga_y\left(t_0;\widetilde{\Delta}_y^{\mathbf{u}, \mathbf{t}} -\widetilde{\Delta}_y^{\mathbf{u}, \mathbf{t}'}\right)
=2\Re \sum_{p \le y} \frac{f(p) \left(\widetilde{\Delta}_y^{\mathbf{u}, \mathbf{t}}(p) -\widetilde{\Delta}_y^{\mathbf{u}, \mathbf{t}'}(p)\right)}{p^{\frac{1}{2} + it_0}}. \]
Using \Cref{lem:couple-estimate}, we have for any $r \ge 2$ that
\begin{align*}
&\sum_{p \le y} \EE\left[\left|2\Re \frac{f(p) \left(\widetilde{\Delta}_y^{\mathbf{u}, \mathbf{t}}(p) -\widetilde{\Delta}_y^{\mathbf{u}, \mathbf{t}'}(p)\right)}{p^{\frac{1}{2} + it_0}}\right|^r \right]
\le \sum_{p \le y} 
\left(\frac{2|f(p)|}{p^{\frac{1}{2}}}\right)^r
\EE\left[\left| \widetilde{\Delta}_y^{\mathbf{u}, \mathbf{t}}(p) -\widetilde{\Delta}_y^{\mathbf{u}, \mathbf{t}'}(p)\right|^r \right]
\\
&\quad \le \sum_{p \le y} 
\left(\frac{2|f(p)|}{p^{\frac{1}{2}}}\right)^r \left[\frac{C (u_1+u_2)}{\log y} \frac{|f(p)| \log^{2} p}{p^{\frac{1}{2}}} |\mathbf{t} - \mathbf{t}'|\right]^r\\
 & \quad \le \left[\frac{C'(u_1 + u_2)|\mathbf{t} - \mathbf{t}'|}{\log y}\right]^r \left[ \sup_{p } \frac{|f(p)|^2 \log^2 p}{p}\right]^{r-2}
 \qquad \text{with} \qquad C' := 2C \sqrt{1 + \sum_{p} \frac{|f(p)|^4 \log^4 p}{p^2}}.
\end{align*}
By \Cref{thm:Berstein}, we have
\begin{align}
\notag
\mathrm{I}&:= \PP\left(\left|\widetilde{\Ga}^{\mathbf{u}, \mathbf{t}}_{y, 0}(t_0) -
\widetilde{\Ga}^{\mathbf{u}, \mathbf{t}'}_{y, 0}(t_0)\right| > \frac{x}{2}\right)\\
\label{eq:field-Bernstein1}
& \le 2 \exp\left(- \dfrac{x^2}{8\left(\dfrac{C'(u_1 + u_2)|\mathbf{t} - \mathbf{t}'|}{\log y}\right)^2 + 4\left(\dfrac{C'(u_1 + u_2)|\mathbf{t} - \mathbf{t}'|}{\log y}\right)\left( \sup_{p } \dfrac{|f(p)|^2 \log^2 p}{p}\right)x } \right).
\end{align}
Let us look at the second term. Recall that for $\mathbf{u}' := (0, 0)$,
\[    \Delta_{y}^{\mathbf{u}', \mathbf{t}}(p) := \alpha_y^{\mathbf{u}', \mathbf{t}(p)} - \alpha(p) = 0
    \qquad \text{and} \qquad 
    \widetilde{\Delta}_{y}^{\mathbf{u}', \mathbf{t}}(p) = \Delta_{y}^{\mathbf{u}', \mathbf{t}}(p) - \EE[\Delta_{y}^{\mathbf{u}', \mathbf{t}}(p)] = 0. \]
By \eqref{eq:coupling-estimate4} from \Cref{lem:couple-estimate}, this means
\begin{equation}\label{eq:residual-ptwise}
\left|\widetilde{\Delta}_y^{\mathbf{u}, \mathbf{t}}(p)\right| \le \frac{C(u_1 + u_2)}{\log y} \frac{|f(p)| \log p}{p^{\frac{1}{2}}}.
\end{equation}
In particular, for any $r \ge 2$ we have
\begin{align*}
&\sum_{p \le y}\EE\left[\left|2\Re\frac{f(p) \widetilde{\Delta}_y^{\mathbf{u}, \mathbf{t}'}(p)}{p^{\frac{1}{2}}} \left(p^{-it_0} - p^{-it_0'}\right)\right|^r\right]
 \le \sum_{p \le y}\left(\frac{2|t_0 - t_0'| |f(p)|\log p}{p^{\frac{1}{2}} } \right)^r\EE\left[\left|\widetilde{\Delta}_y^{\mathbf{u}, \mathbf{t}'}(p)\right|^r\right]\\
& \qquad \le \sum_{p \le y} \left(\frac{2|t_0 - t_0'| |f(p)|\log p}{p^{\frac{1}{2}} } \right)^r \left[\frac{C(u_1 + u_2)}{\log y} \frac{|f(p)| \log p}{p^{\frac{1}{2}}}\right]^r\\
& \qquad \le \left(\frac{C'(u_1+u_2)|t_0 - t_0'|}{\log y} \right)^r 
\left[\sup_{p} \frac{|f(p)|^2 \log^2 p}{p}\right]^{r-2}.
\end{align*}
By another application of \Cref{thm:Berstein}, we have
\begin{align}
\notag
\mathrm{II}&:= \PP\left(\left|\widetilde{\Ga}^{\mathbf{u}, \mathbf{t}'}_{y, 0}(t_0) -
\widetilde{\Ga}^{\mathbf{u}, \mathbf{t}'}_{y, 0}(t_0')\right| > \frac{x}{2}\right)\\
\label{eq:field-Bernstein2}
& \le 2 \exp\left(- \dfrac{x^2}{8\left(\dfrac{C'(u_1 + u_2)|t_0 - t_0'|}{\log y}\right)^2 + 4\left(\dfrac{C'(u_1 + u_2)|t_0 - t_0'|}{\log y}\right)\left( \sup_{p } \dfrac{|f(p)|^2 \log^2 p}{p}\right)x } \right).
\end{align}
The desired claim \eqref{eq:field-Bernstein} now follows by combining \eqref{eq:field-Bernstein1} and \eqref{eq:field-Bernstein2}.
\end{proof}
Let us also include another one-point decay estimate.
\begin{lem}\label{lem:residual-point}
There exists some constant $C_f \in (0, \infty)$ depending only on $f$ such that for any $x \ge 0$,
\[ \PP\left(\left|\widetilde{\Ga}^{\mathbf{u}, \mathbf{t}}_{y, 0}(t_0)\right| > x\right) \le 2 \exp\left\{- \min\left(\frac{x^2}{\left[ \dfrac{C_{f} (u_1 + u_2)}{\log y}\right]^2}, \dfrac{x}{\dfrac{C_{f} (u_1 + u_2)}{\log y}}\right) \right\}. \]
\end{lem}
\begin{proof}
Using the bound \eqref{eq:residual-ptwise}, we have
\begin{align*}
\sum_{p \le y}\EE\left[\left|2\Re\frac{f(p) \widetilde{\Delta}_y^{\mathbf{u}, \mathbf{t}'}(p)}{p^{\frac{1}{2} + it_0}} \right|^r\right]
&\le \sum_{p \le y} \left(\frac{2 |f(p)|}{p^{\frac{1}{2}}} \right)^r \left[\frac{C(u_1 + u_2)}{\log y} \frac{|f(p)| \log p}{p^{\frac{1}{2}}}\right]^r\\
&\le \left(\frac{C'(u_1+u_2)}{\log y} \right)^r 
\left[\sup_{p} \frac{|f(p)|^2 \log p}{p}\right]^{r-2}
\end{align*}
for any $r \ge 2$, and the claim follows from another application of \Cref{thm:Berstein}.
\end{proof}
We are ready to state and prove the following uniform estimate.
\begin{cor}\label{cor:residual-estimate}
There exists some constant $C_f \in (0, \infty)$ depending only on $f$ such that
\begin{equation}\label{eq:residual-exp-moment}
\EE\left[  \exp \left( \lambda \sup_{t_0 \in I, \mathbf{t} \in I^2}\left|\widetilde{\Ga}^{\mathbf{u}, \mathbf{t}}_{y, 0}(t_0)\right|\right)\right] 
\le \exp\left(\lambda \frac{C_f(u_1 + u_2)}{\log y}\right) + \frac{\lambda}{\left[C_f(u_1 + u_2) / \log y\right]^{-1} - \lambda}
\end{equation}
for any $0 \le \lambda < \left[\dfrac{C_f(u_1 + u_2)}{\log y}\right]^{-1}$. In particular, 
\begin{equation}\label{eq:residual-exp-uniform}
\sup_{y \ge 2} \EE\left[  \exp \left( \lambda \sup_{t_0 \in I, \mathbf{t} \in I^2}\left|\widetilde{\Ga}^{\mathbf{u}, \mathbf{t}}_{y, 0}(t_0)\right|\right)\right] < \infty
\end{equation}
for any fixed $\lambda \ge 0$, and $\sup_{t_0 \in I, \mathbf{t} \in I^2} \left|\widetilde{\Ga}^{\mathbf{u}, \mathbf{t}}_{y, 0}(t_0)\right| \xrightarrow[y \to \infty]{p} 0$.
\end{cor}
\begin{proof}
As before, let us write $\vec{\mathbf{t}}:= (t_0, \mathbf{t})$ and $\vec{\mathbf{t}}':= (t_0', \mathbf{t}')$. By H\"older's inequality,
\[\EE\left[  \exp \left( \lambda \sup_{\vec{\mathbf{t}} \in I^3}\left|\widetilde{\Ga}^{\mathbf{u}, \mathbf{t}}_{y, 0}(t_0)\right|\right)\right]
\le 
\EE\left[  \exp \left( 2\lambda \left|\widetilde{\Ga}^{\mathbf{u}, \mathbf{t}'}_{y, 0}(t_0)'\right|\right)\right]^{\frac{1}{2}}
\EE\left[  \exp \left( 2\lambda \sup_{\vec{\mathbf{t}}, \vec{\mathbf{t}}' \in I^3}\left|\widetilde{\Ga}^{\mathbf{u}, \mathbf{t}}_{y, 0}(t_0)
- \widetilde{\Ga}^{\mathbf{u}, \mathbf{t}'}_{y, 0}(t_0')\right|\right)\right]^{\frac{1}{2}},
\]
and we need to estimate each expectation on the right-hand side. For this, consider a random variable $X \ge 0$ such that
\[ \PP(X > x) \le C_1 \exp \left\{-\frac{1}{C_2} \min\left(\frac{x^2}{A^2}, \frac{x}{A}\right)\right\} \qquad \forall x \ge 0 \]
for some $A, C_1, C_2 > 0$. Then for any $\lambda \in [0, A^{-1})$, we have
\begin{align*}
\EE\left[e^{\lambda X}\right]
 \le e^{\lambda A} + \EE\left[\left(e^{\lambda X} - e^{\lambda A}\right) \mathbf{1}_{\{X > A\}}\right]
&= e^{\lambda A} + \lambda \EE\left[\mathbf{1}_{\{X > A\}}\int_A^{\infty} e^{\lambda x} \mathbf{1}_{\{x < X\}} dx \right]\\
& \le e^{\lambda A} + \lambda \int_A^\infty e^{\lambda x} C_1\exp\left(-\frac{x}{C_2A}\right)dx
\le e^{\lambda A} + \frac{C_1\lambda}{(C_2A)^{-1} - \lambda}.
\end{align*}
Let us apply this inequality to $\left|\widetilde{\Ga}^{\mathbf{u}, \mathbf{t}'}_{y, 0}(t_0')\right|$. Using \Cref{lem:residual-point}, we immediately obtain
\[ \EE\left[  \exp \left( 2\lambda \left|\widetilde{\Ga}^{\mathbf{u}, \mathbf{t}'}_{y, 0}(t_0')\right|\right)\right]
\le \exp\left(2\lambda \frac{C_f(u_1 + u_2)}{\log y}\right) + \frac{4\lambda}{\left[C_f(u_1 + u_2) / \log y\right]^{-1} - 2\lambda}. \]
As for $\sup_{\vec{\mathbf{t}}, \vec{\mathbf{t}}' \in I^3}\left|\widetilde{\Ga}^{\mathbf{u}, \mathbf{t}}_{y, 0}(t_0)
- \widetilde{\Ga}^{\mathbf{u}, \mathbf{t}'}_{y, 0}(t_0')\right|$, recall the tail probability in \Cref{lem:residual-field-bound}. In the notation of \Cref{thm:chaining}, if we define $d_1(\vec{\mathbf{t}}, \vec{\mathbf{t}}') = d_2(\vec{\mathbf{t}}, \vec{\mathbf{t}}') : =\frac{C_f(u_1 + u_2)}{\log y} d(\vec{\mathbf{t}}, \vec{\mathbf{t}}')$, then
\[    \PP\left(\sup_{\vec{\mathbf{t}}, \vec{\mathbf{t}}' \in I^3}\left|\widetilde{\Ga}^{\mathbf{u}, \mathbf{t}}_{y, 0}(t_0)
- \widetilde{\Ga}^{\mathbf{u}, \mathbf{t}'}_{y, 0}(t_0')\right| > x \right)
\le 20 \exp\left(- \frac{x^2}{16\gamma_2(I^3, d_2)^2 + 8\gamma_1(I^3, d_1)x}\right)  \qquad \forall x \ge 0 \]
where $\gamma_k(I^3, d_k) \ll C_f(u_1 + u_2)/\log y$ for both $k=1$ and $2$ by \Cref{lem:metric-entropy-estimate}. Therefore, there exists some constant $C > 0$ such that
\[ \EE\left[  \exp \left( 2\lambda \sup_{\vec{\mathbf{t}}, \vec{\mathbf{t}}' \in I^3}\left|\widetilde{\Ga}^{\mathbf{u}, \mathbf{t}}_{y, 0}(t_0)
- \widetilde{\Ga}^{\mathbf{u}, \mathbf{t}'}_{y, 0}(t_0')\right|\right)\right]
\le \exp\left(2\lambda \frac{C_f(u_1 + u_2)}{\log y}\right) + \frac{C \cdot 2\lambda}{\left[\dfrac{CC_f(u_1 + u_2)}{\log y}\right]^{-1} - 2\lambda} \]
and this gives \eqref{eq:residual-exp-moment} up to a different choice of the constant $C_f$.

Now, for any fixed $\lambda \ge 0$, \eqref{eq:residual-exp-moment} is uniformly bounded for all $y$ sufficiently large. As for finite values of $y$, the field $\widetilde{\Ga}^{\mathbf{u}, \mathbf{t}}_{y, 0}(t_0)$ is also trivially bounded because it is given by a finite sum of bounded random variables. This verifies the claim \eqref{eq:residual-exp-uniform}. Finally, from \eqref{eq:residual-exp-moment} we observe that
\[
1
\le 
\liminf_{y \to \infty}\EE\left[  \exp \left( \lambda \sup_{\vec{\mathbf{t}} \in I^3}\left|\widetilde{\Ga}^{\mathbf{u}, \mathbf{t}}_{y, 0}(t_0)\right|\right)\right]
\le \limsup_{y \to \infty}\EE\left[  \exp \left( \lambda \sup_{\vec{\mathbf{t}} \in I^3}\left|\widetilde{\Ga}^{\mathbf{u}, \mathbf{t}}_{y, 0}(t_0)\right|\right)\right]
\le 1 \qquad \forall \lambda \ge 0.\]
This implies $\sup_{\vec{\mathbf{t}} \in I^3}\left|\widetilde{\Ga}^{\mathbf{u}, \mathbf{t}}_{y, 0}(t_0)\right|$ converges in distribution, and hence in probability, to $0$ as $y \to \infty$.
\end{proof}
\subsection{Elementary estimates}
In this subsection we collect a few estimates for moment generating functions that are needed to perform the argument of change of measure. The following lemma will be repeatedly used.
\begin{lem}[{\cite[Lemma 2.5]{GWL1}}]\label{lem:perturb-mgf}
Let $X$ be an $\RR^d$-valued random variable satisfying $\EE[X] = 0$ and $\PP(|X| \le r) = 1$ for some fixed $r > 0$. Then
\begin{align*}
\frac{\EE\left[\langle a_1, X \rangle e^{\langle a_2, X \rangle}\right]}{\EE[e^{\langle a_2, X\rangle}]} 
&= a_1^T \EE[XX^T]a_2 + O( |a_1||a_2|^2)\\
\text{and} \qquad \frac{\EE\left[e^{\langle a_1 + a_2, X \rangle}\right]}{\EE[e^{\langle a_2, X\rangle}]} 
&= \exp \left(\frac{1}{2}a_1^T \EE\left[ X X^T \right] (a_1 +2a_2) +O( |a_1||a_2|^2 + |a_1|^3)\right)
\end{align*}
uniformly for $a_1, a_2$ in compact subsets of $\RR^d$.
\end{lem}
The first estimate explains how much $M_y(u) := \EE[\exp \left(\Ga_{y, u}(t; \alpha)\right)]$ differs from $M_y(0)$.
\begin{lem}\label{lem:mgf-estimate1}
Let $u \ge 0$ be fixed. As $y \to \infty$, we have
\[   \frac{M_y(u)}{M_y(0)} = \exp\left\{ \sum_{p \le y} \frac{|f(p)|^2}{p}\left[ \OurEpsilon_{y, u}(p)^2 + 2\OurEpsilon_{y, u}(p)\right] + O\left(\sum_{p \le y} \frac{|f(p)|^3}{p^{\frac{3}{2}}} |\OurEpsilon_{y, u}(p)|\right)\right\} \]
where the error is uniform for $u$ in compact subsets of $\RR_{\ge 0}$.
\end{lem}
\begin{proof}
By rotational invariance of $\alpha$ we may pretend that $t =0$ and $f(p) = |f(p)|$. If we set
\[    X(p) = \begin{pmatrix} \Re \alpha(p) \\ \Im \alpha(p) \end{pmatrix}, 
    \qquad a_1(p) = \begin{pmatrix} \frac{2|f(p)|}{p^{\frac{1}{2}}}\OurEpsilon_{y, u}(p)\\ 0 \end{pmatrix}
    \qquad \text{and} \qquad a_2(p) = \begin{pmatrix} \frac{2|f(p)|}{p^{\frac{1}{2}}}\\ 0 \end{pmatrix}, \]
then $M_y(u)/M_y(0) = \prod_{p \le y} \EE\left[e^{\langle a_1(p) + a_2(p), X(p) \rangle}\right]/\EE[e^{\langle a_2(p), X(p)\rangle}]$ and the result follows from \Cref{lem:perturb-mgf}.
\end{proof}
The next estimate controls an expectation that appears in the change of measure $\PP \leftrightarrow \PP_y^{\mathbf{u}, \mathbf{t}}$.
\begin{lem}\label{lem:mgf-estimate2}
Let $\mathbf{u} := (u_1, u_2) \in \RR_{\ge 0}^2$, $\mathbf{t} := (t_1, t_2) \in I^2$. As $y \to \infty$ we have
\begin{equation}\label{eq:asymptotic-E_y}
\begin{split}
E_y(\mathbf{u}, \mathbf{t})
& =\exp\Bigg\{\sum_{p \le y} \frac{|f(p)|^2}{p} \left[2 \OurEpsilon_{y, u_1}(p)\OurEpsilon_{y, u_2}(p)\cos(|t_1 - t_2| \log p) + \sum_{j=1}^2 \OurEpsilon_{y, u_j}(p)^2\right] \\
& \qquad \qquad \qquad + O\left(\sum_{j=1}^2\sum_{p \le y} \frac{|f(p)|^3}{p^{\frac{3}{2}}}|\OurEpsilon_{y, u_j}(p)|^3\right)\Bigg\}
\end{split}
\end{equation}
where the error is uniform for $\mathbf{u}$ in compact subsets of $\RR_{\ge 0}^2$ and $\mathbf{t} \in I^2$.
\end{lem}
\begin{proof}
If we set
\[    X(p) = \begin{pmatrix} \Re \alpha(p) \\ \Im \alpha(p) \end{pmatrix}, 
    \qquad a_1(p) = \frac{2}{p^{\frac{1}{2}}} \sum_{j=1}^2 \OurEpsilon_{y, u_j}(p)\begin{pmatrix} \Re \frac{f(p)}{p^{it_j}} \\ -\Im \frac{f(p)}{p^{it_j}}\end{pmatrix}
    \qquad \text{and} \qquad a_2(p) = \begin{pmatrix} 0\\ 0 \end{pmatrix}, \]
then $\sum_{j=1}^2 \Ga_{y, 0}(t_j; \OurEpsilon_{y, u_j}\alpha) = \sum_{p \le y} \langle a_1(p), X(p)\rangle$, and by \Cref{lem:perturb-mgf} we have
\[E_y(\mathbf{u}, \mathbf{t}) = 
\prod_{p \le y} \frac{\EE\left[e^{\langle a_1(p) + a_2(p), X(p)\rangle}\right]}{\EE[e^{\langle a_2(p), X(p)\rangle}]}
= \exp\left\{
\sum_{p \le y} \left[\frac{1}{4} |a_1(p)|^2 + O\left(|a_1(p)|^3\right) \right]
\right\} \]
which implies \eqref{eq:asymptotic-E_y}.
\end{proof}
The next result quantifies the effect of shift of mean induced by the measure $\PP_y^{\mathbf{u}, \mathbf{t}}$.
\begin{lem}\label{lem:mean-estimate}
Let $t_0 \in I$ and $\mathbf{t} := (t_1, t_2) \in I^2$. We have
\[   \Ga_{y, 0}\left(t_0; \EE\left[\alpha_y^{\mathbf{u}, \mathbf{t}}\right]\right)
    =\sum_{j=1}^2\sum_{p \le y} \frac{2|f(p)|^2}{p} \cos\left(|t_0 - t_j| \log p\right) \OurEpsilon_{y, u_j}(p) + O\left(\sum_{j=1}^2 \sum_{p \le y} \frac{|f(p)|^3}{p^{\frac{3}{2}}} \OurEpsilon_{y, u_j}(p)^2\right) \]
where the error is uniform for $\mathbf{u}$ in compact subsets of $\RR_{\ge 0}^2$ and also $t_0, t_1, t_2 \in I$.
\end{lem}
\begin{proof}
Observe that
\[ \Ga_{y, 0}\left(t_0; \EE\left[\alpha_y^{\mathbf{u}, \mathbf{t}}\right]\right)
= \Ga_{y, 0}\left(t_0; \EE_y^{\mathbf{u}, \mathbf{t}}\left[\alpha\right]\right)
= \EE_y^{\mathbf{u}, \mathbf{t}}\left[\Ga_{y, 0}\left(t_0; \alpha\right)\right]
= \sum_{p \le y} \frac{\EE\left[\langle a_1(p), X(p) \rangle e^{\langle a_2(p), X(p) \rangle}\right]}{\EE[e^{\langle a_2(p), X(p)\rangle}]}  \]
where
\[ X(p) = \begin{pmatrix} \Re \alpha(p) \\ \Im \alpha(p) \end{pmatrix}, 
\qquad a_1(p) = \frac{2}{p^{\frac{1}{2}}} \begin{pmatrix} \Re \frac{f(p)}{p^{it_0}} \\ -\Im \frac{f(p)}{p^{it_0}}\end{pmatrix}
\qquad \text{and} \qquad a_2(p) = \frac{2}{p^{\frac{1}{2}}} \sum_{j=1}^2 \OurEpsilon_{y, u_j}(p)\begin{pmatrix} \Re \frac{f(p)}{p^{it_j}} \\ -\Im \frac{f(p)}{p^{it_j}}\end{pmatrix}. \]
We have
\[    a_1(p)^T \EE[X(p)X(p)^T] a_2(p) = \frac{2}{p} \sum_{j=1}^2 \OurEpsilon_{y, u_j}(p) \Re \left[ \frac{f(p)}{p^{it_0}}\overline{\left(\frac{f(p)}{p^{it_j}}\right)} \right]
    = \frac{2|f(p)|^2}{p} \sum_{j=1}^2 \cos(|t_0 - t_j|\log p) \OurEpsilon_{y, u_j}(p), \]
and it is also easy to check that $O(|a_1(p)||a_2(p)|^2)$ has the desired bound.
\end{proof}
Finally, the following estimate controls the oscillatory terms that appear in all previous lemmas.
\begin{lem}\label{lem:ingredient-bound}
Write $\displaystyle C_y(h; v) := \sum_{p \le y} \frac{|f(p)|^2}{p}v(p)\cos(|h| \log p)$. For any fixed $u_1, u_2 \ge 0$, we have
\[ C_y(h; \OurEpsilon_{y, u_1}) \ll 1 \qquad \text{and}\qquad  C_y(h; \OurEpsilon_{y, u_2}\OurEpsilon_{y, u_1}) \ll 1 \]
uniformly for $h$ in compact subsets of $\RR$. Moreover,
\[    C_y(h; \OurEpsilon_{y, u_1}) \xrightarrow[y \to \infty]{} 0 \qquad \text{and} \qquad 
    C_y(h; \OurEpsilon_{y, u_1}\OurEpsilon_{y, u_2}) \xrightarrow[y \to \infty]{} 0 \]
uniformly for $h$ in compact subsets of $\RR \setminus\{0\}$.
\end{lem}
\begin{proof}
Let us write
\[    C_y(h; \OurEpsilon_{y, u_1})
    = \int_{2^-}^{y^+} \frac{\OurEpsilon_{y, u_1}(p) \cos(|h| \log p)}{p} d \pi_{|f|^2}(p)
    = \int_{2^-}^{y^+} \frac{\OurEpsilon_{y, u_1}(p) \cos(|h| \log p)}{p} \left[\frac{dp}{\log p} + d \Ea_{|f|^2}(p)\right]
    = \mathrm{I} + \mathrm{II}. \]
For the main term, we have
\[ I  = \int_{2}^{y} \frac{\OurEpsilon_{y, u_1}(p) \cos(|h| \log p)}{p} \frac{dp}{\log p}
= \int_{\frac{\log 2}{\log y}}^1 \left[\exp\left(-\frac{u_1}{2} x\right) - 1\right] \cos(|h| x \log y) \frac{dx}{x}\]
and integration by parts easily suggests that $|\mathrm{I}| \ll (h \log y)^{-1} $.

As for the error term, we have
\begin{align*}
\mathrm{II} 
& =  \left[\frac{\OurEpsilon_{y, u_1}(p) \cos(|h| \log p)}{p} \Ea_{|f|^2}(p)\right]_{2^-}^{y^+}
- \int_{2}^{y} \left[\frac{\partial}{\partial p} \frac{\OurEpsilon_{y, u_1}(p) \cos(|h| \log p)}{p}\right] \Ea_{|f|^2} (p) dp\\
& = O\left(\frac{\Ea_{|f|^2}(y)}{y} + | \OurEpsilon_{y, u_1}(p)|\right)
+ \int_2^y \OurEpsilon_{y, u_1}(p) \left[\frac{u_1}{2 \log y} \cos(|h| \log p) + |h|\sin(|h| \log p) + 1\right]\Ea_{|f|^2}(p) \frac{dp}{p^2}.
\end{align*}
The integral on the last line is uniformly bounded in absolute value by  $\displaystyle \left[ 1 + |h| +  \frac{u_1}{2 \log y}\right]\int_2^\infty p^{-2}\Ea_{|f|^2}(p)dp$ (which is finite by the assumption on $f$) and vanishes as $y \to \infty$ by dominated convergence. This verifies both claims for $C_y(h; \OurEpsilon_{y, u_1})$. The proof for $C_y(h; \OurEpsilon_{y, u_1}\OurEpsilon_{y, u_2})$ is similar.
\end{proof}

\subsection{Proof of \texorpdfstring{\Cref{lem:mod-cross-mom}}{Theorem \ref{lem:mod-cross-mom}}}\label{sec:main_lemma}
We adopt an approach based on generalised dominated convergence (or Vitali convergence theorem): to establish $\lim_{n \to \infty} \EE[X_n] = \EE[X_\infty]$, it suffices to verify that 
\[ \mathrm{(i)}\qquad (X_n)_{n\ge 1} \quad \text{are uniformly integrable} 
\qquad \qquad \text{and} \qquad  \qquad \mathrm{(ii)}\quad  X_n \xrightarrow[n \to \infty]{p} X_\infty. \]
We also recall that a sufficient criterion for uniform integrability is a uniform bound on the $r$-th absolute moments for some $r > 1$, because 
$\sup_{n \ge 1} \EE[|X_n| \mathbf{1}_{\{|X_n| > K\}}] 
\le \left(\sup_{n \ge 1} \EE[|X_n|^r]\right) /K^{r-1}\xrightarrow[K \to \infty]{} 0.$
\paragraph{Step 1: uniform integrability.} We first show that: 
\begin{lem}
The collection of random variables 
\begin{equation}\label{eq:key-integral-msmm}
\int_{I \times I} \frac{M_y(0)^2 E_y(\mathbf{u}, \mathbf{t})}{M_y(u_1)M_y(u_2)} \nu_{y,0}(dt_1; \alpha_y^{\mathbf{u}, \mathbf{t}})\nu_{y, 0}(dt_2; \alpha_y^{\mathbf{u}, \mathbf{t}}) 
\exp\left\{-L \nu_{y, 0}(I; \alpha_{y}^{\mathbf{u}, \mathbf{t}})\right\}, \qquad y \ge 3
\end{equation}
are uniformly integrable.
\end{lem}
\begin{proof}
Recall the notation $\alpha_y^{\mathbf{u}, \mathbf{t}}(p) = \alpha(p) + \EE\left[\alpha_y^{\mathbf{u}, \mathbf{t}}(p)\right] + \widetilde{\Delta}_y^{\mathbf{u}, \mathbf{t}}(p)$. By definition, we have
\[
\nu_{y, 0}(dt_0; \alpha_y^{\mathbf{u}, \mathbf{t}})
    = \exp\left\{\Ga_{y, 0}(t_0;\EE\left[\alpha_y^{\mathbf{u}, \mathbf{t}}\right] )\right\}
    \exp\left\{\widetilde{\Ga}^{\mathbf{u}, \mathbf{t}}_{y, 0}(t_0)\right\}
    \nu_{y, 0}(dt_0; \alpha).
\]
Using \Cref{lem:mgf-estimate1}, \Cref{lem:mgf-estimate2} and \Cref{lem:mean-estimate}, we see that
\begin{align}
\notag 
    & \left[\prod_{j=1}^2\frac{M_y(0)}{M_y(u_j)} \exp\left\{\Ga_{y, 0}(t_j;\EE\left[\alpha_y^{\mathbf{u}, \mathbf{t}}\right] )\right\}\right] E_y(\mathbf{u}, \mathbf{t})\\
\label{eq:extra-density}
    & \qquad = \exp\left\{2C_y\left(|t_1 - t_2|; \OurEpsilon_{y, u_1} \OurEpsilon_{y, u_2} + \OurEpsilon_{y, u_1} +\OurEpsilon_{y, u_2} \right) + O\left( \sum_{j =1}^2\sum_{p \le y} \frac{|f(p)|^3}{p^{\frac{3}{2}}}|\OurEpsilon_{y, u_j}(p)|\right) \right\}
\end{align}
where we have again used the notation
\[C_y\left(h; v\right) := \sum_{p \le y} \frac{|f(p)|^2}{p} v(p) \cos(h \log p).\]
By \Cref{lem:ingredient-bound}, there exists some constant $C>0$ such that
\begin{align*}
    \frac{M_y(0)^2 E_y(\mathbf{u}, \mathbf{t})}{M_y(u_1)M_y(u_2)} \nu_{y,0}(dt_1; \alpha_y^{\mathbf{u}, \mathbf{t}})\nu_{y, 0}(dt_2; \alpha_y^{\mathbf{u}, \mathbf{t}}) 
    &\le C
    \exp\left\{\widetilde{\Ga}^{\mathbf{u}, \mathbf{t}}_{y, 0}(t_1)+ \widetilde{\Ga}^{\mathbf{u}, \mathbf{t}}_{y, 0}(t_2)\right\}
    \nu_{y, 0}(dt_1; \alpha)
    \nu_{y, 0}(dt_2; \alpha)\\
    &\le C
    \exp\left\{2\sup_{\vec{\mathbf{t}}' \in I^3}\left|\widetilde{\Ga}^{\mathbf{u}, \mathbf{t}'}_{y, 0}(t_0')\right|\right\}
    \nu_{y, 0}(dt_1; \alpha)
    \nu_{y, 0}(dt_2; \alpha).
\end{align*}
Similarly, up to a different choice of the constant $C$, we also have
\[\nu_{y, 0}(I; \alpha_{y}^{\mathbf{u}, \mathbf{t}})
\ge C^{-1} \exp\left\{-\sup_{\vec{\mathbf{t}}' \in I^3}\left|\widetilde{\Ga}^{\mathbf{u}, \mathbf{t}'}_{y, 0}(t_0')\right|\right\} \nu_{y, 0}(I; \alpha).\]
But this means
\begin{align*}
\int_{I \times I} & \frac{M_y(0)^2 E_y(\mathbf{u}, \mathbf{t})}{M_y(u_1)M_y(u_2)} \nu_{y,0}(dt_1; \alpha_y^{\mathbf{u}, \mathbf{t}})\nu_{y, 0}(dt_2; \alpha_y^{\mathbf{u}, \mathbf{t}}) 
e^{-L \nu_{y, 0}(I; \alpha_{y}^{\mathbf{u}, \mathbf{t}})}\\
& \qquad \le 
C\exp\left\{2\sup_{\vec{\mathbf{t}}' \in I^3}\left|\widetilde{\Ga}^{\mathbf{u}, \mathbf{t}'}_{y, 0}(t_0')\right|\right\}
\nu_{y, 0}(I; \alpha)^2
\exp\left\{-L C^{-1} \exp\left\{-\sup_{\vec{\mathbf{t}}' \in I^3}\left|\widetilde{\Ga}^{\mathbf{u}, \mathbf{t}'}_{y, 0}(t_0')\right|\right\} \nu_{y, 0}(I; \alpha)\right\}\\
& \qquad \le \frac{4 C^3}{L^2} \exp\left\{4\sup_{\vec{\mathbf{t}}' \in I^3}\left|\widetilde{\Ga}^{\mathbf{u}, \mathbf{t}'}_{y, 0}(t_0')\right|\right\}
\end{align*}
where the last step follows from the elementary inequality $x^2 e^{-Ax} \le 4A^{-2}$ for all $x, A \ge 0$. The last expression has uniformly bounded moments for any fixed order by \Cref{cor:residual-estimate}, from which we deduce the desired uniform integrability.
\end{proof}
\paragraph{Step 2: convergence in probability.} It remains to show that:
\begin{lem}\label{lem:I_y-convergence}
The random variables in \eqref{eq:key-integral-msmm} converge in probability to $\nu_{\infty}(I)^2 e^{-L\nu_\infty(I)}$ as $y \to \infty$.
\end{lem}
The proof of \Cref{lem:I_y-convergence} relies on the following probabilistic squeezing lemma.
\begin{lem}\label{lem:plim-squeeze}
Let $(X_n)_{n \ge 1}$ be a sequence of real-valued random variables. Suppose for each $m \in \NN$ there exist two sequences of random variables $(X_{n, \pm}^{(m)})_{n \ge 1}$ such that $X_{n, -}^{(m)} \le X_n \le X_{n, +}^{(m)}$ and $X_{n, \pm}^{(m)} \xrightarrow[n \to \infty]{p} X_{\infty, \pm}^{(m)}$ where $X_{\infty, \pm}^{(m)}$ satisfies $X_{\infty, \pm}^{(m)} \xrightarrow[m \to \infty]{p} X_\infty$. Then $X_n \xrightarrow[n \to \infty]{p}X_\infty$.
\end{lem}
\begin{proof}
Let $\delta > 0$. We have
\begin{align*}
\PP\left(X_n - X_\infty > \delta\right)
 \le \PP\left(X_{n, +}^{(m)} - X_\infty > \delta\right)
& \le \PP\left(X_{n, +}^{(m)} - X_{\infty, +}^{(m)} > \frac{\delta}{2}\right)
+ \PP\left(X_{\infty, +}^{(m)} - X_\infty > \frac{\delta}{2}\right)\\
&\xrightarrow[n \to \infty]{}
\PP\left(X_{\infty, +}^{(m)} - X_\infty > \frac{\delta}{2}\right)
\xrightarrow[m \to \infty]{} 0.
\end{align*}
Similarly, we have
\[\PP\left(X_n - X_\infty < - \delta\right)
\le \PP\left(X_{n, -}^{(m)} - X_{\infty, -}^{(m)} < - \frac{\delta}{2}\right)
+ \PP\left(X_{\infty, -}^{(m)} - X_\infty < - \frac{\delta}{2}\right)
\xrightarrow[n \to \infty, m \to \infty]{} 0\]
which concludes the proof.
\end{proof}
\begin{proof}[Proof of \Cref{lem:I_y-convergence}]
For each $m \in \NN$, partition $I$ into $2^m$ dyadic subintervals $(I_j^{(m)})_{j = 1, \dots, 2^m}$ ordered from left to right, and write $\widetilde{I}_j^{(m)} := \bigcup_{k=j-1}^{j+1}I_{k}^{(m)}$ for the union of the three consecutive dyadic intervals centred\footnote{We are abusing the terminology here, since $I_0^{(m)}$ and $I_{2^m+1}^{(m)}$ when appearing should be interpreted as empty sets.} at $I_j^{(m)}$. We will now construct upper bounds in the sense of \Cref{lem:plim-squeeze}.

If $t_1 \in I_j^{(m)}$ and $t_2 \in I_k^{(m)}$, then
\begin{align}
\notag 
&\nu_{y, 0}(I; \alpha_{y}^{\mathbf{u}, \mathbf{t}})
\ge \int_{I \setminus (I_j^{(m)} \cup I_k^{(m)})} \exp\left\{\EE\left[\Ga_{y, 0}(t_0;\alpha_y^{\mathbf{u}, \mathbf{t}} )\right]\right\} \exp\left\{\widetilde{\Ga}^{\mathbf{u}, \mathbf{t}}_{y, 0}(t_0)\right\}\nu_{y, 0}(dt_0; \alpha)\\
\notag 
& \qquad \ge 
  \exp\left\{-\sup_{t_0' \in I \setminus (\widetilde{I}_j^{(m)} \cup \widetilde{I}_k^{(m)}), \mathbf{t}' \in I_j^{(m)} \times I_k^{(m)}}\EE\left[\Ga_{y, 0}(t_0';\alpha_y^{\mathbf{u}, \mathbf{t}'} )\right]
  -\sup_{\vec{\mathbf{t}}' \in I^3}\left|\widetilde{\Ga}^{\mathbf{u}, \mathbf{t}'}_{y, 0}(t_0')\right|\right\}
  \nu_{y, 0}\left(I \setminus (\widetilde{I}_j^{(m)} \cup \widetilde{I}_k^{(m)}); \alpha\right)\\
\label{eq:upperbound-ingredient1}
& \qquad =: \exp\left\{- R_{y, j, k}^{(m)}\right\}\nu_{y, 0}\left(I \setminus (\widetilde{I}_j^{(m)} \cup \widetilde{I}_k^{(m)}); \alpha\right).
\end{align}
Note that $\sup_{\vec{\mathbf{t}}' \in I^3}\left|\widetilde{\Ga}^{\mathbf{u}, \mathbf{t}'}_{y, 0}(t_0')\right| \xrightarrow[y \to \infty]{p} 0$ by \Cref{lem:residual-field-bound}, and
\[    \sup_{t_0' \in I \setminus (\widetilde{I}_j^{(m)} \cup \widetilde{I}_k^{(m)}), \mathbf{t}' \in I_j^{(m)} \times I_k^{(m)}}\EE\left[\Ga_{y, 0}(t_0';\alpha_y^{\mathbf{u}, \mathbf{t}'} )\right]
    \xrightarrow[y \to \infty]{} 0 \]
by \Cref{lem:ingredient-bound} since $|t_0' - t_1'| \ge 2^{-m}|I|$ and $|t_0' - t_2'| \ge 2^{-m}|I|$ under the condition for evaluating the supremum. This shows that the lower bound in \eqref{eq:upperbound-ingredient1} converges in probability to $\nu_\infty\left(I \setminus (\widetilde{I}_j^{(m)} \cup \widetilde{I}_k^{(m)})\right)$ as $y \to \infty$.

Next we inspect \eqref{eq:extra-density}. By dominated convergence, we have $\sum_{j =1}^2\sum_{p \le y} \frac{|f(p)|^3}{p^{\frac{3}{2}}}|\OurEpsilon_{y, u_j}(p)| \to 0$ as $y \to \infty$. Combined with \Cref{lem:ingredient-bound}, we see that 
\[ D_y(\mathbf{u}, \mathbf{t}) := \left[\prod_{j=1}^2\frac{M_y(0)}{M_y(u_j)} \exp\left\{\EE\left[\Ga_{y, 0}(t_j;\alpha_y^{\mathbf{u}, \mathbf{t}} )\right]\right\}\right] E_y(\mathbf{u}, \mathbf{t})  \]
is uniformly bounded in $y \ge 2$ and $\mathbf{t} \in I^2$ by some constant $C \ge 1$, and in addition $D_y(\mathbf{u}, \mathbf{t}) \to 1$ as $y \to \infty$ uniformly for $|t_1 - t_2| \ge 2^{-m}|I|$. Therefore,
\begin{align*}
& \int_{I_j^{(m)} \times I_k^{(m)}} \frac{M_y(0)^2 E_y(\mathbf{u}, \mathbf{t})}{M_y(u_1)M_y(u_2)} \nu_{y,0}(dt_1; \alpha_y^{\mathbf{u}, \mathbf{t}})\nu_{y, 0}(dt_2; \alpha_y^{\mathbf{u}, \mathbf{t}}) 
\exp\left\{-L \nu_{y, 0}(I; \alpha_{y}^{\mathbf{u}, \mathbf{t}})\right\}\\
& \le \exp\left\{2\sup_{\vec{\mathbf{t}}' \in I^3}\left|\widetilde{\Ga}^{\mathbf{u}, \mathbf{t}'}_{y, 0}(t_0')\right|\right\}  \nu_{y, 0}\left(I_j^{(m)}; \alpha\right)\nu_{y, 0}\left(I_k^{(m)}; \alpha\right) \exp\left\{-L\exp\left\{- R_{y, j, k}^{(m)}\right\}\nu_{y, 0}\left(I \setminus (\widetilde{I}_j^{(m)} \cup \widetilde{I}_k^{(m)}); \alpha\right)\right\}\\ 
& \qquad \times
\begin{cases}
    C & \text{for $|j-k| \le 2$}\\
    \sup_{|t_1' - t_2'| \ge 2^{-m}} D_y(\mathbf{u}, \mathbf{t}') & \text{for $|j - k| \ge 3$}
\end{cases}\\
& =: \Ia_{y, j, k, +}^{(m)}
\xrightarrow[y \to \infty]{p}
\Ia_{\infty, j, k, +}^{(m)}
 :=
\nu_{\infty}\left(I_j^{(m)}\right)
\nu_{\infty}\left(I_k^{(m)}\right)
\exp\left\{-L\nu_{\infty}\left(I \setminus (\widetilde{I}_j^{(m)} \cup \widetilde{I}_k^{(m)})\right)\right\}
\times
\begin{cases}
    C & \text{for $|j-k| \le 2$,}\\
    1 & \text{for $|j - k| \ge 3$.}
\end{cases}
\end{align*}
Obviously $\sum_{j, k=1}^{2^m} \Ia_{\infty, j, k, +}^{(m)} \ge \nu_\infty(I)^2 e^{-L\nu_\infty(I)}$. On the other hand, by separately considering $j - k \in \{-2, -1, 0, 1, 2\}$ it is easy to see that
\[\sum_{|j-k| \le 2} \nu_{\infty}\left(I_j^{(m)}\right)
\nu_{\infty}\left(I_k^{(m)}\right)
\le 5 \sum_{j =1}^{2^m} \nu_\infty\left(I_j^{(m)}\right)^2
\]
by Cauchy--Schwarz. In particular,
\begin{align*}
\sum_{j, k=1}^{2^m} \Ia_{\infty, j, k, +}^{(m)}
&\le \exp\left\{-L\nu_\infty(I)+ L\sup_{j, k \le 2^m}\nu_\infty\left(\widetilde{I}_j^{(m)} \cup \widetilde{I}_k^{(m)}\right)\right\}\\
& \qquad  \times 
\left\{\sum_{|j-k| \ge 3}\nu_{\infty}\left(I_j^{(m)}\right)
\nu_{\infty}\left(I_k^{(m)}\right)
+ C\sum_{|j-k| \le 2} \nu_{\infty}\left(I_j^{(m)}\right)
\nu_{\infty}\left(I_k^{(m)}\right)
\right\}\\
& \le  \exp\left\{-L\nu_\infty(I)+ 6L\sup_{j\le 2^m}\nu_\infty\left(I_j^{(m)}\right)\right\} 
\left[\nu_\infty(I)^2 + 5C \sum_{j=1}^{2^m} \nu_\infty\left(I_j^{(m)}\right)^2\right]\\
& \xrightarrow[m \to \infty]{p} \nu_\infty(I)^2 e^{-L\nu_\infty(I)}
\end{align*}
where the last step follows from the observations that
\[\sup_{j\le 2^m}\nu_\infty\left(I_j^{(m)}\right) \xrightarrow[m \to \infty]{a.s.} 0
\qquad \text{and} \qquad 
 \sum_{j=1}^{2^m} \nu_\infty\left(I_j^{(m)}\right)^2
 \le \nu_\infty(I)\sup_{j\le 2^m}\nu_\infty\left(I_j^{(m)}\right)
 \xrightarrow[m \to \infty]{a.s.} 0  \]
thanks to the non-atomicity of $\nu_\infty$. Thus, we have established $\sum_{j, k=1}^{2^m} \Ia_{\infty, j, k, +}^{(m)} \xrightarrow[m \to \infty]{p} \nu_\infty(I)^2 e^{-L\nu_\infty(I)}$.

A matching lower bound can be obtained similarly, and we conclude our proof by \Cref{lem:plim-squeeze}.
\end{proof}

\subsection{Proof of \texorpdfstring{\Cref{cor:mc-critical}}{Corollary \ref{cor:mc-critical}}}
Fix $u\ge 0$. Let $\alpha_y(n)=\alpha(n) \mathbf{1}_{p\mid n \implies p \le y}$. For $\delta = u / (2\log y)$,
\begin{align*}
\int_{0}^{1} \left|A_y(1/2+\delta+it)\right|^2 dt
&\ll_u \int_{\RR} \bigg|\frac{A_y(1/2+\delta+it)}{1/2+it+\delta}\bigg|^2 dt \\
& =2\pi\int_{0}^{\infty} |\sum_{n\le t}\alpha_y(n)|^2 \frac{dt}{t^{2+2\delta}} \le 2\pi \int_{0}^{\infty} |\sum_{n\le t}\alpha_y(n)|^2 \frac{dt}{t^2}=\int_{\RR} \bigg|\frac{A_y(1/2+it)}{1/2+it}\bigg|^2 dt 
\end{align*}
where the equalities are due to Plancherel's theorem \cite[Theorem~5.4]{MV}. This means for any $n \in \ZZ$,
\[    m_{y, u}([n, n+1]) 
    \overset{d}{=} m_{y, u}([0,1])
    \ll_u \frac{\sqrt{\log \log y}}{\log y} \int_{0}^{1} \left|A_y\left(1/2+ \delta+ it\right)\right|^2 dt \]
has uniformly bounded $q$-th moments for any fixed $q \in [0,1)$ thanks to \Cref{thm:ayintegral} (we used  $\EE[|A_y(1/2+\delta+it)|^2] = \prod_{p \le y}(1-1/p^{1+2\delta})^{-1} \asymp_{u} \log y$ which follows from Mertens' estimate \eqref{eq:mertens4}, see e.g.~\Cref{lem:tshift}). In particular, using the subadditivity of $x \mapsto x^q$, we have for any $n \in \NN$ that
\begin{equation}\label{eq:extend-tail}
    \EE\bigg[\bigg(\int_{|t| \ge n} h(t) m_{y, u}(dt)\bigg)^q\bigg]
    \ll \sum_{k \ge n} k^{-aq} \EE\bigg[\bigg(\int_{k}^{k+1} m_{y, u}(dt)\bigg)^q\bigg] \ll n^{1-aq}
\end{equation}
by choosing $q \in (0, 1)$ sufficiently large so that $aq > 1$, and this vanishes as $n \to \infty$ uniformly in $y \ge 3$.

Now, for any $\OurEpsilon > 0$,
\begin{align*}
&\PP\left(|m_{y, u}(h) - m_\infty(h)| > \OurEpsilon \right)\\
& \qquad \le \PP\left(|m_{y, u}(h\mathbf{1}_{[-n, n]}) - m_\infty(h\mathbf{1}_{[-n, n]})| > \frac{\OurEpsilon}{2} \right)
+ \PP\left(|m_{y, u}(h\mathbf{1}_{[-n, n]^c}) - m_\infty(h\mathbf{1}_{[-n, n]^c})| > \frac{\OurEpsilon}{2} \right)\\
& \qquad \le  \PP\left(|m_{y, u}(h\mathbf{1}_{[-n, n]}) - m_\infty(h\mathbf{1}_{[-n, n]})| > \frac{\OurEpsilon}{2} \right)
+ (2/\OurEpsilon)^q\EE\left[|m_{y, u}(h\mathbf{1}_{[-n, n]^c}) - m_\infty(h\mathbf{1}_{[-n, n]^c})|^q\right]
\end{align*}
by Markov's inequality. The first term on the right-hand side vanishes as $y \to \infty$ by \Cref{thm:mc-critical}. Meanwhile, the second term is bounded by
\[     (2/\OurEpsilon)^q\EE\left[m_{y, u}(h\mathbf{1}_{[-n, n]^c})^q + m_\infty(h\mathbf{1}_{[-n, n]^c})^q\right]  \]
which vanishes uniformly in $y \ge 3$ as $n \to \infty$ by \eqref{eq:extend-tail} and Fatou's lemma. This concludes our proof.

\section{Preparation for proof of \texorpdfstring{\Cref{thm:summain}}{Theorem \ref{thm:summain}}}\label{sec:prepsum}
In practice we prove the following slight generalisation of \Cref{thm:summain}. Given $\weight \colon \RR_{\ge 0}\to \CC$ with compact support, let
\[S^{\weight}_x  :=\frac{(\log \log x)^{\frac{1}{4}}}{\sqrt{x}}\sum_{n\ge 1}\alpha(n) \weight\left(\frac{n}{x}\right).\]
For $\weight = \mathbf{1}_{[0,1]}$ this recovers the sum in \Cref{thm:summain}. For $\Re s >0$ we define
\[K_{\weight}(s) := \int_{0}^{\infty} \weight(x)x^{s-1}dx.\]
For $\weight=\mathbf{1}_{[0,1]}$ we have $K_{\weight}(s)=1/s$. If $\weight$ is step function with compact support then $K_{\weight}(1/2+it)$ satisfies $K_{\weight}(1/2+it)\ll 1/|1/2+it|$ and $\sqrt{2\pi} \|  \weight\|_2=\|K_{\weight}(1/2+it)\|_2$ by Plancherel's theorem.
\begin{thm}\label{thm:summainw}
Let $\weight\colon \RR_{\ge 0}\to \CC$ be a step function with compact support such that $\|\weight\|_2>0$. Then
\[    S^{\weight}_x \xrightarrow[x \to \infty]{d} \sqrt{V^{\weight}_\infty} \ G \]
where 
\begin{equation}\label{eq:Vinfweight}
V^{\weight}_\infty:= \frac{1}{2\pi}\int_{\RR} \big|K_{\weight}\big(\tfrac{1}{2}+it\big)\big|^2 m_\infty(dt)
\end{equation}
is almost surely finite and strictly positive, and is independent of $G \sim \Na_\CC(0, 1)$. 
Moreover, the convergence in distribution is stable in the sense of \Cref{def:stable}.
\end{thm}
\begin{definition}\label{def:stable}
    Let $(\Omega, \Fa, \PP)$ be a probability space. We say a sequence of random variables $Z_n$ converges stably as $n \to \infty$ if $(Y, Z_n)$ converges in distribution as $n \to \infty$ for any bounded $\Fa$-measurable random variable $Y$.
\end{definition}
Similarly to the discussion in \cite[{\S}B.1]{GWL2}, \Cref{thm:summainw} can be rephrased as
\begin{equation}\label{eq:stable3w}
    \lim_{x \to \infty} \EE\big[Y \widetilde{h}(S^{\weight}_x)\big]
    = \EE\big[Y\widetilde{h}(\sqrt{V^{\weight}_\infty} G) \big]
\end{equation}
for any bounded continuous function $\widetilde{h}\colon \CC \to \RR$, and any bounded $\Fa_\infty$-measurable random variable $Y$ where $\Fa_\infty  := \sigma(\alpha(p), p=2,3,5,\ldots)$. In the end of the introduction we will show how the proof of \Cref{thm:summainw} follows from \Cref{thm:mc-critical} and the results explained in the next subsection. These technical results will in turn be established in the next sections.
\subsection{Technical propositions}
Let $\weight\colon \RR_{\ge 0}\to \CC$ be a function supported on $[0,A]$ for some $A> 1$. Let $P(1) = 1$, and for $n >1$  denote by $P(n)$ the largest prime factor of $n$. We introduce parameters $\OurEpsilon,\delta>0$. For every $k\ge 0$ we define
\[ x_k=x_{k,\OurEpsilon,\delta} := x^{\OurEpsilon+k\delta}, \qquad I_k=I_k(x,\OurEpsilon,\delta) := (x_{k}, x_{k+1}].\]
We let $K$ be the smallest non-negative integer such that $x_{K+1} \ge Ax$. Strictly speaking $K$ depends on $\OurEpsilon$, $\delta$, $A$ and $x$, but once $x$ is sufficiently large (in terms of $\OurEpsilon$, $\delta$ and $A$) $K$ stabilises as $K=\lfloor (1-\OurEpsilon)/\delta \rfloor$. We consider the following truncation of $S^{\weight}_{x}$:
\begin{equation}\label{eq:trunc}
S^{\weight}_{x,\OurEpsilon,\delta} :=\frac{(\log \log x)^{\frac{1}{4}}}{\sqrt{x}}\sum_{k=0}^{K}\sum_{\substack{ P(n)\in I_k\\ P(n/P(n)) \le x_{k}}} \alpha(n) \weight\left(\frac{n}{x}\right).
\end{equation}
In words, we discard from $S^{\weight}_{x}$ integers $n$ satisfying one of the following two conditions: (i) $P(n)\notin \cup_{k=0}^{K} I_k=(x_0,x_{K+1}]$, i.e.~the largest prime factor of $n$ is at most $x_0=x^{\OurEpsilon}$, (ii) $P(n/P(n))$ (the second largest prime factor of $n$, counting with multiplicity) and $P(n)$ (the largest prime factor) belong to the same $I_k$ for some $k$. In particular, if $P(n)^2 \mid n$ (i.e.~$P(n/P(n))=P(n)$), $n$ is discarded.

The same truncation was used in our earlier work \cite{GWL1}, which in turn was directly motivated by a truncation introduced recently in the context of holomorphic multiplicative chaos \cite{NPSV} by Najnudel, Paquette, Simm and Vu \cite{NPSV}. One of the novelties in our paper is showing that this truncation is justified, in the sense of the following proposition.
\begin{proposition}\label{prop:negdelta}
We have $\limsup_{\OurEpsilon\to 0^+} \limsup_{\delta\to 0^+}\limsup_{x \to \infty}\EE\big[|S^{\weight}_x - S^{\weight}_{x,\OurEpsilon,\delta}|\big]=0$ for every step function $\weight\colon \RR_{\ge 0}\to \CC$ with compact support.
\end{proposition}
Since any $n$ in the inner sum in the right-hand side of \eqref{eq:trunc}  satisfies $P(n)> x_{0}=x^{\OurEpsilon}$, we may write
\[S^{\weight}_{x,\OurEpsilon,\delta} = \sum_{p>x^{\OurEpsilon}} Z'_{x,p}\]
where, if $p \in I_k$ for some (unique) $k \ge 0$, we define
\[ Z'_{x,p} := \frac{(\log \log x)^{\frac{1}{4}}}{\sqrt{x}}\sum_{\substack{\substack{P(n)=p \\ P(n/P(n)) \le x_{k}}}} \alpha(n) \weight\left(\frac{n}{x}\right).\]
Note that $Z'_{x,p} \equiv 0$ if $p> Ax$.
The following lemma is related to the \textit{Lindeberg condition} appearing in the martingale central limit theorem.
\begin{lem}\label{lem:lind}
We have $\sum_{p>x^{\OurEpsilon}} \EE [|Z'_{x,p}|^4] \ll(\log x)^5/x^{\OurEpsilon}$ for every bounded $\weight\colon \RR_{\ge 0}\to\CC$ with compact support.
\end{lem} 
We introduce the $\sigma$-algebra $\Fa_{y^-} := \sigma(\alpha(p), p <y)$ generated by $\{\alpha(p): p < y\}$. 
\begin{proposition}\label{prop:bracket}
Let $\weight\colon \RR_{\ge 0}\to \CC$ be a step function with compact support. Then
\[T_{x,\OurEpsilon,\delta} :=\sum_{p>x^{\OurEpsilon}} \EE\left[ |Z'_{x,p}|^2 \mid \Fa_{p^-}\right]   \xrightarrow[x \to \infty]{p} C_{\OurEpsilon,\delta} V^{\weight}_\infty\]
where $V^{\weight}_\infty$ is defined in \eqref{eq:Vinfweight} and $\lim_{\OurEpsilon\to0^+}\lim_{\delta\to0^+} C_{\OurEpsilon,\delta}=1$.
\end{proposition}
\subsection{Lindeberg condition: proof of \texorpdfstring{\Cref{lem:lind}}{Lemma \ref{lem:lind}}}
This is similar to \cite[Lemma~3.5]{GWL1} but we give a standalone proof. Let $\weight\colon \RR_{\ge 0}\to\CC$ be bounded and supported on $[0,A]$. By \eqref{eq:orth},
\[\sum_{p>x^{\OurEpsilon}} \EE\left[|Z'_{x,p}|^4\right]=\frac{\log \log x}{x^2}\sum_{k=0}^{K}\sum_{p\in I_k}G(x,p,k)\]
for 
\[ G(x,p,k):=\sum_{\substack{ab=cd\\ P(a)=P(b)=P(c)=P(d)=p\\   P(a/p),P(b/p),P(c/p),P(d/p)\le x_k}} \weight(a/x)\weight(b/x)\overline{\weight(c/x)\weight(d/x)}.\]
One can write $G(x,p,k)$ as
\[G(x,p,k)=\sum_{\substack{P(m)=p,\, p^2 \mid m\\ P(m/p^2)\le x_k}}h(m), \qquad h(m):=\bigg|\sum_{\substack{ab=m\\ P(a)=P(b)}}\weight(a/x)\weight(b/x)\bigg|^2,\]
where $m$ stands for the common value of $ab$ and $cd$. Let $\tau(n) =\sum_{ab=n}1$ be the divisor function. We have the pointwise bound
$h \le \tau^2 \cdot (\sup |\weight|)^4$ so that
\[ G(x,p,k) \ll \sum_{\substack{p^2  \mid m \le A^2 x^2\\P(m)=p\\ P(m/p^2)\le x_k}}\tau^2(m).\]
It follows that 
\begin{equation}\label{eq:magicupper2}
\sum_{p>x^{\OurEpsilon}} \EE\left[|Z'_{x,p}|^4\right]\ll \frac{\log \log x}{x^2} \sum_{k=0}^{K} \sum_{p\in I_k} \sum_{\substack{p^2  \mid m \le A^2 x^2 \\ P(m)=p\\P(m/p^2)\le x_k}} \tau^2(m)\le \frac{\log \log x}{x^2}\sum_{p>x^{\OurEpsilon}} \sum_{\substack{p^2 \mid \mid m \le A^2 x^2 \\ P(m)=p}} \tau^2(m),
\end{equation}
where we replaced the condition on $P(m/p^2)$ with the less strict condition $p^2 \mid \mid m$, that indicates that $p^2$ divides $m$ but $p^3$ does not. We have 
\begin{equation}\label{eq:magicupper}
\sum_{p>x^{\OurEpsilon}}\sum_{\substack{p^2 \mid \mid m \le A^2x^2\\ P(m)= p}} \tau^2(m) = \sum_{p>x^{\OurEpsilon}}\tau^2(p^2) \sum_{\substack{m' \le A^2 x^2/p^2\\ P(m')<p}}\tau^2(m')
\end{equation}
where $m'$ stands for $m/p^2$, which is coprime to $p$. Next,
\[\sum_{\substack{m' \le A^2x^2/p^2\\ P(m')<p}}\tau^2(m') \le \sum_{m' \le A^2x^2/p^2}\tau^2(m') \ll (x/p)^2 (\log x)^4\]
since 
\[\sum_{n\le x}\tau^2(n) \le x\sum_{n\le x} \frac{\tau^2(n)}{n} \le x \prod_{p \le x} \left(\sum_{i=0}^{\infty}\frac{\tau^2(p^i)}{p^i}\right) \le x\prod_{p \le x}(1-1/p)^{-4} \ll   x \log^4 x\]
by Mertens' estimate, see \eqref{eq:mertens3}. This implies that
 \begin{equation}\label{eq:Hb}
 \sum_{p>x^{\OurEpsilon}}\tau^2(p^2) \sum_{\substack{m' \le A^2 x^2/p^2\\ P(m')<p}}\tau^2(m')\ll  x^2 (\log x)^{4}\sum_{p>x^{\OurEpsilon}}\frac{1}{p^2}\ll x^{2-\OurEpsilon}(\log x)^4.
 \end{equation}
We conclude by plugging \eqref{eq:Hb} in \eqref{eq:magicupper} and then \eqref{eq:magicupper} in \eqref{eq:magicupper2}.

\subsection{Proof of \texorpdfstring{\Cref{thm:summainw}}{Theorem \ref{thm:summainw}} under \texorpdfstring{\Cref{prop:negdelta}}{Proposition \ref{prop:negdelta}}, \texorpdfstring{\Cref{lem:lind}}{Lemma \ref{lem:lind}} and \texorpdfstring{\Cref{prop:bracket}}{Lemma \ref{prop:bracket}}}
Fix $\OurEpsilon,\delta > 0$. Let $\weight$ be a step function with compact support. We shall use the martingale central limit theorem in the following form.
\begin{lem}[{Theorem 3.2 and Corollary 3.1 of \cite{HH1980} (cf.~Lemma 2.1 of \cite{GWL2})}]\label{lem:mCLT}
For each $n$, let $(M_{n,j})_{j\le k_n}$ be a complex-valued, mean-zero and square integrable martingale with respect to the filtration $(\Fa_{j})_j$, and $\Delta_{n, j} := M_{n, j} - M_{n, j-1}$ be the corresponding martingale differences. Suppose the following conditions are satisfied.
\begin{itemize}
    \item[(a)] The conditional covariances converge: $\sum_{j =1}^{k_n} \EE\left[\Delta_{n, j}^2 | \Fa_{j-1}\right] \xrightarrow[n \to \infty]{p} 0$ and
    \[ \sum_{j =1}^{k_n} \EE\left[|\Delta_{n, j}|^2 | \Fa_{j-1}\right]  \xrightarrow[n \to \infty]{p} V_\infty.\]
    \item[(b)] The conditional Lindeberg condition holds: for any $\delta > 0$,
 $\sum_{j = 1}^{k_n} \EE\left[|\Delta_{n, j}|^2 \mathbf{1}_{\{|\Delta_{n, j}| > \delta\}} | \Fa_{j-1}\right] \xrightarrow[n \to \infty]{p} 0$.
\end{itemize}
Then $M_{n, k_n} \xrightarrow[n \to \infty]{d} \sqrt{V_\infty} \ G$ where $G \sim \Na_{\CC}(0,1)$ is independent of $V_\infty$, and the convergence in distribution is stable.
\end{lem}
We claim one may apply \Cref{lem:mCLT} to the martingale sequence associated to $S^{\weight}_{x,\OurEpsilon,\delta} =\sum_{p>x^{\OurEpsilon}}Z'_{x,p}$. Indeed, since $Z'_{x, p}$ is linear in $\alpha(p)$ by construction, for any $x \ge 3$ we automatically obtain \[\sum_{p>x^{\OurEpsilon}} \EE[(Z'_{x, p})^2 | \Fa_{p^-}] =  0\]
almost surely. By Cauchy--Schwarz and \Cref{lem:lind}, the conditional Lindeberg condition holds. Combining this with \Cref{prop:bracket} and applying \Cref{lem:mCLT}, we obtain that
$S^{\weight}_{x, \OurEpsilon,\delta} \xrightarrow[x \to \infty]{d}  \sqrt{C_{\OurEpsilon,\delta} V^{\weight}_\infty} G$ where the distributional convergence is also stable.

To establish the stable convergence of $S^{\weight}_{x}$ we use the formulation \eqref{eq:stable3w}: we would like to show that
\[    \lim_{x \to \infty} \EE\big[Y \widetilde{h}(S^{\weight}_x)\big]
    = \EE\big[Y\widetilde{h}(\sqrt{V^{\weight}_\infty} G) \big] \]
holds for any bounded $\Fa_\infty$-measurable random variable $Y$ and any bounded continuous function $\widetilde{h}\colon \CC \to \RR$. By a density argument, it suffices to establish this claim for bounded Lipschitz functions $\widetilde{h}$. Consider
\begin{align*}
    \left|\EE\left[Y \widetilde{h}(S^{\weight}_x)\right] - 
    \EE\left[Y\widetilde{h}(\sqrt{V^{\weight}_\infty} G) \right]\right|
    & \le \left|\EE\left[Y \widetilde{h}(S^{\weight}_{x, \OurEpsilon,\delta})\right] - 
    \EE\left[Y\widetilde{h}(\sqrt{C_{\OurEpsilon,\delta} V^{\weight}_\infty} G) \right]\right|\\
    & \qquad + \EE\left[Y \left|\widetilde{h}(S^{\weight}_x) - \widetilde{h}(S^{\weight}_{x, \OurEpsilon,\delta})\right|\right]
    + \EE\left[Y \left|\widetilde{h}(\sqrt{V^{\weight}_\infty} G) - \widetilde{h}(\sqrt{C_{\OurEpsilon,\delta} V^{\weight}_\infty} G)\right|\right].
\end{align*}
We see that the first term on the right-hand side converges to $0$ as $x \to \infty$ as a consequence of the stable convergence of $S^{\weight}_{x, \OurEpsilon,\delta}$, whereas the third term converges to $0$ as $\delta \to 0^+$ and then $\OurEpsilon \to 0^+$ as a consequence of dominated convergence. Meanwhile, the middle term is bounded by
$\|\widetilde{h}\|_{\mathrm{Lip}} \EE[|Y| \big|S^{\weight}_x - S^{\weight}_{x, \OurEpsilon,\delta}\big|\big] $ which goes to $0$ in the limit as $x \to \infty$, then $\delta \to 0^+$ and then $\OurEpsilon\to 0^+$ by \Cref{prop:negdelta}. Finally, $V_\infty$ is a.s.~finite because of the uniform estimate \eqref{eq:harpermoms}; it is a.s.~positive since $|K_{\weight}(1/2+it)|^2$ is strictly positive in some non-empty open interval $I$ and $m_\infty(I) > 0$ by its support property. This completes the proof.
\section{Truncation: proof of \texorpdfstring{\Cref{prop:negdelta}}{Proposition \ref{prop:negdelta}}}\label{sec:trunc}
By the triangle inequality, the general case of \Cref{prop:negdelta} is reduced to the case  $\weight = \mathbf{1}_{[0,A]}$
for some $A>0$. Throughout we suppose $\weight=\mathbf{1}_{[0,A]}$ and implied constants will depend on $A$. We introduce a simpler truncation of $S_x^{\weight}$:
\[S^{\weight}_{x,\OurEpsilon} :=\frac{(\log \log x)^{\frac{1}{4}}}{\sqrt{x}}
 \sum_{P(n)> x^{\OurEpsilon},\,P(n)^2 \nmid n} \alpha(n) \weight(n/x).\]
By the triangle inequality it suffices to establish the following two lemmas.
\begin{lem}\label{lem:smallp}
For $\weight = \mathbf{1}_{[0,A]}$, we have
\[\limsup_{\OurEpsilon\to 0^+} \limsup_{x \to \infty}\EE\left[|S^{\weight}_x - S^{\weight}_{x,\OurEpsilon}|\right]=0.\]
\end{lem}
\begin{lem}\label{lem:truncatedelta}
For $\weight = \mathbf{1}_{[0,A]}$, we have
\[\limsup_{\delta\to 0^+}\limsup_{x \to \infty}\EE\left[|S^{\weight}_{x,\OurEpsilon} - S^{\weight}_{x,\OurEpsilon,\delta}|\right]=0.\]
\end{lem}
The proofs will use the notation
\begin{equation}\label{eq:sxy}
s_{x,y} := \frac{1}{\sqrt{x}} \sum_{P(n)\le  y} \alpha(n)\weight(n/x).
\end{equation}
\subsection{Conditioning}
An important idea in the proofs of \Cref{lem:smallp} and \Cref{lem:truncatedelta} is the following conditioning lemma. Given $y\ge 2$ let \[\Fa_y:=\sigma(\alpha(p),p\le y)\]
be the $\sigma$-algebra generated by $\{\alpha(p): p \le y\}$. 
\begin{lem}\label{lem:conditioning}
Let $\weight=\mathbf{1}_{[0,A]}$. Uniformly for $2\le y <z\le Ax$ we have
\begin{equation}\label{eq:conditioning} 
\EE\left[ |s_{x,z}|^2 \mid \Fa_y\right] \ll \frac{\log z}{\log x\log y} \int_{0}^{\infty} |s_{t,y}|^2 \frac{dt}{t} + X
\end{equation}
where $X$ is a non-negative random variable with $\EE \left[X\right]\ll e^{-c \frac{\log x}{\log y}}$, and  $c>0$ is an absolute constant.
\end{lem}
\Cref{lem:conditioning} is a generalisation of an argument appearing in Harper's works \cite{Har2020,harper2023typical} (see \cite[Lemma~7.5]{NPS} for a similar argument in the HMC setting and \cite[Lemma~1.2]{GWBLMS} for an explicit reference), which corresponds to the $z=x$ case of \Cref{lem:conditioning}. By Plancherel's theorem (see \eqref{eq:plancherel} below), the integral in the right-hand side of \eqref{eq:conditioning} equals an integral involving $|A_y(1/2+it)|^2$. This relation leads us to another required ingredient in our proofs, namely the bound given in \Cref{thm:ayintegral} (only needed with $q=1/2$).
\subsection{Proof of \texorpdfstring{\Cref{lem:smallp}}{Lemma \ref{lem:smallp}}}
We break the proof into two parts.
\begin{lem}\label{lem:repeated}
Let $\weight=\mathbf{1}_{[0,A]}$. Let $X$ be a set of positive integers $n$ with $P(n)^2 \mid n$. Then we have the bound $\EE[| \sum_{n \in X} \alpha(n) \weight(n/x)|^2] \ll x/\log x$.
\end{lem}
\begin{lem}\label{lem:smallprimes}
Let $\weight=\mathbf{1}_{[0,A]}$. Uniformly for $2 \le y< z \le Ax$ we have
		\[ \EE \left[|s_{x,z}|\right] \ll \sqrt{\frac{\log z}{\log x}} \EE\left[\left( \frac{1}{\log y}\int_{\RR} \left|\frac{A_y(1/2+it)}{1/2+it}\right|^2 dt\right)^{1/2}\right]  + e^{-c\frac{\log x}{\log y}} \]
for some absolute constant $c>0$.
\end{lem}
\Cref{lem:smallprimes} is a generalisation of a bound of Harper \cite[Proposition 1]{Har2020}, which corresponds to $z=x$. We explain why these two lemmas together imply 
\[\limsup_{x \to \infty}\EE\left[|S^{\weight}_x - S^{\weight}_{x,\OurEpsilon}|\right] \ll \sqrt{\OurEpsilon}\]
and so \Cref{lem:smallp} follows. Observe that \[ S_x^{\weight}-S_{x,\OurEpsilon}^{\weight}=(\log \log x)^{\frac{1}{4}} s_{x,x^{\OurEpsilon}} + \frac{(\log \log x)^{\frac{1}{4}} }{\sqrt{x}}\sum_{n \in X} \alpha(n)\weight(n/x)\]
for $X:=\{ n\ge 1: P(n)^2 \mid n \text{ and } P(n)>x^{\OurEpsilon}\}$, so by the triangle inequality
\[ \EE |S_x^{\weight}-S_{x,\OurEpsilon}^{\weight}|\le (\log \log x)^{\frac{1}{4}} \EE|s_{x,x^{\OurEpsilon}}| + \frac{(\log \log x)^{\frac{1}{4}} }{\sqrt{x}}\EE|\sum_{n \in X} \alpha(n)\weight(n/x)|.\]
The second term in the right-hand side goes to $0$ in the limit as $x\to \infty$ by Cauchy--Schwarz and \Cref{lem:repeated}. For $x$ large enough (in terms of $\OurEpsilon$), the first expectation can be bounded by \Cref{lem:smallprimes} (with $z=x^{\OurEpsilon}$) as
\[(\log \log x)^{\frac{1}{4}}\EE|s_{x,x^{\OurEpsilon}}|  \ll \sqrt{\OurEpsilon} (\log \log x)^{\frac{1}{4}}\EE\bigg[\bigg( \frac{1}{\log y}\int_{\RR} \left|\frac{A_y(1/2+it)}{1/2+it}\right|^2 dt\bigg)^{1/2}\bigg]  + (\log x)^{-c} \]
for some $c>0$, by taking $y=x^{1/\log \log x}$. Since $\log \log y \asymp \log \log x$, the proof of \Cref{lem:smallp} is concluded by appealing to \eqref{eq:Harper} with $q=1/2$.
\subsection{Proof of \texorpdfstring{\Cref{lem:truncatedelta}}{Lemma \ref{lem:truncatedelta}}}
Let $\weight=\mathbf{1}_{[0,A]}$. Recall $I_k=(x_k,x_{k+1}]$ where $x_k=x^{\OurEpsilon+k\delta}$.  Let $T=T(x)$ be a function of $x$ that tends to $\infty$ slower than any power of $x$. Let
\begin{align*}
b_{k,1}&:=\frac{1}{\sqrt{x}} \sum_{p \in I_k} \sum_{q \in (x_{k},p)}\sum_{\substack{ P(n)=p\\ P(n/p)=q,\,q^2 \mid n}} \alpha(n)\weight(n/x),\\
b_{k,2}&:=\frac{1}{\sqrt{x}} \sum_{p \in I_k} \sum_{\substack{q \in (x_{k},p) \\ pq\le x/T}}\sum_{\substack{P(n)=p\\ P(n/p)=q,\, q^2 \nmid n}}\alpha(n)\weight(n/x).
\end{align*}
Observe that 
\begin{equation}\label{eq:sdiffvia} S_{x,\OurEpsilon}^{\weight}-S_{x,\OurEpsilon,\delta}^{\weight}=(\log \log x)^{\frac{1}{4}} (\sum_{k=0}^{K} (b_{k,1}+b_{k,2})+b_3)
\end{equation}
where
\[ b_3=\frac{1}{\sqrt{x}}\sum_{\substack{0\le k\le K,\\ q<p\in I_k,\, pq>x/T\\ pq \mid n}}\alpha(n)\weight(n/x)\]
if $x$ is sufficiently large (in terms of $\OurEpsilon$). Here we used the fact that if $n\le Ax$ is divisible by two primes $p>q>x^{\OurEpsilon}$ with $pq>x/T$ then $n/(pq)$ is at most $AT$ and in particular satisfies $P(n/(pq))<q$ (if $x$ is sufficiently large). Moreover, if $x$ is large enough then there can be at most one pair of primes satisfying these conditions for a given $n\le Ax$ (a pair of primes $p>q>x^{\OurEpsilon}$ with $pq>x/T$ such that $pq$ divides $n$).

We break the proof of \Cref{lem:truncatedelta} into three parts.
\begin{lem}\label{lem:repeateddelta}
Let $\weight=\mathbf{1}_{[0,A]}$. We have $\EE[|b_{k,1}|^2]\ll_{\OurEpsilon,\delta} 1/x_k$ for all $0 \le k \le K$.
\end{lem}
\begin{lem}\label{lem:closeprimes}
Let $\weight=\mathbf{1}_{[0,A]}$. If $x$ is sufficiently large in terms of $\OurEpsilon$, then
\[ \EE \bigg[\big|\sum_{k=0}^{K}b_{k,2}\big|\bigg] \ll \bigg(\sum_{k=0}^{K}\bigg(\sum_{p\in I_k}\frac{1}{p}\bigg)^2\bigg)^{1/2} \bigg(\EE\bigg[\bigg( \frac{1}{\log y}\int_{\RR} \left|\frac{A_y(1/2+it)}{1/2+it}\right|^2 dt\bigg)^{1/2}\bigg]  +e^{-c\frac{\log T}{\log  y}} \bigg)\]
holds uniformly for $y \in [2,T)$ for some absolute constant $c>0$.
\end{lem}
\begin{lem}\label{lem:closeprimes2}
Let $\weight=\mathbf{1}_{[0,A]}$. Let $Y$ be any set of positive integers $n$ of the form $pqm$ where $p,q$ are primes with $p>q\ge x^{1/4}$ and $pq>x/T$. Then
\[\EE\big[\big| \sum_{n \in Y} \alpha(n) \weight(n/x)\big|^2\big] \ll \frac{\log T}{\log x}.\]
\end{lem}
Let $\OurEpsilon\in (0,1/2)$. Suppose $x$ is sufficiently large in terms of $\OurEpsilon$. We choose $T$ to be $T=x^{1/\log \log x}$ and choose $y$ (from \Cref{lem:closeprimes}) to be $y=T^{1/\log \log T}$. From \Cref{lem:repeateddelta}, \Cref{lem:closeprimes2} and Cauchy--Schwarz,
\[ \sum_{k=0}^{K} \EE|b_{k,1}|+\EE|b_3|\ll_{\OurEpsilon,\delta}  x^{-\OurEpsilon/2} + \sqrt{\frac{\log T}{\log x}} ,\]
implying, by our choice of $T$, that 
\[ \limsup_{x \to \infty}(\log \log x)^{\frac{1}{4}}\EE\bigg[|\sum_{k=0}^{K} b_{k,1}+b_3|\bigg] = 0.\]
To conclude the proof, given \eqref{eq:sdiffvia} and the last limit it remains to show that
\[\limsup_{\delta\to 0^+}\limsup_{x \to \infty}(\log \log x)^{\frac{1}{4}}\EE\bigg[ \big|\sum_{k=0}^{K} b_{k,2}\big|\bigg]=0.\]
Applying \Cref{lem:closeprimes} with our choices of $T$ and $y$, and simplifying the bound using \eqref{eq:Harper} with $q=1/2$, it follows that
\[ (\log \log x)^{\frac{1}{4}}\EE\bigg[ \big|\sum_{k=0}^{K}b_{k,2}\big|\bigg] \ll  \bigg(\sum_{k=0}^{K}\bigg(\sum_{p\in I_k}\frac{1}{p}\bigg)^2\bigg)^{1/2}\]
for large enough $x$, so it suffices to show that
\begin{equation}\label{eq:sufficessumk}
 \limsup_{\delta\to 0^+}\limsup_{x \to \infty}  \bigg(\sum_{k=0}^{K}\bigg(\sum_{p\in I_k}\frac{1}{p}\bigg)^2\bigg)  = 0.
\end{equation}
We now invoke Mertens' theorem in the form 
\begin{equation}\label{eq:mer}
\sum_{p \le x} \frac{1}{p} =\log \log x + M + O(1/\log x)
\end{equation}
for $x\ge 2$, see \eqref{eq:mertens35}, to obtain that
\[ \limsup_{x\to \infty} \sum_{k=0}^{K}\bigg(\sum_{p\in I_k}\frac{1}{p}\bigg)^2 \le \sum_{k\ge 0}\log^2 \bigg(\frac{\OurEpsilon+(k+1)\delta}{\OurEpsilon+k\delta}\bigg)=\sum_{k\ge 0}\log^2\left(1+\frac{\delta}{\OurEpsilon+k\delta}\right). \]
Using $\log(1+t)\ll t$ for $t\ge 0$ and considering $k\delta \le \OurEpsilon$ and $k\delta > \OurEpsilon$ separately, we find that
\[ \sum_{k= 0}^{K}\bigg(\sum_{p\in I_k}\frac{1}{p}\bigg)^2 \ll \frac{\delta}{\OurEpsilon} + \frac{\delta^2}{\OurEpsilon^2} \]
and so \eqref{eq:sufficessumk} holds. This finishes the proof of \Cref{lem:truncatedelta}.
\subsection{Proofs of  \texorpdfstring{\Cref{lem:conditioning}}{Lemma \ref{lem:conditioning}},  \texorpdfstring{\Cref{lem:smallprimes}}{Lemma \ref{lem:smallprimes}} and \texorpdfstring{\Cref{lem:closeprimes}}{Lemma \ref{lem:closeprimes}}: lemmas involving conditioning}
\subsubsection{Proof of \texorpdfstring{\Cref{lem:conditioning}}{Lemma \ref{lem:conditioning}}}
We need the notions of $y$-smooth and $y$-rough integers. A positive integer is called $y$-smooth (resp.~$y$-rough) if $p\mid n \implies p \le y$ (resp.~$p\mid n\implies p>y$). We write $\Psi(x,y)$ for the number of $y$-smooth integers less than or equal to $x$. De Bruijn proved the bound \cite[Equation~(1.9)]{debruijn}
\begin{equation}\label{eq:smooth}
\Psi(x,y) \ll x e^{-c\frac{\log x}{\log y}}
\end{equation}
for some absolute $c>0$, uniformly for $x,y\ge 2$. We have the following standard bound:
\begin{equation}\label{eq:brunprim}
\sum_{n \le x,\, n\text{ is }y\text{-rough}} 1 \ll \frac{x}{\log y}
\end{equation}
holds for $x\ge y \ge 2$. It follows by taking $f(n)=\mathbf{1}_{n\text{ is }y\text{-rough}}$ in \cite[Corollary~2.15]{MV} and invoking Mertens' theorem \eqref{eq:mertens35}.

We now proceed with the body of the proof. Let $\weight=\mathbf{1}_{[0,A]}$. Any  $z$-smooth integer can be written uniquely as $m m'$ where $m$ is $y$-rough and simultaneously $z$-smooth, and  $m'$ is $y$-smooth. It allows us to express $s_{x,z}$ as
\[s_{x,z} = \sum_{\substack{1 \le m\le Ax\\ 
 m \text{ is }y\text{-rough}\\ \text{and }z\text{-smooth}}} \frac{\alpha(m)}{\sqrt{m}} s_{x/m,y}.\]
If $m_1,m_2$ are $y$-rough integers and $n_1,n_2$ are  $y$-smooth integers, then
\begin{equation}\label{eq:orthgeneral}
\EE\left[ \alpha(m_1 n_1) \overline{\alpha(m_2 n_2)}\mid \Fa_y\right] = \alpha(n_1)\overline{\alpha(n_2)}\mathbf{1}_{m_1=m_2}
\end{equation}
holds. It follows that
\begin{equation}\label{eq:orthapp}
\EE\left[ |s_{x,z}|^2 \mid \Fa_y\right] = \sum_{\substack{1\le m\le Ax\\ m \text{ is }y\text{-rough}\\ \text{and }z\text{-smooth}}} m^{-1} |s_{x/m,y}|^2.
\end{equation}
We split the terms in the $m$-sum in \eqref{eq:orthapp} according to whether $m\le x^{3/4}$ holds or not, and write the sum as $X+Y$ where
\[X := \sum_{\substack{m\le x^{3/4} \\ m \text{ is }y\text{-rough} \\ \text{and }z\text{-smooth}}}m^{-1}|s_{x/m,y}|^2,\qquad 
Y := \sum_{\substack{x^{3/4}<m\le Ax\\ m \text{ is }y\text{-rough}\\\text{and }z\text{-smooth}}}m^{-1} |s_{x/m,y}|^2.\]
The expectation of $X$ satisfies
\begin{align*} 
\EE \left[X\right] &= x^{-1}\sum_{\substack{m\le x^{3/4}\\ m \text{ is }y\text{-rough}\\ \text{and }z\text{-smooth}}} \sum_{m'\text{ is }y\text{-smooth}} |\weight(m'm/x)|^2\\
&=  x^{-1}\sum_{\substack{m\le x^{3/4}\\ m \text{ is }y\text{-rough}\\ \text{and }z\text{-smooth}}} \Psi(Ax/m,y)\ll \sum_{\substack{m\le x^{3/4}\\ m\text{ is }y\text{-rough}}} m^{-1} e^{-c\frac{\log(x/m)}{\log y}} 
\end{align*}
by \eqref{eq:orth} and \eqref{eq:smooth}. To treat the last sum we group elements $m$ according to the interval of the shape $[e^k,e^{k+1})$ to which they belong:
\[\sum_{\substack{m\le x^{3/4}\\ m\text{ is }y\text{-rough}}} m^{-1} e^{-c\frac{\log(x/m)}{\log y}}=   e^{-c\frac{\log x}{\log y}}\sum_{\substack{m\le x^{3/4}\\ m\text{ is }y\text{-rough}}} m^{-1} e^{c\frac{\log m}{\log y}} \\
 \ll  e^{-c\frac{\log x}{\log y}}\sum_{k\ge 0:\, e^k\le ex^{3/4}} e^{c\frac{k}{\log y}} \sum_{\substack{m \in [e^k,e^{k+1})\\m \text{ is }y\text{-rough}}} 1 /e^k.\]
 The inner $m$-sum in the right-hand side is $1$ if $k=0$ because $1$ is $y$-rough. It is $0$ if $k\in (0,\log (y/e)]$ because there are no $y$-rough integers in $(1,y]$. It is $\ll e^k/\log y$ otherwise, by \eqref{eq:brunprim}. We obtain
 \[ \EE \left[X\right] \ll e^{-c\frac{\log x}{\log y}} \bigg( 1+ \frac{1}{\log y}\sum_{k:\, y/e< e^k\le ex^{3/4}} e^{c\frac{k}{\log y}}\bigg) \ll e^{-c\frac{\log (x/x^{3/4})}{\log y}} \]
 by summing a geometric progression. It remains to bound $Y$. Since we assume $\weight=\mathbf{1}_{[0,A]}$ it follows that $t\mapsto s_{t,y}\sqrt{t}=\sum_{P(n)\le y,\, n\le At} \alpha(n)$ is a function of $\lfloor At\rfloor$ only, that is, $s_{t,y}\sqrt{t} = s_{\lfloor At\rfloor/A, y}\sqrt{\lfloor At\rfloor/A} \asymp s_{\lfloor At\rfloor/A,y} \sqrt{t}$. Setting $r:=\lfloor Ax/m\rfloor$ we may write
\begin{equation}\label{eq:t2upp}
Y\ll \sum_{1\le r <Ax^{1/4}}|s_{r/A,y}|^2 \sum_{\substack{x^{3/4}<m \le Ax \\ m \text{ is }y\text{-rough} \\ \text{ and }z\text{-smooth}\\ m \in (Ax/(r+1),Ax/r]}}m^{-1} \le \sum_{1\le r <Ax^{1/4}} |s_{r/A,y}|^2 (Ax/(r+1))^{-1} \sum_{\substack{x^{3/4}<m \le Ax \\ m \text{ is }y\text{-rough} \\ \text{ and }z\text{-smooth}\\ m \in (Ax/(r+1),Ax/r]}}1.
\end{equation}
To treat the $m$-sum in the right-hand side of \eqref{eq:t2upp} we need the following lemma, a very special case of the work of Shiu \cite[Theorem~1]{Shiu}; it will be discussed further in \Cref{sec:nt}.
\begin{lem}[Shiu]\label{lem:upperbnd}
Let $I=[x,x+h]$ be an interval with $2\le h\le x$. Suppose $2\le y <z$. If $h \ge x^{1/4}$ then
\[ \sum_{\substack{m \in I \\ m \text{ is }y\text{-rough} \\ \text{ and }z\text{-smooth}}}1 \ll \frac{\log z}{\log x\log y}|I|.\]
\end{lem}
We apply the lemma with $I=[Ax/(r+1),Ax/r]$. It follows that the inner sum in the right-hand side of \eqref{eq:t2upp} is $\ll\frac{\log z}{\log x\log y} \frac{x}{r^2}$, and so we can bound the right-hand side of \eqref{eq:t2upp} by 
\[ Y \ll \frac{\log z}{\log x\log y}\sum_{1\le r <Ax^{1/4}}\frac{|s_{r/A,y}|^2}{r} \ll\frac{\log z}{\log x\log y} \int_{0}^{Ax^{1/4}+1} |s_{t/A,y}|^2 \frac{dt}{t}\ll \frac{\log z}{\log x\log y} \int_{0}^{\infty} |s_{t,y}|^2 \frac{dt}{t}\]
as needed.

\subsubsection{Proof of \texorpdfstring{\Cref{lem:smallprimes}}{Lemma \ref{lem:smallprimes}}}
By \Cref{lem:conditioning}, Cauchy--Schwarz and the inequality $\sqrt{a+b}\le \sqrt{a}+\sqrt{b}$,
\[\EE \left[|s_{x,z}| \mid \Fa_{y}\right] \le (\EE \left[|s_{x,z}|^2 \mid \Fa_{y}\right])^{1/2} \ll \bigg( \frac{\log z}{\log x\log y} \int_{0}^{\infty} |s_{t,y}|^2 \frac{dt}{t} \bigg)^{1/2}  + X^{1/2}\]
where $X$ is a non-negative random variable with $\EE \left[X\right]\ll e^{-c \frac{\log x}{\log y}}$. By the law of total expectation and Cauchy--Schwarz,
\begin{align*}
\EE \left[|s_{x,z}| \right]
&\ll \EE\bigg[\bigg( \frac{\log z}{\log x\log y} \int_{0}^{\infty} |s_{t,y}|^2 \frac{dt}{t} \bigg)^{1/2}\bigg]  +(\EE \left[X\right])^{1/2}\\
& \ll \EE\bigg[\bigg( \frac{\log z}{\log x\log y} \int_{0}^{\infty} |s_{t,y}|^2 \frac{dt}{t} \bigg)^{1/2} \bigg] + \left[e^{-c\frac{\log x}{\log y}}\right]^{1/2}.
\end{align*}
By Plancherel’s theorem in the form given in \cite[Theorem~5.4]{MV},
\begin{equation}\label{eq:plancherel}
\int_{0}^{\infty} |s_{t,y}|^2 \frac{dt}{t}=\frac{1}{2\pi}\int_{\RR} \left|A_y(1/2+it)K_{\weight}(1/2+it)\right|^2 dt
\end{equation}
holds, which gives the result since $|K_{\weight}(1/2+it)| \asymp 1/|1/2+it|$ when $\weight=\mathbf{1}_{[0,A]}$.

\subsubsection{Proof of \texorpdfstring{\Cref{lem:closeprimes}}{Lemma \ref{lem:closeprimes}}}
We denote $m=m(x,p,q):=\min\{Ax/(pq),q-1\}$. By definition
\begin{equation}\label{eq:bydef}
    b_{k,2}=\sum_{\substack{p,q \in I_k\\ q<p,\, pq\le x/T}} \frac{\alpha(p)\alpha(q)}{\sqrt{pq}}  s_{\frac{x}{pq},q-1} =  \sum_{\substack{p,q \in I_k\\ q<p,\, pq\le x/T}}\frac{\alpha(p)\alpha(q)}{\sqrt{pq}}s_{\frac{x}{pq},m}.
\end{equation}
We now take $y\ge 2$ which satisfies $y<m$ for all $p,q$ in our sum. If $x$ is large enough in terms of $\OurEpsilon$ then any $y<T$ works (recall $T=T(x)$ is a fixed function of $x$ that grows slower than any power of $x$). From \eqref{eq:bydef} and \eqref{eq:orthgeneral},
\[	\EE\bigg[ \big|\sum_{k=0}^{K}b_{k,2}\big|^2 \mid \Fa_y\bigg] =	\sum_{k=0}^{K}\sum_{\substack{p,q \in I_k\\ q<p,\, pq\le x/T}}  \frac{1}{pq} \EE\left[ |s_{\frac{x}{pq},m}|^2 \mid \Fa_y\right].\]
By \Cref{lem:conditioning}
\[ \EE\left[ |s_{\frac{x}{pq},m}|^2 \mid \Fa_y\right] \ll\frac{\log (2+m)}{\log (2+x/(pq))\log y} \int_{0}^{\infty} |s_{t,y}|^2 \frac{dt}{t} + X_{p,q}\]
where $X_{p,q}$ is a non-negative random variable with $\EE \left[X_{p,q}\right]\ll e^{-c \frac{\log (x/(pq))}{\log y}}$. Summing the last inequality over $0\le k\le K$ and then over $p,q\in I_k$ with $q<p$ and $pq\le x/T$ , it follows that
\begin{equation}\label{eq:bydefimp}
    	\EE\bigg[ \big|\sum_{k=0}^{K} b_{k,2}\big|^2 \mid \Fa_y\bigg] \ll \frac{D}{\log y}\int_{0}^{\infty} |s_{t,y}|^2 \frac{dt}{t} + X
\end{equation}
where
\[ D\le \sum_{k=0}^{K}\sum_{\substack{p,q \in I_k\\ q<p,\, pq\le x/T}}\frac{1}{pq} \frac{\log (2+m)}{\log(2+x/(pq))}\ll \sum_{k=0}^{K}\bigg(\sum_{p\in I_k}\frac{1}{p}\bigg)^2 \]
and $X$ is a non-negative random variable with 
\[\EE \left[X\right]\ll\sum_{k=0}^{K} \sum_{\substack{p,q \in I_k\\ q<p,\, pq\le x/T}}\frac{1}{pq} e^{-c \frac{\log (x/(pq))}{\log y}} \ll \sum_{k=0}^{K}\bigg(\sum_{p\in I_k}\frac{1}{p}\bigg)^2 e^{-c \frac{\log T}{\log y}}.\]
By \eqref{eq:bydefimp}, Cauchy--Schwarz and the inequality $\sqrt{a+b}\le \sqrt{a}+\sqrt{b}$,
\[\EE \bigg[\big|\sum_{k=0}^{K}b_{k,2}\big| \mid \Fa_{y}\bigg] \le \bigg(\EE \bigg[\big|\sum_{k=0}^{K}b_{k,2}\big|^2 \mid \Fa_{y}\bigg]\bigg)^{1/2} \ll \bigg(\sum_{k=0}^{K}\bigg(\sum_{p\in I_k}\frac{1}{p}\bigg)^2\bigg)^{1/2}\bigg(\frac{1}{\log y}\int_{0}^{\infty} |s_{t,y}|^2 \frac{dt}{t} \bigg)^{1/2}  + X^{1/2}.\]
By the law of total expectation and Cauchy--Schwarz, the last equation implies
\[\EE \left[\big|\sum_{k=0}^{K} b_{k,2}\big| \right] \ll   \bigg(\sum_{k=0}^{K}\bigg(\sum_{p\in I_k}\frac{1}{p}\bigg)^2\bigg)^{1/2} \left( \EE\left[\bigg( \frac{1}{\log y} \int_{0}^{\infty} |s_{t,y}|^2 \frac{dt}{t} \bigg)^{1/2} \right] + e^{-c \frac{\log T}{2\log y}} \right).\]
To conclude we apply \eqref{eq:plancherel} to replace the integral $\int_{0}^{\infty} |s_{t,y}|^2dt/t$ with $\int_{\RR} \left|A_y(1/2+it)K_{\weight}(1/2+it)\right|^2 dt$; since $\weight=\mathbf{1}_{[0,A]}$ we have $|K_{\weight}(1/2+it)|\asymp 1/|1/2+it|$.

\subsection{Proofs of \texorpdfstring{\Cref{lem:repeated}}{Lemma \ref{lem:repeated}}, \texorpdfstring{\Cref{lem:repeateddelta}}{Lemma \ref{lem:repeateddelta}} and \texorpdfstring{\Cref{lem:closeprimes2}}{Lemma \ref{lem:closeprimes2}}: anatomy lemmas}
\subsubsection{Proof of \texorpdfstring{\Cref{lem:repeated}}{Lemma \ref{lem:repeated}}}
Let $\weight=\mathbf{1}_{[0,A]}$. By \eqref{eq:orth},
\[\EE\big[\big| \sum_{n \in X} \alpha(n) \weight(n/x)\big|^2\big] = \sum_{n\in X} |\weight(n/x)|^2= \sum_{n \in X\cap [1,Ax]} 1.\]
The last sum is at most
\[ \sum_{n \in X\cap [1,Ax]} 1\le \sum_{p\le \sqrt{Ax}} \sum_{\substack{n \le Ax:\\ p^2 \mid n,\, P(n)\le p}} 1  = \sum_{p \le \sqrt{Ax}} \Psi(Ax/p^2,p).\]
The contribution of $p\ge \log x$ to the last sum is $\ll x/\log x$ using the trivial bound $\Psi(Ax/p^2,p)\le Ax/p^2$. To bound the contribution of $p< \log x$ we use \eqref{eq:smooth} to obtain
\[ \sum_{n \in X\cap [1,Ax]} 1 \ll \frac{x}{\log x} +x\sum_{p<\log x} p^{-2} e^{-c\log x/\log p}\ll \frac{x}{\log x}\]
because $\sum_{p<\log x} p^{-2} e^{-c\log x/\log p}\ll \sum_{p} p^{-2} \cdot \max_{p<\log x} e^{-c\log x/\log p}  \ll e^{-c\log x /\log \log x}\ll 1/\log x$. 
\subsubsection{Proof of \texorpdfstring{\Cref{lem:repeateddelta}}{Lemma \ref{lem:repeateddelta}}}
Let $\weight=\mathbf{1}_{[0,A]}$. By \eqref{eq:orth},
\[ \EE [|b_{k,1}|^2]= x^{-1}\sum_{p \in I_k} \sum_{q \in (x_{k},p)}\sum_{\substack{ P(n)=p\\ P(n/p)=q,\,q^2 \mid n}}|\weight(n/x)|^2= x^{-1}\sum_{p \in I_k} \sum_{q \in (x_{k},p)}\sum_{\substack{ n\le Ax,\,P(n)=p\\ P(n/p)=q,\,q^2 \mid n}}1\]
by \eqref{eq:orth}. If we denote the multiplicity of $q$ in $n$ by $i$ then we can write
\[\EE [|b_{k,1}|^2] = x^{-1} \sum_{p \in I_k}\sum_{q\in (x_k,p)}  \sum_{i \ge 2}  \sum_{\substack{m \le Ax/(pq^i)\\ P(m)<q}}1\le x^{-1} \sum_{p \in I_k}\sum_{q\in (x_k,p)}  \sum_{i \ge 2}  \sum_{m \le Ax/(pq^i)}1\]
where $m$ stands for $n/(pq^i)$, and in the inequality we discarded the condition $P(m)<q$. The inner sum is $\le Ax/(pq^i)$, so that
\[\EE [|b_{k,1}|^2] \le A  \sum_{p \in I_k}\frac{1}{p}\sum_{\substack{ q \in (x_k,p) \\ i \ge 2: \, pq^i \le Ax}} \frac{1}{q^i}\le  A\sum_{p \in I_k}\frac{1}{p} \sum_{\substack{q > x_k\\  i \ge 2: \, q^i \le Ax/x_k}}\frac{1}{q^i}\ll  \sum_{p \in I_k}\frac{1}{p} \sum_{q > x_k}\frac{1}{q^2} \ll \frac{1}{x_k}  \sum_{p \in I_k}\frac{1}{p}.\]
To conclude, observe $\sum_{p \in I_k} 1/p \ll_{\OurEpsilon,\delta} 1$ by Mertens' estimate $\sum_{p \le x} 1/p= \log \log x + O(1)$, see \eqref{eq:mertens1}.

\subsubsection{Proof of \texorpdfstring{\Cref{lem:closeprimes2}}{Lemma \ref{lem:closeprimes2}}}
Let $\weight=\mathbf{1}_{[0,A]}$. By \eqref{eq:orth},
\begin{multline}\label{eq:twolargeprimes}
\EE\big[\big| \sum_{n \in Y} \alpha(n) \weight(n/x)\big|^2\big] = \sum_{n\in Y} |\weight(n/x)|^2= \sum_{n \in Y\cap [1,Ax]}1 \le \sum_{\substack{p>q\ge x^{1/4} \\ Ax\ge pq>x/T}} \sum_{n\le Ax:\, pq\mid n}1 \\
 \le Ax \sum_{\substack{p>q\ge x^{1/4}\\Ax\ge pq>x/T}}\frac{1}{pq}= Ax\sum_{x^{1/4}\le q<\sqrt{Ax}}\frac{1}{q}\sum_{Ax/q\ge p>\max\{q,x/(Tq)\}}\frac{1}{p}.
\end{multline}
If $q\ge \sqrt{x/T}$ then the range in the innermost sum in the right-hand side of \eqref{eq:twolargeprimes} is $p\in (q,Ax/q]$, so this $p$-sum is of size $O(1)$ by \eqref{eq:mer}. Thus the total contribution of these $q\ge \sqrt{x/T}$ is $\ll x \sum_{\sqrt{Ax} > q\ge \sqrt{x/T}}1/q \ll x\log T/\log x$, by \eqref{eq:mer} as well.

If $q<\sqrt{x/T}$ then the range in the $p$-sum is $p\in (x/(Tq),Ax/q]$, so this $p$-sum is $O(\log T/\log x)$ by \eqref{eq:mer} and the contribution of these $q$-s is at most $\ll x(\log T/\log x)\sum_{\sqrt{Ax} > q\ge x^{1/4}}1/q$. This is  $\ll x(\log T/\log x)$ by \eqref{eq:mer} which gives the desired conclusion.
\section{Bracket process: proof of \texorpdfstring{\Cref{prop:bracket}}{Proposition \ref{prop:bracket}}}\label{sec:bracket}
This section is quite similar to the proof of \cite[Proposition~3.6]{GWL1} in the $L^1$-regime, with the main change being the proof of \Cref{lem:close}. The following lemma is proved in \Cref{sec:nt} as an easy consequence of Mertens' theorem.
\begin{lem}\label{lem:tshift}
For any $t\ge 0$, we have
\[ \sum_{P(n)\le y} \frac{1}{n^{1+t/\log y}} \sim \exp\bigg( \gamma - \int_{0}^{t} \frac{1-e^{-s}}{s}ds\bigg) \log y, \qquad y\to \infty.\]
\end{lem}
We use \Cref{thm:mc-critical} and  \Cref{lem:tshift} to deduce the following result, proved below. We use the notation $s_{x,y}$ defined in \eqref{eq:sxy}.
\begin{lem}\label{lem:plancherelapp}
Let $\weight\colon \RR_{\ge 0}\to\CC$ be a step function with compact support. Then, for every $r>0$,
\[\sqrt{\log \log y}(\log y)^{-1} \int_{0}^{\infty}  \frac{|s_{t,y}|^2 dt}{t^{1+r/\log y}}\xrightarrow[y \to \infty]{p}    \exp\bigg( \gamma - \int_{0}^{r} \frac{1-e^{-s}}{s}ds\bigg) \, V^{\weight}_\infty.\]
\end{lem}
Recall the Dickman function $\rho\colon (0,\infty)\to(0,\infty)$. For $t\le 1$ it is given by $\rho(t)=1$ and for $t> 1$ by $t\rho(t)= \int_{t-1}^{t}\rho(v)dv$. It satisfies
\[ \int_{0}^{\infty} e^{-tv}\rho(v)dv= \exp\bigg(  \gamma -  \int_{0}^{t} \frac{1-e^{-s}}{s}ds\bigg)\]
for $t\ge 0$ \cite[Equation~(2.7)]{debruijn}, so \Cref{lem:plancherelapp} can be written as
\begin{equation}\label{eq:dickmanequiv}
\sqrt{\log \log y}(\log y)^{-1} \int_{0}^{\infty}  \frac{|s_{t,y}|^2 dt}{t^{1+r/\log y}}\xrightarrow[y \to \infty]{p}   \bigg(\int_{0}^{\infty} e^{-rv}\rho(v)dv \bigg)\, V^{\weight}_\infty.
\end{equation}
For any given $a>0$, we substitute $y=x^a$, $t=x^{1-c}$ and $r=ka$ in \eqref{eq:dickmanequiv}, obtaining
\begin{equation}\label{eq:poly}
\sqrt{\log \log x}\int_{\RR} e^{k(c-1)}|s_{x^{1-c},x^a}|^2 dc \xrightarrow[x \to \infty]{p}   \bigg(a\int_{0}^{\infty} e^{-kav}\rho(v)dv \bigg) V^{\weight}_\infty
\end{equation}
for $k=1,2,\ldots$. By a argument based on the Weierstrass approximation theorem, the following corollary (proved below) is obtained from \eqref{eq:poly}.
\begin{cor}\label{cor:sandwich}
Let $\weight\colon \RR_{\ge 0}\to\CC$ be a step function with compact support. Fix $a>0$ and $0<a_0<a_1$. If $Q\colon \RR_{>0}\to \RR_{\ge 0}$ is equal to a continuous function times the indicator $\mathbf{1}_{[a_0,a_1]}$, then
\[\sqrt{\log \log x} \int_{\RR} Q(e^{c-1})|s_{x^{1-c},x^a}|^2 dc \xrightarrow[x \to \infty]{p}   \bigg(a\int_{0}^{\infty} Q(e^{-av})\rho(v)dv \bigg) V^{\weight}_\infty. \]
\end{cor}
We now connect \Cref{cor:sandwich} with the bracket process. By \eqref{eq:orthgeneral} we may write $T_{x,\OurEpsilon,\delta}$ as
\[T_{x,\OurEpsilon,\delta} = \frac{\sqrt{\log \log x}}{x}\sum_{k=0}^{K}\sum_{p \in I_k} \bigg| \sum_{\substack{P(m)\le x_{k}}}\alpha(m)\weight\left(\frac{m}{x/p}\right)\bigg|^2=\sum_{k=0}^{K} T_{x,\OurEpsilon,\delta,k}\]
for 
\[T_{x,\OurEpsilon,\delta,k}:=\frac{\sqrt{\log \log x}}{x} \sum_{p \in I_k}  \bigg| \sum_{P(m)\le x_{k}}\alpha(m)\weight\left(\frac{m}{x/p}\right)\bigg|^2.\]
We introduce
\[R_{a,b}:=\sum_{p \in (x^a,x^b]}\frac{1}{p}|s_{x/p,x^a}|^2.\]
In this notation,
\begin{equation}\label{eq:TR}
T_{x,\OurEpsilon,\delta,k} = \sqrt{\log \log x}  R_{\OurEpsilon+k\delta,\OurEpsilon+(k+1)\delta}.
\end{equation}
The following lemma, proved below, shows that $R_{a,b}$ is close to $\int_{x^a}^{x^b} |s_{x/t,x^a}|^2dt /(t\log t)$. Its proof is surprisingly delicate, requiring information about primes in short intervals and the fact that $s_{x_1,y}$ is close to $s_{x_2,y}$ if $x_1$ is close to $x_2$.
\begin{lem}\label{lem:close}
Let $0<a<b$. The statement
\[ \sqrt{\log \log x}R_{a,b} \xrightarrow[x \to \infty]{p}C \, V^{\weight}_\infty \]
for some constant $C$ is equivalent to 
\[\sqrt{\log \log x}\int_{x^{a}}^{x^{b}}\frac{|s_{x/t,x^a}|^2}{t\log t} d t \xrightarrow[x \to \infty]{p}C\, V^{\weight}_\infty.\]
\end{lem}
Using \Cref{lem:close} and the change of variables $t=x^c$, we find from \eqref{eq:TR} that the statement
\[ T_{x,\OurEpsilon,\delta,k} \xrightarrow[x \to \infty]{p}C(k,\OurEpsilon,\delta) \, V^{\weight}_\infty \]
for some constant $C(k,\OurEpsilon,\delta)$ is equivalent to 
\begin{equation}\label{eq:equiv} 
\sqrt{\log \log x}\int_{\OurEpsilon+k\delta}^{\OurEpsilon+(k+1)\delta} c^{-1}|s_{x^{1-c},x^{\OurEpsilon+k\delta}}|^2 d c \xrightarrow[x \to \infty]{p}C(k,\OurEpsilon,\delta) \, V^{\weight}_\infty.
\end{equation}
By \Cref{cor:sandwich} with $Q(e^{c-1})=c^{-1}\mathbf{1}_{c\in [\OurEpsilon+k\delta,\OurEpsilon+(k+1)\delta]}$ and $a=\OurEpsilon+k\delta$, \eqref{eq:equiv} holds with
\[ C(k,\OurEpsilon,\delta) = (\OurEpsilon+k\delta)\int_{\max\{\frac{1-(\OurEpsilon+(k+1)\delta)}{\OurEpsilon+k\delta},0\}}^{\max\{\frac{1-(\OurEpsilon+k\delta)}{\OurEpsilon+k\delta},0\}} \rho(v) (1-(\OurEpsilon+k\delta)v)^{-1}dv.\]
It follows, since $T_{x,\OurEpsilon,\delta} = \sum_{k=0}^{K} T_{x,\OurEpsilon,\delta,k}$, that
\[T_{x,\OurEpsilon,\delta}  \xrightarrow[x \to \infty]{p} C_{\OurEpsilon,\delta} \, V^{\weight}_\infty \]
with
\begin{equation}\label{eq:CG}
C_{\OurEpsilon,\delta}:=\sum_{k=0}^{K} C(k,\OurEpsilon,\delta)= \int_{0}^{\frac{1}{\OurEpsilon}-1} \rho(v) G_{\OurEpsilon,\delta}(v) dv
\end{equation}
where
\[G_{\OurEpsilon,\delta}(v):=\sum_{\substack{0\le k \le K\\ \frac{1-\delta}{v+1}< \OurEpsilon+k\delta\le \frac{1}{v+1}  }}\frac{\OurEpsilon+k\delta}{1-(\OurEpsilon+k\delta)v}\]
for $v\in (0,1/\OurEpsilon-1)$. We write $\{t\}$ for the fractional part of $t$: $\{t\}=t-\lfloor t\rfloor$. Then
\begin{equation}\label{eq:Gsimp}
G_{\OurEpsilon,\delta}(v)=\frac{ \frac{1+O(\delta)}{v+1} }{1-\frac{1+O(\delta)}{v+1}v}\mathbf{1}_{\left\{(\frac{1}{v+1}-\OurEpsilon)/\delta \right\} < \frac{1}{v+1}}=(1+O_{\OurEpsilon}(\delta))\mathbf{1}_{\left\{(\frac{1}{v+1}-\OurEpsilon)/\delta \right\} < \frac{1}{v+1}}
\end{equation}
for $v\in (0,1/\OurEpsilon-1)$. We need one more lemma in order to estimate $C_{\OurEpsilon,\delta}$.
\begin{lem}\cite[Lemma~3.11]{GWL1}\label{lem:equid}
Let $A,a>0$. If $g \colon [0,A]\to \RR$ is continuous then
\[\lim_{\delta\to0^+}\int_{0}^{A} g(v) \left(\mathbf{1}_{\left\{(\frac{1}{v+1}-a)/\delta \right\} < \frac{1}{v+1}} - \frac{1}{v+1}\right)dv =0.\]
\end{lem}
Taking $\delta \to 0^+$, and using \Cref{lem:equid} with $g(v)=\rho(v)$, we obtain from \eqref{eq:CG} and \eqref{eq:Gsimp} that
\begin{equation}\label{eq:Ceps}
\lim_{\delta\to 0^+}C_{\OurEpsilon,\delta}= \int_{0}^{\frac{1}{\OurEpsilon} - 1}\rho(v)(1+v)^{-1}dv.
\end{equation}
Since $\rho(u)/(1+v) = -\rho'(v+1)$ by definition, the integral in the right-hand side of \eqref{eq:Ceps} is equal to $\rho(1)-\rho(1/\OurEpsilon)=1-\rho(1/\OurEpsilon)$ which tends rapidly to $1$ as $\OurEpsilon\to 0^+$.

\subsection{Proof of \texorpdfstring{\Cref{lem:plancherelapp}}{Lemma \ref{lem:plancherelapp}}}
Since $\EE |s_{t,y}|^2 \ll 1$, the expectation of $\int_{0}^{\infty}  |s_{t,y}|^2 t^{-1-r} dt$ is finite and so the integral converges almost surely. By Plancherel’s theorem \cite[Theorem 5.4]{MV},
\begin{equation}\label{eq:Planch} \int_{0}^{\infty} \frac{ |s_{t,y}|^2 dt}{t^{1+r}}=\frac{1}{2\pi}\int_{\RR} \big| A_y\big(\tfrac{1}{2}+\tfrac{r}{2}+it\big) K_{\weight}\big(\tfrac{1}{2}+\tfrac{r}{2}+it\big)\big|^2dt
\end{equation}
holds a.s.~for $r>0$. A direct computation using \eqref{eq:orth} shows that
\begin{equation}\label{eq:directcomp}
\EE [\big|A_y(\tfrac{1}{2}+\tfrac{r}{2}+it\big)|^2] = \sum_{P(n)\le y}\frac{1}{n^{1+r}}
\end{equation}
for every $t\in \RR$. By \eqref{eq:Planch} and \eqref{eq:directcomp},
\begin{equation}\label{eq:corapp}
\begin{split}
&\sqrt{\log \log y}\bigg(\sum_{P(n)\le y}\frac{1}{n^{1+r/\log y}}\bigg)^{-1} \int_{0}^{\infty}  \frac{|s_{t,y}|^2 dt}{t^{1+r/\log y}}\\
&\qquad = \frac{1}{2\pi}\int_{\RR} \big|K_{\weight}\big(\sigma_y(r)+it\big)\big|^2m_{y, r}(dt)\\
&\qquad = \frac{1}{2\pi}\int_{\RR} \big|K_{\weight}\big(\tfrac{1}{2}+it\big)\big|^2m_{y, r}(dt)+\frac{1}{2\pi}\int_{\RR} \big(\big|K_{\weight}\big(\sigma_y(r)+it\big)\big|^2 - \big|K_{\weight}\big(\tfrac{1}{2}+it\big)\big|^2\big)m_{y, r}(dt),
\end{split}
\end{equation}
where $\sigma_y(r)$ and $m_{y,r}$ are introduced in \Cref{thm:mc-critical}. Since $\weight$ is a step function with compact support, $K_{\weight}(s)=F(s)/s$ where $F(s)$ is a linear combination of exponential functions $a^s$ for some $a>0$. In particular, $|K_{\weight}(1/2+it)|^2 \ll 1/|t|^2$. By \Cref{cor:mc-critical}, \begin{equation}\label{eq:thmcor}
\frac{1}{2\pi}\int_{\RR} \big|K_{\weight}\big(\tfrac{1}{2}+it\big)\big|^2m_{y, r}(dt) \xrightarrow[y \to \infty]{p}   \frac{1}{2\pi}\int_{\RR} \big|K_{\weight}\big(\tfrac{1}{2}+it\big)\big|^2 m_\infty(dt).
\end{equation}
The right-hand side of \eqref{eq:thmcor} is $V^{\weight}_\infty$. The $n$-sum in the left-hand side of \eqref{eq:corapp} can be simplified using \Cref{lem:tshift}. It remains to show that 
\[\int_{\RR} \big(\big|K_{\weight}\big(\sigma_y(r)+it\big)\big|^2 - \big|K_{\weight}\big(\tfrac{1}{2}+it\big)\big|^2\big)m_{y, r}(dt) \xrightarrow[y \to \infty]{p}  0.\]
Taking expectations, it suffices to show that
\begin{equation}\label{eq:remain}
\lim_{y \to \infty}\sqrt{\log \log y}\int_{\RR} \left| \big|K_{\weight}\big(\sigma_y(r)+it\big)\big|^2 - \big|K_{\weight}\big(\tfrac{1}{2}+it\big)\big|^2\right|dt=0.
\end{equation}
By the mean value theorem,
\begin{align*} \left|\big|K_{\weight}\big(\sigma_y(r)+it\big)\big|^2 - \big|K_{\weight}\big(\tfrac{1}{2}+it\big)\big|^2\right| &\ll (1+|t|)^{-1}\big(\big|K_{\weight}\big(\sigma_y(r)+it\big)\big| - \big|K_{\weight}\big(\tfrac{1}{2}+it\big)\big|\big) \\
 &\ll (1+|t|)^{-2}|\sigma_y(r)-1/2| \ll (1+|t|)^{-2}\frac{r}{\log y} 
\end{align*}
which implies \eqref{eq:remain} and concludes the proof of the lemma.
\subsection{Proof of \texorpdfstring{\Cref{cor:sandwich}}{Corollary \ref{cor:sandwich}}}
This is identical to \cite[Corollary~3.9]{GWL1} apart from the definitions of  $\Ia_x$ and $\Ia_{\infty}$ which should be 
\[\Ia_x( g)
:=\sqrt{\log \log x}\int_{\RR} g(e^{c-1}) |s_{x^{1-c},x^a}|^2 d c,\qquad 
\Ia_\infty(g)
 :=\bigg( a \int_{0}^{\infty} g(e^{-av}) \rho(v) dv \bigg) V^{\weight}_\infty. \]
 \subsection{Proof of \texorpdfstring{\Cref{lem:close}}{Lemma \ref{lem:close}}}
 Let $\weight\colon \RR_{\ge0}\to \CC$ be a step function supported on $[0,A]$. We may assume $b \le 1$ since $s_{x/t,x^a}\equiv 0$ if $t> Ax$, and $s_{x/t,x^a}$ is bounded by a deterministic quantity for $t\in [x,Ax]$ (note $\int_{x}^{Ax}dt/(t\log t)$ and $\sum_{p\in [x,Ax]}1/p$ are $O(1/\log x)$ by Mertens \eqref{eq:mertens35}).  We also need the following Lipschitz-type property of $s_{x,y}$: when $x_2 \ge x_1$,
\begin{equation}\label{eq:lip}
 |s_{x_1,y}-s_{x_2,y}|\le |s_{x_2,y}|(\sqrt{x_2/x_1}-1) + X_{x_1,x_2}/\sqrt{x_1}
\end{equation}
holds where $\EE[|X_{x_1,x_2}|^2]\ll x_2-x_1+1$ (the implied constant depends only on $\weight$). The proof is immediate from the equality $\sqrt{x_1}s_{x_1,y}-\sqrt{x_2}s_{x_2,y} =\sum_{P(n)\le y} \alpha(n)(\weight(n/x_1)-\weight(n/x_2))$, the triangle inequality and \eqref{eq:orth}. 

We fix the function $h(t)=t/\log t$, and apply \eqref{eq:lip} with $x_1=x/p$ and $x_2=x/t$, obtaining that
\[|s_{x/p,x^a}|^2 = \frac{1}{h(p)} \int_{p-h(p)}^{p}(|s_{x/t,x^a}|^2 +Y_{x,p,t})dt \]
holds for the random variable $Y_{x,p,t}=|s_{x/p,x^a}|^2 - |s_{x/t,x^a}|^2$ which satisfies
\begin{equation}\label{eq:Ybnd}
\EE |Y_{x,p,t}| \ll \sqrt{p/t}-1+ \sqrt{(x/t - x/p+ 1)/(x/p)}\ll\sqrt{(p/t) - 1} + \sqrt{p/x}
\end{equation}
for $t \le p$. This bound follows from \eqref{eq:lip} by writing $Y_{x,p,t} = (|s_{x/p,x^a}|+|s_{x/t,x^a}|)(|s_{x/p,x^a}|-|s_{x/t,x^a}|)$, applying Cauchy--Schwarz and recalling $\EE [|s_{x,y}|^2]\ll 1$. Next, let
\[ M(t) := \sum_{p\in (x^a,x^b]: \, t \in [p-h(p),p]} \frac{1}{ph(p)}\]
so that we can write
\begin{equation}\label{eq:Rabeq}
R_{a,b} =\int_{x^{a}-h(x^a)}^{x^b} |s_{x/t,x^a}|^2 M(t) dt + Z_{a,b} \quad \text{where} \quad Z_{a,b}:=\sum_{p\in (x^a,x^b]}\frac{1}{ph(p)}\int_{p-h(p)}^{p}Y_{x,p,t}dt.
\end{equation}
Using \eqref{eq:Ybnd} (and Mertens \eqref{eq:mertens35}), we see that $\sqrt{\log \log x}\EE|Z_{a,b}|$ goes to $0$. By sieve theory (say, \eqref{eq:brun}), the bound $M(t) \ll 1/(t\log t)$ holds for $t\in [x^a-h(x^a),x^b]$. Moreover, the Prime Number Theorem with Error Term, \eqref{eq:pnt}, implies that $\sum_{y\le p\le y+h(y)} 1 = h(y)(1+O(1/\log y))/\log y$ for $y\ge 2$. This implies
\[M(t) = (1+O(1/\log t))/(t\log t)\]
for $t\in [x^{a},x^b-2h(x^b)]$. The proof now follows from \eqref{eq:Rabeq} by discarding
\[\sqrt{\log \log x}\int_{x^a-h(x^a)}^{x^a}|s_{x/t,x^a}|^2 M(t)dt\qquad \text{and} \qquad \sqrt{\log \log x}\int_{x^a}^{x^b} |s_{x/t,x^a}|^2 (M(t) - 1/(t \log t))dt.\]
To discard these, we use the bounds on $M(t)$ to verify that the expectation of their absolute values goes to $0$.

\subsection*{Acknowledgments}
O.G.~is supported by the Israel Science Foundation (grant no.~2088/24) and an Alon Fellowship. O.G.~is the incumbent of the Rabbi Dr.~Roger Herst Faculty Fellowship, which supported this work. M.D.W. is supported by the Royal Society
Research Grant RG\textbackslash R1\textbackslash 251187.

\appendix
\section{Number-theoretic estimates}\label{sec:nt}
Mertens proved the following estimate:
\begin{equation}\label{eq:mertens2} \sum_{p \le x} \frac{\log p}{p} = \log x+ O(1)
\end{equation}
for $x\ge 2$. By summation by parts it implies 
\begin{equation}\label{eq:mertens1} \sum_{p \le x} 1/p = \log \log x + O(1),
\end{equation}
which in turn implies
\begin{equation}\label{eq:mertens3} 
\prod_{p \le x} \left(1-\frac{1}{p}\right)^{-1} \ll \log x.
\end{equation}
For many applications these are enough, but Mertens actually proved the more precise estimate
\begin{equation}\label{eq:mertens35}
\sum_{p \le x} 1/p =\log \log x +M + O(1/\log x)
\end{equation}
for $x \ge 2$ for some absolute constant $M$ (the Meissel–-Mertens constant) as well as
\begin{equation}\label{eq:mertens4} 
\prod_{p \le x} \left(1-\frac{1}{p}\right)^{-1} \sim e^{\gamma} \log x
\end{equation}
as $x\to \infty$ where $\gamma$ is the Euler--Mascheroni constant. See \cite[Theorem~2.7]{MV} for the proofs of Mertens' estimates. Observe \eqref{eq:mertens35} includes Cheybshev's bound $\sum_{x\le p \le 2x} 1 \ll x/\log (2x)$ for $x \ge 1$. The Prime Number Theorem with Error Term \cite[Theorem~6.9]{MV} says, in particular,
\begin{equation}\label{eq:pnt}
\sum_{p \le x} 1 = \int_{2}^{x} \frac{dt}{\log t} + O(x/\log^3 x).
\end{equation}
\subsection{Discussion of  \texorpdfstring{\Cref{lem:upperbnd}}{Lemma \ref{lem:upperbnd}}}
Shiu \cite{Shiu} proved a general upper bound on the sum of a non-negative multiplicative function over a short interval. It is stated for a wide class of  functions. Applying Shiu's result to $f(n) = \mathbf{1}_{p\mid n \implies p \in (y,z]}$ one obtains \Cref{lem:upperbnd}.\footnote{In fact one may take $h\ge x^{\OurEpsilon}$ instead of $h\ge x^{1/4}$, at the cost of an implied constant depending on $\OurEpsilon$.} The function $f$ depends on parameters ($y$ and $z$) and we need a bound uniform in these parameters. Uniformity is explicitly verified in Koukoulopoulos' presentation of Shiu's theorem and proof \cite{Koukoulopoulos}. Below we prove a toy version of \Cref{lem:upperbnd} which suffices for the purpose of this paper. 
\begin{lem}\label{lem:suffice}
Let $I=[x,x+h]$ be an interval with $2\le h\le x$. Suppose $\log^2 x\le y <z$. If $h \ge yz$ then
\[ \sum_{\substack{m \in I \\ m \text{ is }y\text{-rough} \\ \text{ and }z\text{-smooth}}}1 \ll \frac{\log z}{\log x\log y}|I|.\]
If $\log z \gg \log x$, the conditions $\log^2 x\le y <z$ and $h\ge yz$ can be replaced with $2\le y <z$ and $h \gg y$.
\end{lem}
This suffices because in practice we take $h\ge x^{1/4}$ and $y$ is chosen to be smaller than any power of $x$ but greater than any power of $\log x$.
\begin{proof}
We recall the following basic sieve bound \cite[Theorem~3.6]{MV}:
\begin{equation}\label{eq:brun}
\sum_{\substack{n \in I\\ n\text{ is }y\text{-rough}}} 1 \ll \frac{|I|}{\log y}
\end{equation}
which holds for $y\ge 2$ and any finite interval $I\subseteq [1,\infty)$ of length $\gg y$. The case $\log z \gg \log x$ of \Cref{lem:suffice} is in \eqref{eq:brun}: we discard the condition that $m$ has to be $z$-smooth and then the claim follows from \eqref{eq:brun}. Next we treat the case $h\ge yz$ where we may assume $y<z\le x$. We have $\log x \le \log m$ for all $m\in I$ and so 
\[  \sum_{\substack{m \in I \\ m \text{ is }y\text{-rough} \\ \text{ and }z\text{-smooth}}}1  \le \frac{1}{\log x} \sum_{\substack{m \in I \\ m \text{ is }y\text{-rough} \\ \text{ and }z\text{-smooth}}}\log m.\]
We now apply \textit{Wirsing's trick}, namely we write $\log m $ as $\sum_{p,\, i\ge 1: \, p^i \mid m} \log p$ where the sum is over all prime powers that divide $m$. Then
\begin{equation}\label{eq:smoothrough} \sum_{\substack{m \in I \\ m \text{ is }y\text{-rough} \\ \text{ and }z\text{-smooth}}}\log m = \sum_{p \in (y,z],\, i\ge 1} \log p \sum_{\substack{m \in I,\, p^i \mid m \\ m \text{ is }y\text{-rough} \\ \text{ and }z\text{-smooth}}} 1\le \sum_{p \in (y,z],\, i\ge 1} \log p \sum_{\substack{m \in I \\ p^i \mid m\\ m \text{ is }y\text{-rough} }} 1,
\end{equation}
where in the last inequality we discarded the condition that $m$ is $z$-smooth. The inner sum in the right-hand side of \eqref{eq:smoothrough} counts $y$-rough integers in  the interval $I/p^i := \{ t/p^i: t\in I\}$. If $|I|/p^i\ge y$, it is $\ll |I|/(p^i \log y)$ by \eqref{eq:brun}, so these prime powers contribute
\[ \ll \sum_{p \in (y,z],\, i\ge 1} \log p \frac{|I|}{p^i \log y} \ll \frac{|I|}{\log y} \sum_{p \le z} \frac{\log p}{p}\]
to the right-hand side of \eqref{eq:smoothrough}, which is $\ll |I| \log z/\log y$ by \eqref{eq:mertens2}. The contribution of prime powers $p^i$ with $|I|/p^i<y$ and $i \ge 2$ is \[ \ll \sum_{\substack{p\in (y,z],\, i\ge 2:\\ |I|/p^i<y,\, p^i \le 2x}} \log p\left(\frac{|I|}{p^i}+1\right) \ll |I|\sum_{p>y} \frac{\log p}{p^2} + \sum_{p \le z} \log p \frac{\log 2x}{\log p} \ll \frac{\log y}{y}|I|+ z\log x \]
which is $\ll (\log z/\log y)|I|$ under the assumptions $|I|\ge yz$ and $y\ge \log^2x$. The contribution of prime powers $p^i$ with $|I|/p^i<y$ and $i=1$ is empty because $|I|/p^i=|I|/p\ge |I|/z$ and we assume $|I|\ge yz$.
\end{proof}
\subsection{Proof of \texorpdfstring{\Cref{lem:tshift}}{Lemma \ref{lem:tshift}}}
We have the Euler product
\[ \sum_{P(n)\le y} \frac{1}{n^{1+t/\log y}} = \prod_{p \le y}\left(1-p^{-1-t/\log y}\right)^{-1},\]
so for $t=0$ the claim is Mertens' estimate \eqref{eq:mertens4}. For $t>0$, we consider the ratio
\[ ( \sum_{P(n)\le y} \frac{1}{n^{1+t/\log y}}) /  (\sum_{P(n)\le y} \frac{1}{n})= \prod_{p \le y}\left(1-p^{-1-t/\log y}\right)^{-1}\left(1-p^{-1}\right)\]
and it suffices to show that it tends to  \[\exp\bigg(  - \int_{0}^{t} \frac{1-e^{-s}}{s}ds\bigg)\]
as $y\to \infty$. The logarithm of the above ratio is 
\[ \sum_{p \le y} \left(\log\left(1-p^{-1}\right)-\log \left(1-p^{-1-t/\log y}\right)\right) =\sum_{p \le y} p^{-1}( p^{-t/\log y}-1) +O\bigg(\frac{t}{\log y} \sum_{p \le y} \frac{\log p}{p^2}\bigg)\]
by the mean value theorem (applied to $s\mapsto\sum_{p\le y}(\log (1-p^{-1-s})+p^{-1-s})$). The error term goes to $0$ as $y\to \infty$. The main term 
\[ \sum_{p \le y} p^{-1} (p^{-t/\log y}-1)\]
is now studied through Mertens' theorem \eqref{eq:mertens35} and Riemann--Stieltjes integration. If we let
\[ M(t):=\sum_{p \le t} \frac{1}{p}=\log \log t + M+\Delta(t)\]
then $\Delta(t)=O(1/\log t)$ for $t\ge 2$ and
\[ \sum_{p\le y}p^{-1} (p^{-t/\log y}-1) = \int_{2}^{y^+} (r^{-t/\log y}-1) dM(r)=\int_{2}^{y} \frac{r^{-t/\log y}-1}{r\log r}dr + \int_{2^-}^{y^+} (r^{-t/\log y}-1) d \Delta(r) .  \]
The first integral in the right-hand side gives the desired expression after the change of variables $r=y^{s/t}$:
\[\int_{2}^{y} \frac{r^{-t/\log y}-1}{r\log r}dr = \int_{t\log 2 /\log y}^{t} \frac{e^{-s}-1}{s}ds=\int_{0}^{t}\frac{e^{-s}-1}{s}ds  + O(t/\log y).\]
The second integral can be seen to go to $0$ as $y\to \infty$ by using integration by parts: for $y\ge 3$,
\[\int_{2^-}^{y^+} (r^{-t/\log y}-1) d \Delta(r) = \Delta(y^+)(y^{-\frac{t}{\log y}}-1) - \Delta(2^-)(2^{-\frac{t}{\log y}}-1) +\frac{t}{\log y} \int_{2}^{y} r^{-\frac{t}{\log y}-1} \Delta(r) dr \ll \frac{t \log \log y}{\log y}.\]
\section{Concentration and generic chaining}
The following concentration inequality is useful when one has good control over moments of random variables.
\begin{thm}[Bernstein's inequality, {cf.~\cite[Corollary 2.11]{BLM2013}}]\label{thm:Berstein}
Let $X_1, \dots, X_n$ be independent centred real-valued random variables. Suppose there exist $v, c > 0$ such that and
\[  \sum_{k =1}^n \EE[|X_i|^r] \le \frac{r!}{2} v c^{r-2} \qquad \text{for all integers $r \ge 2$}. \]
Then for all $x \ge 0$, we have
\[  \PP\bigg(\bigg|\sum_{k \le n} X_k\bigg| \ge x \bigg) \le 2\exp\left(-\frac{x^2}{2(v+cx)}\right). \]
\end{thm}
Let $T$ be a separable metric space with respect to a metric $d$. We say $(T_n)_{n \ge 1}$ is an admissible sequence for $T$ if it is an increasing sequence of subsets of $T$ such that $|T_0| = 1$, $|T_n| \le 2^{2^n}$, $|T_n|^2 \le |T_{n+1}|$ for all $n \ge 1$, and that they are equipped with a sequence of maps $\pi_n \colon T \to T_n$ such that $\lim_{n \to \infty} \pi_n(t) = t$ for all $t \in T$.\footnote{If $T$ is a finite set, the sum on the right-hand side of \eqref{eq:metric-entropy} is finite, i.e.~it stops at the first $n_0$ such that $T_{n_0} = T$.} To measure the size of our admissible sequence we introduce the chaining functionals
\begin{equation}\label{eq:metric-entropy}
\gamma_k(T, d) := \sup_{t \in T} \sum_{n \ge 1} 2^{\frac{n}{k}} d(\pi_n(t), \pi_{n-1}(t)) \qquad (k=1,2).
\end{equation}
This is a slight abuse of notation since we suppress the dependence on $(T_n)_n$ in \eqref{eq:metric-entropy}, but this should not cause any confusion below as we will only work with one admissible sequence at a time.\footnote{The definition of admissible sequence and chaining functional here are slightly different from those used by Talagrand for generic chaining in \cite{Tal2022}. This is to keep our discussion self-contained and avoid too much prerequisite knowledge, but with little effort it can be rewritten in the language of \cite{Tal2022}.} We are now ready to state the generic chaining result.
\begin{thm}[Generic chaining with two distances]\label{thm:chaining}
Let $T$ be a separable metric space with respect to both metrics $d_1$ and $d_2$, and $(T_n)_n$ be an admissible sequence. Suppose $(X_t)_{t \in T}$ is a separable centred stochastic process satisfying
\begin{equation}\label{eq:chaining-assumption}
    \PP\left(|X_s - X_t| > x\right) \le C \exp\bigg\{- \min \left(\frac{x^2}{d_2(s, t)^2} ,\frac{x}{d_1(s, t)}\right)\bigg\} \qquad \forall x \ge 0.
\end{equation}
Then
\[\PP\left(\sup_{s, t \in T}|X_s - X_t| > x\right) \le \max(5C, e^2) \exp\left(-\frac{x^2}{16 \gamma_2(T, d_2)^2 + 8\gamma_1(T, d_1)x} \right) \qquad \forall x \ge 0.\]
\end{thm}
\begin{proof}
This is essentially established in the proof of \cite[Theorem 4.5.13]{Tal2022}. For completeness we sketch the proof and track the exponents more carefully here.

Let us begin by rewriting \eqref{eq:chaining-assumption} in terms of the weaker inequality
\[\PP\left(|X_s - X_t| > x^2 d_1(s,t) + xd_2(s, t)\right) \le C e^{-x^2}.\]
Suppose $|X_{\pi_{n}(t)} - X_{\pi_{n-1}}(t)| \le 2^n x^2 d_1(\pi_{n}(t), \pi_{n-1}(t)) + 2^{\frac{n}{2}} xd_2(\pi_{n}(t), \pi_{n-1}(t))$ for all $n \ge 1$ and $t \in T$. Then
\begin{align*}
\sup_{s, t \in T} |X_s - X_t|
\le 2 \sup_{t \in T} |X_t - X_{t_0}|
& \le 2 \sup_{t \in T} \sum_{n \ge 1} \left[2^n x^2 d_1(\pi_{n}(t), \pi_{n-1}(t)) + 2^{\frac{n}{2}} xd_2(\pi_{n}(t), \pi_{n-1}(t))\right]\\
& \le 2 \left[ x^2 \gamma_1(T, d_1) + x \gamma_2(T, d_2)\right]
\end{align*}
where $t_0$ is the unique element in $T_0$. In other words,
\begin{align*}
& \PP\left( \sup_{s, t \in T} |X_s - X_t| > 2 \left[ x^2 \gamma_1(T, d_1) + x \gamma_2(T, d_2)\right]\right)\\
&\qquad  \le \PP\left(\exists (n, t) \in \NN \times T: |X_{\pi_{n}(t)} - X_{\pi_{n-1}(t)}| > 2^n x^2 d_1(\pi_{n}(t), \pi_{n-1}(t)) + 2^{\frac{n}{2}} xd_2(\pi_{n}(t), \pi_{n-1}(t))\right) \\
&\qquad \le C \sum_{n \ge 1} |T_n| |T_{n-1}| e^{-2^n x^2}
\le C \sum_{n \ge 1} 2^{2^{n+1}} e^{-2^n x^2}.
\end{align*}
If $x \ge 2$, it is straightforward to verify that $2^n x^2 \ge \frac{x^2}{2} + 2^{n+1}$. Using $1-2/e \ge 1/5$, we obtain
\[ \PP\left( \sup_{s, t \in T} |X_s - X_t| > 2 \left[ x^2 \gamma_1(T, d_1) + x \gamma_2(T, d_2)\right]\right)
\le C \bigg[\sum_{n \ge 1} (2/e)^{2^{n+1}} \bigg]e^{-\frac{x^2}{2}} \le \max(5C, e^2) e^{-\frac{x^2}{2}} \]
where the last inequality also holds for $x \in [0, 2]$ because the left-hand side is always bounded by $1$.

Finally, consider the quadratic equation $Ay^2 + By = z$ where $A, B, y, z \ge 0$. Then 
\[y = -\frac{-B + \sqrt{B^2 + 4Az}}{2A} = \frac{4Az}{2A\left(B + \sqrt{B^2 + 4Az}\right)} \ge \frac{z}{B + \sqrt{Az}}.\]
Taking $y:= x$, $A:= 2\gamma_1(T, d_1)$ and $B:= 2\gamma_2(T, d_2)$, we conclude that
\[ \PP\left(\sup_{s, t\in T} |X_s - X_t| > z\right)
\le \max(5C, e^2) \exp\bigg\{-\frac{1}{2} \left(\frac{z}{B + \sqrt{Az}}\right)^2\bigg\}
\le \max(5C, e^2) \exp\left\{-\frac{z^2}{4(B^2 + Az)}\right\}. \]
\end{proof}
The following lemma allows us to control the chaining functionals $\gamma_k(T, d)$ defined in \eqref{eq:metric-entropy} when $T$ is a subset of the Euclidean space and $d$ is the Euclidean metric up to a rescaling factor.
\begin{lem}\label{lem:metric-entropy-estimate}
Let $I \subset \RR$ be a compact interval. Consider $T := I^3$ equipped with the metric $d(s, t) := K|s-t| $ for some $K > 0$. Then there exists an admissible sequence $(T_n)_n$ for $T$ such that
\[    \gamma_1(T, d) \ll K|I| \qquad \text{and} \qquad \gamma_2(T, d) \ll K|I|.\]
\end{lem}
\begin{proof}
Let $T_0 = T_1 \subset T$ be any singleton set, and define $T_n := I_n^3$ where $I_n \subset I$ consists of $2^{2^{n-2}}$ dyadic points for each $n \ge 2$. Let $\pi_n\colon T \to T_n$ be such that $\pi_n(t) = \argmin_{s \in T_n} d(s, t)$ (pick any representative if there is a tie). One can quickly verify that $|T_n| \le 2^{2^n}$ and $|T_n|^2 \le |T_{n+1}|$ so that the sequence $(T_n)_n$ is indeed admissible. By considering neighbouring dyadic points,  it is easy to see that $|\pi_n(t)- \pi_{n-1}(t)| \le \sqrt{3}|I| (2^{2^{n-2}} -1)^{-1}$ for all $n \ge 2$ and $t \in I^3$, and thus for $k \ge 1$ we have
\[\gamma_k(T, d)
= \sup_{t \in T} \sum_{n \ge 1}2^{\frac{n}{k}} d(\pi_n(t), \pi_{n-1}(t))
\le \sqrt{3}K|I|\sum_{n \ge 2} \frac{2^\frac{n}{k}}{2^{2^{n-2} }-1} 
\ll K|I|. \]
\end{proof}

\bibliographystyle{abbrv}
\bibliography{references}

\Addresses
\end{document}